\documentclass[11pt,reqno]{amsart}

\usepackage[utf8]{inputenc}
\usepackage[margin=3.1cm]{geometry}
\usepackage{xcolor}
\usepackage{stackengine,graphicx}
\usepackage{todonotes}
\usepackage{subcaption}
\usepackage{mathtools,xfrac}
\usepackage{hyperref}
\usepackage{enumerate}

\usepackage{tikz}
\usetikzlibrary{patterns,snakes}

\mathtoolsset{showonlyrefs}

\usepackage{amsmath, amssymb, amsthm, mathrsfs}
\usepackage{esint}
\usepackage{verbatim}
\usepackage{dsfont}
\usepackage{enumitem}
\usepackage{helvet}
\numberwithin{equation}{section}

\newcommand{\res}{\mathop{\hbox{\vrule height 7pt width .5pt depth 0pt
			\vrule height .5pt width 6pt depth 0pt}}\nolimits}

\DeclareMathOperator{\length}{length}
\DeclareMathOperator{\diam}{diam}
\DeclareMathOperator{\tr}{tr}
\DeclareMathOperator{\Tan}{Tan}
\DeclareMathOperator{\dist}{dist}
\DeclareMathOperator{\Lin}{Lin}
\DeclareMathOperator{\Span}{Span}

\DeclareMathOperator{\curl}{curl}
\DeclareMathOperator{\dive}{div}
\DeclareMathOperator{\DIVE}{\mathsf{div}}
\DeclareMathOperator{\grad}{\mathsf{grad}}
\DeclareMathOperator{\conv}{conv}
\DeclareMathOperator{\supp}{supp}

\DeclareMathOperator{\err}{err}

\newcommand{\Gto}{\overset{\Gamma}{\longrightarrow}}

\newcommand{\one}{{{\bf 1}}}
\renewcommand{\hat}{\widehat}
\renewcommand{\tilde}{\widetilde}
\newcommand{\Lip}{\mathrm{Lip}}
\newcommand{\spane}{\textrm{span}}
\newcommand{\tra}{{{\mathrm{tr}}}}
\newcommand{\toa}{\stackrel{\rma}{\to}}

\newcommand{\ip}[1]{\langle {#1}\rangle}
\newcommand{\bip}[1]{\big\langle {#1}\big\rangle}
\newcommand{\Bip}[1]{\Big\langle {#1}\Big\rangle}

\newcommand{\Lm}{\mathscr{L}} 
\newcommand{\Hm}{\mathcal{H}} 
\newcommand{\N}{\mathbb{N}} 
\newcommand{\Z}{\mathbb{Z}} 
\newcommand{\Q}{\mathbb{Q}} 
\newcommand{\E}{\mathbb{E}} 
\newcommand{\R}{\mathbb{R}} 
\newcommand{\1}{\mathds{1}} 
\newcommand{\calA}{\mathcal{A}}
\newcommand{\cX}{\mathcal X}
\newcommand{\cE}{\mathcal E}
\renewcommand{\P}{\mathbb{P}}

\usepackage{amsmath}

\DeclareMathOperator*{\argmin}{arg\,min}

\newcommand{\ext}{\nu_{\text{ext}}}

\newcommand{\cB}{\mathcal{B}}
\newcommand{\id}{\mathrm{id}}
\newcommand{\M}{\mathcal{M}}

\newcommand{\bF}{\mathbb F}
\newcommand{\bW}{\mathbb{W}}

\newcommand{\eps}{\varepsilon}

\newcommand{\tD}{\text{D}}

\newcommand{\EEE}{\color{black}}

\newcommand{\tand}{\quad \text{and} \quad}

\newcommand{\where}{\quad \text{where} \quad}

\newcommand{\suchthat}{\ensuremath{\ : \ }} 
\newcommand{\de}{\ensuremath{\, \mathrm d}} 
\newcommand{\dd}{\, \mathrm{d}}
\newcommand{\ddd}{\mathrm{d}}

\newcommand{\weX}{\cX_\eps}
\newcommand{\weE}{\cE_\eps}

\newcommand{\fP}{\mathsf P}
\newcommand{\Rep}{\mathsf{Rep}}
\newcommand{\mP}{\mathbb{P}}

\newcommand{\tv}{\text{TV}}
\newcommand{\KR}{\mathrm{KR}}
\newcommand{\TV}{\mathrm{TV}}
\newcommand{\tKR}{\widetilde{\mathrm{KR}}}
\newcommand{\tP}{\text{P}}

\newcommand{\cC}{\mathcal C}

\newtheorem{theorem}{Theorem}[section]
\newtheorem{prop}[theorem]{Proposition}

\newtheorem{lemma}[theorem]{Lemma}
\newtheorem{defi}[theorem]{Definition}
\newtheorem{cor}[theorem]{Corollary}

\theoremstyle{remark}
\newtheorem{rem}[theorem]{Remark}
\newtheorem{example}[theorem]{Example}

 \def\bbE{{\mathbb E}}

\def\calA{{\mathcal A}} \def\calB{{\mathcal B}}

  \def\calL{{\mathcal L}}
\def\calM{{\mathcal M}}  
  \def\calR{{\mathcal R}}
\def\calS{{\mathcal S}}

\def\rma{{\mathrm a}}  \def\rmc{{\mathrm c}}

\title[Stochastic homogenisation of minimum-cost flow problems]{Stochastic homogenisation of nonlinear minimum-cost flow problems}
\date{\today}

\begin{document}

	\author[P.~Gladbach]{Peter Gladbach}
	\author[J.~Maas]{Jan Maas}
	\author[L.~Portinale]{Lorenzo Portinale}

	\address{Institut f\"ur Angewandte Mathematik \\ Rheinische Friedrich-Wilhelms-Universit\"at Bonn \\ Endenicher Allee 60 \\ 53115 Bonn, Germany}
	\email{gladbach@iam.uni-bonn.de}

	\address{Institute of Science and Technology Austria (IST Austria),
	Am Campus 1, 3400 Klosterneuburg, Austria}
	\email{jan.maas@ist.ac.at}

    \address{Università degli studi di Milano statale, \\ 
    Via Saldini, 50
20133 Milano, Italy}
    \email{lorenzo.portinale@unimi.it}
	
	\subjclass[2010]{}

	\begin{abstract}
		This paper deals with the large-scale behaviour of nonlinear minimum-cost flow problems on random graphs. 
		In such problems, a random nonlinear cost functional is minimised among all flows (discrete vector-fields) with a prescribed net flux through each vertex.
		On a stationary random graph embedded in $\R^d$,
		our main result asserts that these problems converge, 
		in the large-scale limit, 
		to a continuous minimisation problem
		where an effective cost functional is minimised among all 
		vector fields with prescribed divergence. 
		Our main result is formulated using $\Gamma$-convergence 
		and applies to multi-species problems. 
		The proof employs the blow-up technique by Fonseca and M\"uller in a discrete setting.
		One of the main challenges to overcome is the construction of the homogenised energy density on random graphs without a periodic structure.

	\end{abstract}

	\maketitle

	\setcounter{tocdepth}{1}
	\tableofcontents

\section{Introduction}
\label{sec:intro}

This article studies the large-scale behaviour of minimum-cost flow problems on large random graphs. 

\subsection*{Nonlinear minimum-cost flow problems}
Let $(\cX,\cE)$ be a graph with 
vertex set $\cX$ and edge set 	
	$\cE\subseteq \cX\times \cX$. 
We always assume in this paper that graphs are undirected, i.e., $\cE$ is symmetric.
A \emph{scalar flow} on such a graph is an antisymmetric function $J: \cE \to \R$.
In this case, we write $J: \cE \toa \R$.
For each $e \in \cE$, let  
$f_e : \R \to [0,+\infty]$ be a cost function, so that 
	$f_{e}(j)$ 
represents the cost of flowing $j$ units of mass through the edge $e \in \cE$. 
Note that capacity constraints can be incorporated, since $f_e$ may attain the value $+\infty$.
Let $m \in \M_0(\cX)$ be a signed measure of total mass $0$, which prescribes the desired net flux through each of the vertices.

The \emph{minimum-cost flow problem} \cite{FordFulkerson} 
consists in minimising the total transport costs 
among all scalar flows with prescribed flux through each vertex, 
that is,
	\begin{align}
		\label{eq: min cost}
		& \text{minimise} \	
		&& 
			F(J) := 
			\sum_{(x,y)\in\cE} 
			f_{(x,y)}\bigl(J(x,y)\bigr)
		\quad \text{ among all } J : \cE \stackrel{\rma}{\to} \R
			\\
		& \text{subject to }
		&& 	\sum_{y: y\sim x} J(x,y) = m(x)
		\ \text{ for all } x \in \cX 
		\, .
	\end{align}	
As usual, $x \sim y$ means that $(x,y) \in \cE$.
The left-hand side in this constraint will sometimes be denoted  
$\DIVE J(x)$, as 
it can be thought of as the discrete divergence of the discrete vector field $J$ at the vertex $x$. 

In the simplest non-trivial example, which is of special interest, 
the cost through each edge is proportional to the flow, i.e., 
	$f_{e}(j) = \alpha_e |j|$
for all $j \in \R$,
for given edge-weights $\alpha_e > 0$.
In that case, \eqref{eq: min cost} is the classical 
\emph{Beckmann problem} \cite{MR68196}, which is equivalent to a Monge--Kantorovich optimal transport problem on $\cX$ with cost function given by the weighted graph distance $d_\alpha$ induced by the edge weights $\alpha_e$. 
In other words, the infimum in \eqref{eq: min cost} 
coincides with $W_1(m_+, m_-)$, 
the 1-Wasserstein distance (also known as earth-mover's distance) between the positive and negative part of $m$.

In this paper we allow for nonlinear and nonconvex cost functions $f_{(x,y)}$, which are assumed to be Lipschitz.
Our setting below also covers the multi-species generalisation of \eqref{eq: min cost}, 
where $m \in \M(\cX; V)$ is a multi-dimensional measure taking values in a finite-dimensional
vector space 
	$V$,
the cost functions $f_e: V \to [0,\infty)$ are given,
and the minimisation in 
	\eqref{eq: min cost} 
runs over antisymmetric functions
$J: \cE \toa V$ (we write $J \in V_a^\cE$). 
Moreover, our results apply to more general cost functions, 
not necessarily sums of edge contributions, 
but to lighten notation, 
we restrict to \eqref{eq: min cost} in the introduction.

\subsection*{Minimal cost-flow problems in the continuum}

The discrete minimisation problem \eqref{eq: min cost} has a natural counterpart in the continuum, that we shall now introduce. 
Let $\mu \in \M(\R^d)$ be a signed measure of total mass $0$, describing a spatial distribution of sources and sinks. 
Furthermore, let $f: \R^d \to [0,\infty]$ be a given cost function.
For $j \in \R^d$, 
	$f(j)$ represents the 
	energy associated to transporting a unit mass in the direction $j$. 
To ensure good lower-semicontinuity of the functional defined below, we assume that $f$ is lower semicontinuous and $\dive$-quasiconvex.

The continuum minimal cost-flow problem is the following: 
\begin{align}
	\label{eq:min-cost-cont}
	& \text{minimise} \	
	&& 
	\bF(\nu) :=
			\int_{\R^d} 
				f\Bigl(\frac{\dd \nu}{\dd \Lm^d}\Bigr)
			\dd x 
			+ 
			\int_{\R^d} 
				f^\infty\Bigl(\frac{\dd \nu}{d|\nu|}\Bigr)
			\dd \nu^s
	\quad \text{ among all } 
	\nu \in \M(\R^d;\R^d)
		\\
	& \text{subject to }
	&& 	\dive(\nu) = \mu
	\, .
\end{align}	
Here, 
$f^\infty : \R^d \to [0,\infty]$
is the recession function of $f$ defined by 
		$
			f^\infty(j) := \limsup_{t\to \infty} 
				\frac{f(tj)}{t}
		$,
and 
$\nu^s$ denotes the singular part of $\nu$ with respect to the Lebesgue measure $\Lm^d$.
We refer to \cite[Section 4.4]{santambrogio} for a discussion of minimisation problems with divergence constraints in the continuum.	

It will be convenient below to incorporate the constraint in the objective functional, as is often done. We thus define
	\begin{align}
		\label{eq:limitform}
		\bF(\cdot | \mu)
		:	 	\M(\R^d ; \R^d) \to [0,\infty] \, ,
		\qquad
		\bF(\nu | \mu) 
		:=
			\begin{cases}
			\displaystyle
			\bF(\nu)
				 &\text{ if }\dive \nu = \mu \, , \\+\infty&\text{ otherwise.} 
			\end{cases}
		\end{align}

\subsection*{Random setup}

In this paper we consider the discrete nonlinear minimum-cost flow problem \eqref{eq: min cost} on a large random graph embdedded in $\R^d$.
Let us informally present the setup in a simplified setting. For full details we refer to Section \ref{sec:setup} below.  

Let $(\cX, \cE)$ be a random graph embedded in $\R^d$, which 
is assumed to be \emph{stationary} (in distribution) with respect to translations of $\Z^d$.
We further assume that there exist (possibly random) constants $R_1, R_2, R_3 < \infty$ such that the following assertions hold:
 \begin{enumerate}[align=left]
	\item [{(G1)}] \label{item:G1-intro}
	(Absence of large gaps.) \ For all $x\in \R^d$ we have $ \cX \cap B(x,R_1
	)\neq \emptyset$.
	\item [{(G2)}] \label{item:G2-intro}
	(Quantitative connectedness.) For all $x, y\in \cX$ there is a path $P$ in $(\cX,\cE)$ connecting $x$ and $y$ with Euclidean length
	\begin{align*}
		\length(P) \leq R_2 \bigl(|x-y| + 1\bigr) \, .
	\end{align*}
	\item [{(G3)}] \label{item:G3-intro}
	(Bounded edge-lengths.) For all $(x,y) \in \cE$ we have
			$| x-y | \leq R_3$.
	\end{enumerate}
The simplest example is the Cartesian grid $(\Z^d, \E^d)$, 
but our framework also covers random configurations of points. 
It does not cover the case where $\cX$ is a Poisson point process, which remains an interesting challenge.

We endow the edges of the random graph $(\cX, \cE)$ 
with stationary random cost functions 
	$f_e: \R \to [0,\infty)$ 
for $e \in \cX$, which are assumed to be Lipschitz and of linear growth; see Section~\ref{sec:edge_based} below for more details.
Our goal is to describe the large-scale behaviour of the resulting minimum-cost flow problem.

For this purpose, fix $\eps \in (0,1]$, and consider  the rescaled random graph $(\cX_\eps, \cE_\eps)$ in which edge-lengths are of order $\eps > 0$; that is, 
	$\cX_\eps := \eps \cX$
and 
	$\cE_\eps := \eps \cE$.
We endow each edge $e \in \cE_\eps$ 
with the rescaled random cost function $f_e^\eps$
defined by
	$f_e^\eps(j) := \eps^d f_{e/\eps}(j/\eps^{d-1})$.
For given $m_\eps \in \M(\cX_\eps)$ with $m_\eps(\cX_\eps) = 0$,
we thus arrive at the rescaled minimisation problem
	\begin{align}
		\label{eq: min cost rescaled}
		& \text{minimise} \	
		&& 
			F_\eps(J) 
			:= 
			\eps^d
			\sum_{(x,y)\in\cE_\eps} 
			 f_{e/\eps}
			 \Bigl(\frac{J(x,y)}{\eps^{d-1}}\Big)
		\quad \text{ among all } 
		J : \cE_\eps \stackrel{\rma}{\to} \R
			\\
		& \text{subject to }
		&& 	\sum_{y: y\sim x} J(x,y) = m_\eps(x)
		\ \text{ for all } x \in \cX_\eps
		\, .
	\end{align}	
Suppose now that the measures $m_\eps$, viewed as signed measures on  $\R^d$, converge in the Kantorovich--Rubinstein norm $\|\cdot \|_{\tKR}$ to a signed measure $\mu \in \M_0(\R^d)$.
Our main result describes the asymptotic behaviour of these minimisation problems as $\eps \to 0$.

\begin{rem}
	Consider the special case where $(\cX, \cE)$ is the Cartesian grid 
		$(\Z^d, \E^d)$ 
	endowed with iid random positive edge weights $\tau_e$.
	If the cost functions $f_e$ take the form 
		$f_e(j) = \tau_e |j|$, 
	then \eqref{eq: min cost} is closely related to the problem of 
	\emph{first passage percolation} \cite{Auffinger-Damron-Hanson:2017}. 
	Indeed, if $m = \delta_a - \delta_b$ for $a, b \in \Z^d$, 
	then the infimum in \eqref{eq: min cost} coincides with the first passage time between $a$ and $b$.
\end{rem}

\subsection*{Main result: 
\texorpdfstring{$\Gamma$}{Gamma}-convergence 
in the scaling limit}

In order to state the main convergence result, 
we embed the discrete problem in a continuous framework. 
In particular, we will identify a discrete vector field 
	$J : \cE_\eps \toa \R$
with a singular continuous vector field
	$\iota_\eps J \in  \M(\R^d; \R^d)$
supported on the line segments 
	$[x,y] \subseteq \R^d$ 
for $(x,y) \in \cE$, namely
 	\begin{align*}
	 	\iota_\eps J 
		:= \frac12\sum_{(x,y)\in \cE_\eps} 
		 J(x,y) \frac{y-x}{|y-x|} \Hm^1\llcorner[x,y] \, .
	\end{align*}
Note that $\dive \iota_\eps J = \DIVE J$, where the latter is identified with a signed measure on $\R^d$ supported on $\cX_\eps$. 

For $m \in \M_0(\cX_\eps)$
we then consider the random functionals
	\begin{align}
		\label{eq:def_Fmu-intro}
		F_\eps(\cdot | m)
	:	 
	\M(\R^d ; \R^d) \to [0,+\infty] \, ,
	\qquad
		F_\eps(\nu | m)
			:=
		\begin{cases}
			F_\eps(J)
				& \text{ if } \nu = \iota_\eps J
					\, , \quad
					\DIVE J = m\, ,
					\\
			+\infty
				&\text{ otherwise}  \, .
		\end{cases}
	\end{align}
By the remarks above, the condition	
	$\DIVE J = m$ can be replaced equivalently by $\dive \nu = m$.

In the simplified setting of this introduction, our main result reads as follows; 
see Theorem \ref{thm:main} for a more general version.

	\begin{theorem}[Main result, special case]
	\label{thm:main-intro}
		Let $(\cX,\cE)$ be a stationary random graph satisfying the assumptions (G1)--(G3), endowed with 
		a stationary random family of cost functions $f_e$
		that are Lipschitz and of linear growth.
		Assume that the random measures
			$m_\eps \in \M(\cX_\eps)$
		converge narrowly to the random measure
			$\mu\in \M(\R^d)$ 
		almost surely.
		Then we have almost surely
			$\Gamma$-convergence 
		\begin{align*}
			F_\eps (\cdot | m_\eps) 
				\Gto
			\bF_{\hom} (\cdot | \mu) 
			\qquad  \text{ as } 
			\eps \to 0\, ,
		\end{align*}
		in the vague and narrow topologies on $\M(\R^d; \R^d)$.
		Moreover,
		\begin{align*}
			\inf_{ J : \cE_\eps \toa \R } 
				\bigl\{ F_\eps (J) 
					\suchthat
					\DIVE J = m_\eps
				\bigr\}
				\to 
			\inf_{ \nu \in \M(\R^d ; \R^d) }
				\bigl\{ \bF_{\hom} (\nu) 
					\suchthat
					\dive \nu = \mu 
				\bigr\} 
				\qquad  \text{ as } 
			\eps \to 0\, .
		\end{align*}
		The limit functional $\bF_{\hom} (\cdot | \mu)$ 
		is of the form \eqref{eq:limitform},
		with a (possibly random) energy density $f_{\hom}$ that is stationary,
		almost surely lower semicontinuous, $\dive$-quasiconvex, and of linear growth. 
	\end{theorem}
Since the functionals $f_e$ are assumed to be Lipschitz and of linear growth, the infima above are attained.

In the special case where $f_{e}(j) = \alpha_e |j|$,
Theorem \ref{thm:main-intro} implies that the 
$1$-Wasserstein distance on $\cX$ converges, after rescaling, to a $1$-Wasserstein distance on $\R^d$
induced by a non-trivial (deterministic, whenever ergodicity is assumed) norm, which depends on the microscopic properties of the random graph $(\cX,\cE)$ and the random edge-weights $\alpha_e$.

\subsection*{The effective energy density}

A significant part of the paper deals with the construction of the energy density $f_{\hom}$ appearing in Theorem \ref{thm:main-intro} as the solution to a suitable variational problem. 
Substantial new ideas are required to treat random graphs without a periodic structure.
Let us informally explain the main steps in the construction.

First, we consider a localised version of the rescaled cost functional $F_\eps$ from \eqref{eq: min cost rescaled}.
That is, for a Borel set $A \subseteq \R^d$, we define
\begin{align*}
	F_\eps(J,A) 
			:= 
			\eps^d
			\sum_{(x,y)\in\cE_\eps} 
			 \lambda_A(x,y)
			 f_{e/\eps}
			 \Bigl(\frac{J(x,y)}{\eps^{d-1}}\Bigr)\, , 
			 \quad \text{ where }
			 \lambda_A(x,y) := \frac{\Hm^1([x,y] \cap A)}{\Hm^1([x,y])} \, .
\end{align*}
Note that the weight $\lambda_A(x,y) \in [0,1]$ is the proportion of the line segment $[x,y]$ that is contained in $A$.

Second, given a direction $j \in \R^d$, we would like to minimise $F_\eps(J,A)$ among all discrete vector fields $J : \cE_\eps \toa \R$ 
which ``behave outside of $A$ like the constant vector field $j$ on large scales.''
While such vector fields can be constructed straightforwardly on the Cartesian grid $(\Z^d, \E^d)$, the existence of such vector fields on non-periodic graphs is not straightforward. 
Nevertheless, under $(G1)$ and $(G2)$, 
we show that there indeed exist 
a linear operator
	$\calR$
with suitable boundedness properties,
that assigns to each direction $j \in \R^d$
a divergence-free discrete vector field 
	$\calR j : \cE \toa \R$,
such that the corresponding rescaled 
fields $\calR_\eps j : \cE_\eps \toa \R$ converge vaguely, 
after embedding,
to the constant vector field $j \Lm^d$,
as $\eps \to 0$.
Having constructed the operator $\calR$, 
we  define the local energy density on $A$ by 
\begin{align*}
		f_{\eps,\calR}(j,A) :=
		 \inf
		\left\{
			F_{\eps}(J,A)
		\suchthat
			J \in \Rep_{\eps,\calR}(j;A)
		\right\} \, ,
\end{align*}
where $\Rep_{\eps,\calR}(j;A)$ denotes the class of all representatives of $j$ on $A$, i.e.,~
all divergence-free discrete vector fields 
on the rescaled graph $\cE_\eps$ that coincide with 
$\calR_\eps j$ on a neighbourhood of $\R^d \setminus A$.

Third, let $Q$ be the centered open unit cube in $\R^d$, 
and define the homogenised energy density $f_{\hom}$ by
\begin{align*}
	f_{\hom}(j)
		:= 
	\lim_{\eps \to 0} 
	\frac{f_{\eps,\calR}(j,Q)} {\Lm^d(Q)}
\end{align*}
for $j \in \R^d$. It is a consequence of 
Kingman's subadditive ergodic theorem 
	\cite{kingman1973subadditive,akcoglu1981ergodic,Licht-Michaille:2002} 
that this limit exists almost surely.
While the operator $\calR$ is not unique, it follows from our 
results that the limiting functional $f_{\hom}$ does not depend on the particular choice of this operator.
Let us remark that this is the only place in the paper where randomness plays a role. The rest of our methods are completely deterministic, in the sense that
$\omega$ in our probability space stays fixed.

\subsection*{Related work}
Discrete-to-continuum problems involving dynamical optimal transport have attracted a lot of attention in recent years. 
Convergence results for transport distances have been obtained in
\cite{
	GiMa13,
	Gladbach-Kopfer-Maas:2020,lavenant2021unconditional,slepvcev2023nonlocal,
	ishida2024quantitative} 
while convergence of the associated gradient-flow structures is studied in  
\cite{
	Disser-Liero:2015,
	Forkert2020,Esposito-Patacchini-Schlichting:2021,Hraivoronska-Tse:2023,Esposito-Heinze-Schlichting:2023,Esposito-Mikolas:2023,Esposito-Patacchini-Schlichting:2023,Hraivoronska-Tse-Schlichting:2023}.

\smallskip
	
The stochastic homogenisation result for the minimum-cost flow problem on stationary random graph obtained in this paper builds on the earlier works \cite{Gladbach-Kopfer-Maas:2020, Gladbach-Kopfer-Maas-Portinale:2023}, 
which differ from the current paper in several ways.
Most importantly,  
\cite{Gladbach-Kopfer-Maas:2020, Gladbach-Kopfer-Maas-Portinale:2023}
deal with deterministic $\Z^d$-periodic graphs. 
The stochastic setting of the current paper poses substantial additional difficulties: 
in particular, the construction of the homogenised energy is considerably more involved on non-periodic graphs.
Another difference is that
\cite{Gladbach-Kopfer-Maas:2020, Gladbach-Kopfer-Maas-Portinale:2023}
treat a dynamical optimal transport problem: 
instead of minimising over time-independent scalar flows subject to a divergence-constraint 
$\DIVE J = m$, the minimisation is over
time-dependent scalar flows satisfying 
a discrete continuity equation 
$\partial_t m_t + \DIVE J_t = 0$ with prescribed boundary conditions for $m_t$ at time $t=0$ and $t=1$.
Yet another difference is that the current paper treats Lipschitz cost functions, 
not necessarily convex, 
whereas 
	\cite{Gladbach-Kopfer-Maas:2020, Gladbach-Kopfer-Maas-Portinale:2023} 
deals with convex cost functions, not necessarily Lipschitz. 
Finally, 
\cite{Gladbach-Kopfer-Maas:2020, Gladbach-Kopfer-Maas-Portinale:2023}
treat scalar flows only,
while the current paper covers multi-species flows. 
While the cost function in \cite{Gladbach-Kopfer-Maas-Portinale:2023} 
is assumed to be of superlinear growth, the subsequent work
\cite{Portinale-Quattrocchi:2024} 
treats cost functions with mere linear growth.
The recent preprint \cite{Gladbach-Kopfer:2024}  builds on the methods developed in the current paper. It treats the dynamical optimal transport problem with quadratic cost function on stationary random graphs.

Earlier convergence results for dynamical optimal transport on random geometric graphs have been proved in \cite{Garcia-Trillos:2020}. 
That paper covers a different regime, where the degree of the graphs grows with number of vertices.
In that situation, the microscopic structure of the random graphs does not not appear in the limiting metric, which is the $2$-Wasserstein metric over Euclidean space.

\smallskip

Many works deal with discretisations of integral functionals 	
	$\int_{\R^d} f(\nabla u(x)) \dd x$  
involving gradients of Sobolev functions.
In particular, the paper 
	\cite{Alicandro-Cicalese-Gloria:2011} 
contains a stochastic homogenisation result for  energy functionals involving discrete gradients, under a superlinear growth assumption. 
The authors work with random Voronoi discretisations, under graph assumptions that are similar to those in the present paper. For a similar geometric setting and homogenisation results for BV functions in a randoom environment, see also \cite{Alicandro-Cicalese-Ruf:2015}.
Energy functionals with degenerate growth have been covered in \cite{Neukamm-Schaffner-Schlomerkemper:2017}.
The paper \cite{Braides-Caroccia:2023} studies stochastic homogenisation of quadratic energy functionals on a Poisson point cloud.

	\smallskip
The limiting functional $\bF_{\hom}(\cdot|\mu)$ that we obtain in this work belongs to a class of functionals that have been widely studied in the literature.
In particular, 
lower semicontinuity and homogenisation of integral functionals under
differential constraints of the form $\calA u = 0$ was proved in \cite{Fonseca-Muller:1999},  
for general first-order differential operators $\calA$.
Extensions have been obtained in 
\cite{Braides-Fonseca-Leoni:2000,
			Baia-Chermisi-Matias-Santos:2013,Matias-Morandotto-Santos:2015,
			Davoli-Fonseca:2016}; 
see also \cite{Conti-Muller-Ortiz:2020}.
The case $\calA = \curl$ corresponds to functionals of gradients, whereas the case $\calA = \dive$ is the relevant one for this paper.
The natural condition ensuring weak lowersemicontinuity 
is the one of $\mathcal{A}$-\emph{quasiconvexity}. This notion reduces to Morrey's classical notion of quasiconvexity \cite{Morrey:1952}
when $\mathcal{A} = \curl$.
		

In \cite{Ruf-Zeppieri:2023}, the authors treat stochastic homogenisation of random integral functionals with linear growth, at the continuous level and in the setting of curl-free measures.
Our approach shares some similarities with the latter work, in particular in the proof of the lower bound, 
where the study of tangent measures and the blow-up method play an important role.

\subsection{Structure of the paper}
In Section~\ref{sec:setup} we introduce the general setting of this paper, discussing assumptions on the graphs and the energies, and present our main results. In Section~\ref{sec:examples} we discuss applications, including the convergence of $1$-Wasserstein distances in a random environment. An overview of the strategy and a sketch of the proof of our main theorem is the content of Section~\ref{sec:sketch}. Section~\ref{sec:existence} contains the existence of discrete uniform flows, which play a crucial role in the description of the limit energy density, as described in Section~\ref{sec:multi-cell}. We proceed in Section~\ref{sec:correctors} with the proof of the existence of discrete correctors, and study the structure of divergence measures and their blow-ups in Section~\ref{sec:tangent}. Finally, the proof of the main result is included in Section~\ref{sec:ub} (upper bound) and Section~\ref{sec:lb} (lower bound). The appendix collects useful properties of the topologies employed in this paper, and some basic discrete calculus rules. 

\section{General setup and main results }
\label{sec:setup}

In this section we present the detailed setup of the 
nonlinear minimal-cost flow problems 
on stationary random graphs in $\R^d$.

\subsection{Assumptions}

 We first introduce the main objects and the main assumptions that will be in force.

\begin{itemize}[leftmargin=5mm]
	\item A probability space
	$(\Omega, \mathcal{F}, \P)$
	endowed with
	a family
		$(\sigma_z)_{z \in \Z^d}$
	consisting of measure-preserving transformations
		$\sigma_z: \Omega \to \Omega$
	satisfying
		$\sigma_{z + w} = \sigma_z \circ \sigma_w$
	for all $z, w \in \Z^d$.
	\item A random, undirected graph $\omega \mapsto (\cX_\omega,\cE_\omega)$ embedded in $\R^d$,
	i.e., $\cX_\omega$ is a countable subset of $\R^d$
	and $\cE_\omega$ is a symmetric subset of
	$\{ (x,y) \in \cX_\omega \times \cX_\omega
	\suchthat x \neq y \}$.
	We assume that $(\cX_\omega,\cE_\omega)$ is
	\textit{stationary},
i.e.,
    \begin{align}
        \big(
            \cX_{\sigma_z \omega} ,
            \cE_{\sigma_z \omega}
        \big)
            =
        ( \cX_\omega + z, \cE_\omega + (z,z) )
    \end{align}
	for all
        $z \in \Z^d$
	and
        $\omega \in \Omega$.
  To lighten notation, we omit the subscripts in $(\cX_\omega, \cE_\omega)$ from now.
  Moreover, for $\mathbb P$-a.e. $\omega \in \Omega$, we assume that there exist constants $R_1, R_2, R_3>0$ (possibly depending on $\omega \in \Omega$) so that:
\begin{enumerate}[align=left]
\item [{(G1)}] \label{item:G1}
For all $x\in \R^d$ we have $ \cX \cap B(x,R_1)\neq \emptyset$.
\item [{(G2)}] \label{item:G2}
For all $x, y\in \cX$ there is a path $P$ in $(\cX,\cE)$ connecting $x$ and $y$ with Euclidean length
\begin{align*}
    \length(P) \leq R_2 \bigl(|x-y| + 1\bigr) \, .
\end{align*}
\item [{(G3)}] \label{item:G3}
For all $(x,y) \in \cE$ we have
		$| x-y | \leq R_3$.
\end{enumerate}
\item A finite-dimensional normed space $(V, |\cdot|_V)$.
		\item A random energy functional
		\begin{align*}
			F  : \Omega \times V_\rma^\cE \times \cB(\R^d)
				\to [0, +\infty],
			\qquad
				F = F_\omega(J, A),
		\end{align*}
	   which is \textit{stationary}, i.e.,
			\begin{align}
				F_{\sigma_z \omega}(\tau_z J,A + z) = F_{\omega}(J, A).
			\end{align}
			for all		$z \in \Z^d$,
			$\omega \in \Omega$,
			$J \in V_\rma^{\cE}$,
				 and
					 $A \in \cB(\R^d)$.
			Here,
		$
		\tau_z J \in V_\rma^{\cE}
		$
		denotes the translated field defined by
		$\tau_z J(x,y) :=J (x-z, y-z)
		$ for $(x,y) \in \cE$.
 Moreover, we assume that, for $\mathbb P$-a.e. $\omega \in \Omega$, there exist constants $C_1, c_2, C_2, R_\Lip > 0$ (possibly depending on $\omega \in \Omega$) so that
\begin{itemize}[align=left]
        \item [{(F1)}] (\emph{Lipschitz continuity})
        \label{item:F1}
        $F_\omega$ is Lipschitz-continuous, in the sense that
        \begin{align}
            \label{eq:ass-F-Lipschitz}
			|F_\omega(J',A) - F_\omega(J,A)|
            &\leq C_1
            \sum_{ (x,y) \in \cE }
            |J(x,y) - J'(x,y)|_V
            \Hm^1\bigl([x,y] \cap B(A,R_\Lip)\bigr)
        \end{align}
        for all $A \in \cB(\R^d)$
        and $J, J' \in V_\rma^{\cE}$.
        \item [{(F2)}] (\emph{linear growth})
        \label{item:F2}
        $F_\omega$ has linear growth, i.e.,
        \begin{align}
            \label{eq:ass-F-linear-growth}
            F_\omega(0,A)
				\leq
			C_2 \Lm^d(A)
        \tand
            F_\omega(J,A)
        		&\geq
			c_2 \sum_{(x,y) \in \cE}
        			|J(x,y)|_V
					 \Hm^1([x,y] \cap A)
            \, ,
        \end{align}
        for all $A \in \cB(\R^d)$ and $J \in V_\rma^{\cE}$, where
        we use the notation
               \item [{(F3)}] \label{item:F3} (\emph{$\sigma$-additivity}) $F_\omega$ is $\sigma$-additive in the second variable, i.e.,
			   for all
			   $J \in V_\rma^{\cE}$
			   and all
			   pairwise disjoint sequences of Borel sets
					$\{A_i\}_{i \in \N}$
				we have
        \begin{align}
		\label{eq:additivity}
        F_\omega \Big(J, \bigcup_{i=1}^\infty A_i\Big)
            = \sum_{i=1}^\infty F_\omega(J,A_i)\, .
        \end{align}
        \end{itemize}
\end{itemize}

\begin{rem}[Locality]
\label{rem:locality}
    The Lipschitz-continuity assumption (F1) clearly implies that $F_\omega$ is a ($R_\Lip$-)local functional, i.e.
    \[
        F_\omega(J,A)
        =
        F_\omega(J',A)
    \]
    whenever $J = J'$ in $B(A,R_\Lip)$.
\end{rem}

\begin{rem}[Growth assumptions]
    The nonnegativity and the condition (F2) could be relaxed, assuming that the graph $(\cX_\omega,\cE_\omega)$ is locally finite and that 
    \begin{align}
    \label{eq:geometric_assumption}
        \alpha_\eps 
            := 
        \eps ^d 
        \sum_{[x,y] \cap \cE_\eps} 
            \Hm^1 \res [x,y] 
            \to 
                \alpha 
                \, , 
    \end{align}
    vaguely in $\M_+(\R^d)$ as $\eps \to 0$, for some $\alpha \in \M_+(\R^d)$ \footnote{Note that in this case, $\alpha = \Lm^d$ for $\mathbb{P}$-a.e. $\omega \in \Omega$, by stationarity.}. 
    In this case it would suffice to assume that 
    \begin{align}
           F_\omega(J,A)
        &\geq
\sum_{(x,y) \in \cE}
        \big(  c|J(x,y)|_V - C  \big) \Hm^1([x,y] \cap A)
                \,  ,
    \end{align}
    for some $c,C\in \R_+$ and then work with the functional $\tilde F_\omega$ given by 
    $$
        \tilde F_\omega(J,A) := F_\omega(J,A) + \alpha_1(A)
            \, , \quad 
        \forall J \in V_a^\cE   \, , \, A \in \cB(\R^d)
            \, , 
    $$
    which is nonnegative and satisfies (F2). 
\end{rem}

\begin{rem}[Additivity]
    The $\sigma$-additivity assumption of $F_\omega$ with respect to the second variable is classical and not very restrictive, as many important examples fit into this wide class (see Section~\ref{sec:examples} for more details). Nonetheless, due to their nonlocal nature,  discrete models might correspond to mildly non additive energies, which do become additive only in the limit as $\eps \to 0$.
    A weaker assumption would be to assume \textit{subadditivity} together with an \textit{almost additivity} property, e.g. the existence of $R \in \R_+$ such that, if $d_{\text{eucl}}(A_i,A_j) \geq R$ for every $i \neq j$, then \eqref{eq:additivity} holds.
    This would suffice for the energy to be local in the limit as $\eps \to 0$, and the general strategy of this work could find applications in this setting too.
    For the sake of simplicity and to avoid extra technical complications, we decided to omit these generalisations, and we restrict ourselves to the additive setting. For similar discussions we refer to \cite{Braides-Maslennikov-Sigalotti:2008} and \cite[Section~7]{DeGiorgi-Letta:1977}.
\end{rem}

\subsection{The minimum-cost flow problem}

As before, let $V$ be a finite-dimensional normed space.
Let $m \in \M(\cX;V)$ be a finite $V$-valued measure on $\cX$ such that $m(\cX) = 0$. 
The measure $m$ can be thought of as prescribing sources and sinks in a multi-commodity flow problem.
Its total variation $|m| \in \M_+(\cX)$
is a finite measure on $\cX$; 
see, e.g.~\cite[Chapter 1]{Ambrosio-Fusco-Pallara:2000}.
Clearly, $m$ can be identified with an element $m \in V^\cX$
with $\sum_{x\in \cX} m(x) = 0$ and
$\sum_{x\in \cX} |m(x)|_V < \infty$.

For $A \in \cB(\R^d)$
we are interested in the nonlinear min-flow problem:
\begin{align}
	\label{eq: min cost two}
	& \text{minimise} \	
		F_\omega(J,A)
	\quad \text{ among all } 
	J \in V_{\rma}^\cE 
	 \text{ satisfying }
	 	\DIVE J = m
	\, .
\end{align}	
We refer to Appendix \ref{sec:preliminaries_derivatives}
for the used notation from discrete calculus.


\subsection{Localisation and rescaling}
Let $U \subset \R^d$ be either $\R^d$ or an open bounded domain with Lipschitz boundary. 
For $\eps \in (0,1]$ we consider the rescaled and localised graph 
	$(\cX_\eps,\cE_\eps)$ 
defined by
\begin{align}
	\cX_\eps:= \eps \cX \cap \overline U
	\tand 
		\cE_\eps :=
		\Big\{
			(\eps x, \eps y) \in \weX \times \weX
				\suchthat
			(x,y)\in \cE
		\Big\}  \, .
\end{align}
To lighten notation, we suppress the dependence on $U$. 
Note that $(\cX_\eps, \cE_\eps)$ is a connected graph if $\eps>0$ is small enough.

Consider the localized and rescaled energy
$F_{\omega,\eps}  : V_\rma^{\cE_\eps} \times \cB(\R^d)
\to [0, +\infty]$ defined by
\begin{align}
	\label{eq:rescaled-energy}
	F_{\omega,\eps}(J,A)
	:=
	\eps^d
	F_{\omega} \bigg(
		\frac{J(\eps \cdot)}{\eps^{d-1}},
		\frac{A}{\eps} \bigg)
\end{align}
for $J \in V_\rma^{\cE_\eps}$ and $A\in \mathcal{B}(\overline U)$.
Here and below we identify 
	$J\in V_\rma^{\weE}$
with its natural extension in $V_\rma^{\eps \cE}$ given by
\begin{align}
	J(x,y) := 0	\, , \quad
		\forall (x,y) \in \eps \cE \setminus \weE\, .
\end{align}
Note the nonstandard scale $\eps^{d-1}$ instead of $\eps^d$ in the denominator, which is due to the missing length scale in the discrete divergence.

\begin{defi}[Embedding and rescaling]
\label{def:embedding}
	Let $\eps \in (0,1]$. For $J\in V^{\cE_\eps}_a$, we consider the embedded Radon measures given by
 	\begin{align}
			\label{eq:def-iota-J}
	\qquad 	\iota_\eps J &:= \frac12\sum_{(x,y)\in \cE_\eps} \Big( J(x,y)\otimes \frac{y-x}{|y-x|} \Big) \Hm^1\llcorner[x,y] \in \M(\R^d;V\otimes \R^d)
			\, .
 	\end{align}
	Slightly abusing notation, 
	let
		$\iota_\eps m \in \M(\R^d; V)$ 
	be the trivial 
	embedding 
	of
	$m \in \M(\cX_\eps; V)$.
\end{defi}
As observed before,  
$(m,J)$ solves $\DIVE J = m$ if and only if $\dive \iota_\eps J = \iota_\eps m$ in the sense of distributions; see also Lemma~\ref{lemma:divergence_discr_cont}.

It will be convenient to reformulate the assumptions  (F1) and (F2) in terms of the embedding $\iota_\eps$.

\begin{lemma}[Properties of $F_\eps$]
	\label{lemma:tv_bounds}
	Let $\eps > 0$.
    \begin{enumerate}
        \item \label{item:F-lip} For all $J, J' \in V_\rma^{\cE_\eps}$ and $A \in \cB(\R^d)$ we have
        \begin{align}
        \label{eq:Lipschitz-F-eps}
            |F_{\omega,\eps}(J,A) - F_{\omega,\eps}(J',A)| \leq 2 C_1 \bigl| \iota_\eps (J - J') \bigr| \bigl(B(A,\eps R_\Lip)\bigr) \, .
        \end{align}
        \item \label{item:F-lower}
        For all  $J \in V_\rma^{\cE_\eps}$ and $A \in \cB(\R^d)$ we have
		\begin{align}
		\label{eq:TV-F-eps}
			F_{\omega,\eps}(0,A)
				\leq
			C_2 \Lm^d(A)
			\tand
			F_{\omega,\eps}(J,A)
				\geq
			2 c_2
			| \iota_\eps J |(A) \, .
		\end{align}
	\end{enumerate}
\end{lemma}

\begin{proof}
    The first bound in \eqref{item:F-lower} follows immediately from the definition of the rescaled energy and the scaling property of the Lebesgue measure.

	We next show the second bound in \eqref{item:F-lower}.
	Using the definition of the rescaled energy, the growth condition (F2), and \eqref{eq:total_var_iota},
	we obtain
	\begin{align}
		F_{\omega,\eps}(J,A)
		  &=
		\eps^d
		F_\omega \bigg(
		\frac{J(\eps \cdot)}{\eps^{d-1}},
		\frac{A}{\eps} \bigg)
            \geq
		c_2 \eps
		\sum_{(x,y) \in \cE}
		|J(\eps x, \eps y)|_V
            \Hm^1\bigl( [x,y] \cap (A/\eps) \bigr)
\\
		  &=
		c_2
		\sum_{(x,y) \in \cE_\eps}
		|J(x,y)|_V
            \Hm^1\bigl( [x,y] \cap A \bigr)
		  =
		2 c_2 |\iota_\eps J|(A)
    \, ,
	\end{align}
	which is the claimed bound.
	Here we used the identity
\begin{align}
	\label{eq:total_var_iota}
	|\iota_\eps J|
	=
	\frac12
	\sum_{(x,y) \in \cE_\eps}
		| J(x,y) |_V \Hm^1 \llcorner [x,y] \, .
\end{align}

	The bound in \eqref{item:F-lip} can be proved in the same way using (F1).
\end{proof}


\subsection{Statement of the main result}

In order to state the main result, we define the relevant functionals, which were already introduced in Section \ref{sec:intro} in a more restrictive setting.

\begin{defi}
	For $\eps  >0$, $\omega \in \Omega$, and $m \in \M_0(\cX_{\eps};V)$, we define the embedded
    functional $F_{\omega,\eps}(\cdot|m): \M(\overline{U} ; V \otimes \R^d) \to [0,+\infty]$ by
	\begin{align}
		\label{eq:def_Fmu}
		F_{\omega,\eps}(\nu|m)
			:=
		\begin{cases}
			F_{\omega,\eps}(J,\overline U)
				& \text{ if } \nu = \iota_\eps J
					\text{ and }
				\DIVE J = m \, ,
					\\
			+\infty
				&\text{ otherwise}  \, .
		\end{cases}
	\end{align}
	\end{defi}	 	
The limit functionals appearing in our main result are of the following form.

\begin{defi}
Let $f : V \otimes \R^d \to [0,+\infty)$ be given.
For $\mu \in  \M_0(\overline{U};V)$	
we define
\begin{align}
	\label{eq:limitform-main}
	\bF(\cdot | \mu)
	:	 	\M(\overline U  ; V \otimes \R^d) 
			\to 
			[0,\infty] \, ,
	\qquad
	\bF(\nu | \mu) 
	:=
		\begin{cases}
		\displaystyle
		\bF(\nu)
			 &\text{ if }\dive \nu = \mu \, , \\+\infty&\text{ otherwise,} 
		\end{cases}
	\end{align}
where
\begin{align}
	\bF(\nu) :=
			\int_{\overline U} 
				f\Bigl(\frac{\dd \nu}{\dd \Lm^d}\Bigr)
			\dd \Lm^d
			+ 
			\int_{\overline U} 
				f^\infty\Bigl(\frac{\dd \nu}{d|\nu|}\Bigr)
			\dd \nu^s
	\, .
\end{align}	
\end{defi}

As above, 
$f^\infty : V \otimes \R^d \to [0,\infty]$
is the recession function of $f$ defined by 
		$
			f^\infty(j) := \limsup_{t\to \infty} 
				\frac{f(tj)}{t}
		$,
and 
$\nu^s$ denotes the singular part of $\nu$ with respect to the Lebesgue measure.

For the definition of the Kantorovich--Rubenstein distance $\tKR$ on the space of measures we refer to Definition~\ref{def:KR}.

\begin{theorem}[Main result]
\label{thm:main}
	Let $U \subset \R^d$ be either $\R^d$ or an open bounded domain with Lipschitz boundary. 
	Let $(\cX,\cE)$ be a stationary random graph satisfying the assumptions \emph{(G1)}--\emph{(G3)}, 
	and let $F$ be a stationary random 
	energy functional
	satisfying \emph{(F1)}--\emph{(F3)}.
	If $U = \R^d$, assume additionally that $C_2=0$. 
	Assume that the random measures
	$m_\eps \in \M(\cX_\eps;V)$
	converge almost surely
	to the random measure
	$\mu\in \M(\overline{U})$ in the $\tKR$ topology.
Then we have almost surely
			$\Gamma$-convergence
		\begin{align*}
			F_{\omega,\eps} (\cdot | m_\eps) 
				\Gto
			\bF_{\omega,\hom} (\cdot | \mu) 
			\qquad  \text{ as } 
			\eps \to 0\, ,
		\end{align*}
		in the vague and narrow topologies on $\M(\overline{U}; V \otimes \R^d)$.
		The limit functional $\bF_{\omega,\hom} (\cdot | \mu)$ 
		is of the form \eqref{eq:limitform-main},
		with a (possibly random) energy density $f=f_{\omega,\hom}$ that is stationary,
		almost surely lower semicontinuous, $\dive$-quasiconvex, and of linear growth. 
\end{theorem}

\begin{rem}[Quasiconvexity]
The notion of div-quasiconvexity  \cite{Fonseca-Muller:1999} 
is natural 
in minimisation problems with divergence constraints, yielding good lower semicontinuity properties of the energy functional.
	A measurable, locally bounded function 
		$f: V \otimes \R^d \to \R$ 
	is said to be $\dive$-quasiconvex if, for every $j \in V \otimes \R^d$ and $Q \subset \R^d$, we have
	\begin{align}
		f(j) \leq
		\fint_Q f(j + h) \de \Lm^d
			\, , \quad
		\forall h:\R^d \to V \otimes \R^d
			\, , \quad h \in C^\infty_c(Q)
			\, , \quad \dive h =0 \, .
	\end{align}
\end{rem}

\begin{rem}[Existence of the limit in the recession function]
\label{rem:existence_recess}
	It has been proved in \cite[Corollary~1.13]{DePhilippis-Rindler:2016} that if $\dive \nu = \mu \in \M(\overline U ; V)$, then
	\begin{align}
		\text{rank}
		\left(
			\frac{\de \nu}{\de |\nu|}(x)
		\right) \leq n-1
			\, , \quad
		\text{for} \ \ |\nu|^s\text{-a.e.} \ x \in \overline U \, .
	\end{align}
	Moreover, $\dive$-quasiconvex functions are convex along directions of rank at most $n-1$ (see e.g. \cite[Lemma~2.4]{Conti-Muller-Ortiz:2020} - also \cite[Proposition~3.4]{Fonseca-Muller:1999} for $f$ upper semicontinuous). Therefore, for what concerns the definition of $\bF_{\hom} (\cdot | \mu) $, the limsup in the definition of the recession function can be replaced by a limit.
\end{rem}

\subsection{Discrete uniform flows and the homogenised energy density}
\label{sec:uniform_repr_intro}
In order to describe and compute the limit density $f_{\hom}$, we shall introduce the concept of discrete uniform flow on a stationary graph. A \emph{uniform-flow operator} for a graph $(\cX,\cE)$ embedded in $\R^d$ is a bounded linear operator $\calR \in \mathrm{Lin}(V\otimes \R^d; V_\rma^\cE)$ so that
\begin{enumerate}
	\item (divergence free) $\DIVE \calR  j = 0$ for all $j\in V\otimes \R^d$.
	\item (convergence) the rescaling $\calR_\eps\in \mathrm{Lin}(V\otimes \R^d;V_\rma^{\cE_\eps})$ defined by $\calR_\eps j(\eps x, \eps y) := \eps^{d-1} \calR  j(x,y)$ is so that
\begin{align}
\label{eq:J0 properties_1_intro}
	\iota_\eps \calR_\eps j \to j\Lm^d \quad \text{ vaguely as } \eps \to 0
        \, , \qquad
    \forall j \in V \otimes \R^d
        \, .
\end{align}
	\item (boundedness) there is a constant $C>0$ such that, for every $\eps >0$,
	\begin{align}
\label{eq:J0 properties_2_intro}
	|\iota_\eps \calR_\eps j|(Q)\leq C |j| \Lm^d(Q)
\end{align}
whenever $Q$ is an orthotope containing a cube of side-length $\eps$.
\end{enumerate}
One important contribution of this work is to show the existence of an uniform-flow operator on a graph which satisfies our geometric assumptions (cfr. Proposition~\ref{prop:J0}).

\begin{prop}[Existence of uniform-flow operators]\label{prop:J0_intro}
	Every graph $(\cX,\cE)$ embedded in $\Z^d$ satisfying (G1) and (G2) admits a uniform-flow operator $\calR \in \mathrm{Lin}(V\otimes \R^d; V_\rma^\cE)$.
\end{prop}

For a given operator $\calR$, we can the define the corresponding variational cell-problems: for $\eps \in (0,1]$, $j \in V \otimes \R^d$, and $A \in \cB(\R^d)$, we define 
\begin{align}
\label{eq:def_feps_intro}
    f_{\omega, \eps,\calR}(j,A) 
        := 
    \inf 
    \left\{
        F_{\omega,\eps}(J,A)
		  \suchthat
		J \in \Rep_{\eps,\calR}(j;A)
    \right\}
        \, ,
\end{align}
where $\Rep_{\eps,\calR}(j;A)$ denotes the set of all \textit{representatives of $j$ on $A$}, i.e. the set of all $J \in V_a^{\cE_\eps}$ such that $\DIVE J = 0$ and 
\begin{align}
    J(x,y) = \calR_\eps j (x,y) 
        \, , \quad 
    \forall (x,y) \in \cE_\eps 
        \quad \text{with} \quad 
    \dist\bigl([x,y],\R^d\setminus A\bigr) \leq \eps R_\partial
        \, , 
\end{align}
for $R_\partial := \max\{ R_\Lip , R_3\}$. Note that all the objects involved are random, we simply omit the explicit dependence on $\omega$ for simplicity. Thanks to stationary and an application of the ergodic theorem, the limit as $\eps \to 0$ of $f_{\eps,\calR}$ can be used to describe the effective energy density $f_{\hom}$, independently of the choice of the uniform-flow operator $\calR$. More precisely, for a given $j \in V \otimes \R^d$, we have that ($\mathbb{P}$-almost every $\omega \in \Omega$)
\begin{align}
    f_{\omega,\hom}(j) 
        = 
    \lim_{\eps \to 0} 
    \frac
        {f_{\omega,\calR}(j,A/\eps)}
        {\Lm^d(A/\eps)}
            \, , 
\end{align}
where $A \subset \R^d$ is an arbitrary nonempty, open, convex, and bounded set. In particular, the limit does not depend on the specific choice of uniform-flow operator $\calR$. 

\subsection{Compactness}
Another important consequence of the growth assumptions on the energy is the sequential (pre)compactness of discrete fluxes with bounded energy. More precisely, let $J_\eps$ be a sequence of discrete fluxes on $\cE_\eps$ so that 
\begin{align}
    \sup_{\eps>0} F_{\omega,\eps}(J_\eps, \overline U)< \infty
        \, , \quad 
    \DIVE J_\eps = m_\eps 
        \, , \quad 
    m_\eps \to \mu 
        \in \M(\overline{U};V)
            \, \,  \text{vaguely as }\eps \to 0
                \, .
\end{align}
Then by the linear growth assumption (Lemma~\ref{lemma:tv_bounds}) we infer that 
\begin{align}
    \sup_{\eps > 0} |\iota_\eps J_\eps|(\overline U) < \infty
        \, .
\end{align}
As a consequence, up to subsequence we have that $\iota_\eps J_\eps \to \nu \in \M(\overline U;V\otimes \R^d)$ vaguely as $\eps \to 0$. Note that if $U$ is bounded, compactness holds in the narrow topology as well. As a consequence, using that the divergence is stable with respect to the distributional convergence (in particular, the vague convergence), we also conclude that $\dive \nu = \mu$.

As a corollary, we obtain the convergence of the associated minimisers and minimal values.
\begin{cor}[Convergence of the minima and minimisers]
\label{cor:minimisers}
    Under the same assumptions of Theorem~\ref{thm:main}, 
	for $\mathbb P$-a.e. $\omega \in \Omega$, and for every $\eps >0$, the constrained functional $F_{\omega,\eps}(\cdot|m_\eps)$ admits a minimiser. Moreover, we have
    \begin{align*}
			\lim_{\eps \to 0}
            \min_{ J : \cE_\eps \toa \R } 
            F_{\omega,\eps}(J|m_\eps)
                =
			\min_{ \nu \in \M(\R^d ; \R^d) }
				\bF_{\omega,\hom} (\nu|\mu) 
				\, .
		\end{align*}
    Furthermore, if $J_\eps$ is an approximate minimiser for $F_{\omega,\eps}(\cdot|m_\eps)$, i.e.,
    \begin{align}
        \lim_{\eps \to 0}
        \Big|
            \min F_{\omega,\eps}(\cdot|m_\eps) - F_{\omega,\eps}(J_\eps|m_\eps)
        \Big|
            = 0
                \, ,
    \end{align}
    then $\{ \iota_\eps J_\eps \}_\eps$ is compact in the vague (narrow if $U$ is bounded) topology, and any limit point $\nu \in \M(\overline U; V\otimes \R^d)$ is a minimiser for $\bF_{\omega,\hom}(\cdot|\mu)$. If $ \bF_{\omega,\hom}(\cdot|\mu)$ admits a unique minimiser $\nu$, then $\iota_\eps J_\eps \to \nu$ vaguely (narrowly if $U$ is bounded).
\end{cor}

\section{Examples}
\label{sec:examples}

In this section we present examples and applications of our main result. 
Throughout this section, $U\subset \R^d$ is either the full space $\R^d$ or a bounded Lipschitz domain.

\begin{example}[Edge-based costs]
\label{sec:edge_based}
We shall discuss how the special class of edge-based energies from Section \ref{sec:intro} fits into the framework of Section \ref{sec:setup}.
As in Section \ref{sec:setup} we fix a probability space
	$(\Omega, \mathcal{F}, \P)$
	endowed with
	a family
		$(\sigma_z)_{z \in \Z^d}$
	consisting of measure-preserving transformations
		$\sigma_z: \Omega \to \Omega$
	satisfying
		$\sigma_{z + w} = \sigma_z \circ \sigma_w$
	for all $z, w \in \Z^d$.
	
Let $(\cX,\cE)$ be a stationary random graph satisfying  (G1)--(G3).
We endow the edges 
	$(x,y) \in \cE$ with
random cost functions
	$f_\omega^{xy} : V \to [0,\infty)$
that are stationary in the sense that 
	\begin{align}
		f_{\sigma_z \omega}^{xy} = f_\omega^{(x+z)(y+z)}
			 \quad 
		\forall \omega \in \Omega 
			\, , \, 
		\forall z \in \Z^d
			\, .
	\end{align}
We impose the following conditions for $\P$-a.e. $\omega \in \Omega$:
	\begin{itemize}
		\item There exists $L \in \R_+$ such that $f_{\omega}^{xy}$ is $L$-Lipschitz for all $(x,y) \in \cE$.
		\item There exists $c_2 \in \R_+$ such that $\displaystyle f_\omega^{xy} (J) \geq c_2 |J|$ for all $J \in V$ and $(x,y) \in \cE$.
		\item We have that 
			$\displaystyle f_\omega^{xy}(0) =0$ for all $(x,y) \in \cE$.
	\end{itemize}	
We then consider the functional
$F_\omega:  V^{\cE}_\rma \to \R$
defined by
\begin{align}
\label{eq:edge_based}
	F_\omega(J)
		:= 
	\frac12
	\sum_{(x,y)\in \cE} 
		f_\omega^{xy}(J(x,y)) 
		\, .
\end{align}
Localising this energy functional, we define 
$F_\omega: V_a^\cE \times \cB(\R^d) \to \R$ by
\begin{align}
	F_\omega(J,A):=
	\frac12
	\sum_{(x,y)\in \cE} 
		f_\omega^{xy}(J(x,y)) 	
		\Hm^1([x,y] \cap A)
	\, . 
\end{align}
It is straightforward to check that the assumptions on 
	$f_{\omega}^{xy}$
ensure that 
(F1)--(F3) are satisfied.
Moreover, $F_\omega(0,A) = 0$
for all Borel sets $A \in \cB(\R^d)$.
Note in particular that (F1) follows from the Lipschitz properties of $f_\omega^{xy}$ with constant $C_1:= L$ and any $R_\Lip >0$.

When the set $U$ in Theorem \ref{thm:main} is bounded, the assumption that $f_\omega^{xy}(0)$ must vanish can be replaced by $\sup_{x,y} f_\omega^{xy}(0)<\infty$ if one additionally assumes that the graph satisfies $\alpha_\eps(A) \lesssim \Lm^d(A)$ for all $A \in \cB(\R^d)$, where $\alpha_\eps \in \M_+(\R^d)$ is the one-dimensional skeleton measure defined in \eqref{eq:geometric_assumption}.
The rescaled energies read as
\begin{align}
	F_{\omega,\eps}(J,A) 
		= 
	\eps^d 
	F_\omega\biggl(\frac{J(\eps \cdot)}{\eps^{d-1}}, \frac{A}{\eps}\biggr) 
		=
	\frac12 \sum_{(x,y) \in \cE_\eps} 
		\eps^{d-1} f_{\omega,\eps}^{xy}
		\Bigl(
					\frac{J(x,y)}{\eps^{d-1}}
				\Bigr) 
			\Hm^1\bigl([x,y]\cap A\bigr) 
	\, ,
\end{align}
for every $J \in V_a^{\cE_\eps}$.
Here we use the notation 
$f_{\omega,\eps}^{xy} := f_\omega^{\frac{x}\eps \hspace{-0.5mm} \frac{y}\eps}$ for $(x,y) \in \cE_\eps$.

The homogenised energy density $f_{\hom}$ takes  the form
\begin{align}
	f_{\omega,\hom}(j) 
		= 
	\lim_{\eps \to 0}
	\inf
		\left\{
			\frac12 
			\sum_{(x,y) \in \cE_\eps} 
				\eps^{d-1} f_{\omega,\eps}^{xy}
				\Bigl(
					\frac{J(x,y)}{\eps^{d-1}}
				\Bigr) 
					\Hm^1\bigl([x,y]\cap Q\bigr) 
						\suchthat 
				J \in \Rep_{\eps,\calR}(j,Q)
		\right\}
	\, , 
\end{align}
where $\Rep_{\eps,\calR}(j,Q)$ denotes the class of all representatives of $j$ on the open unit cube $Q$, i.e.,~
all divergence-free discrete vector fields 
on the rescaled graph $\cE_\eps$ that coincide with a discrete uniform flow
$\calR_\eps j$ on a neighbourhood of $\R^d \setminus Q$ (recall the definitions in Section~\ref{sec:uniform_repr_intro}).
\end{example}

\begin{example}[The integer lattice]
The simplest example of a stationary graph is the integer lattice $\cX = \Z^d$ 
with nearest-neighbour edges 
	$\cE = \bbE^d
		:= 
		\{(x,y)\in \Z^d\times \Z^d
		\suchthat
		|x-y|=1 \}$. 
This is a deterministic and periodic (thus stationary) graph of constant degree 
$2d$ satisfying the graph assumptions 
	(G1)--(G3).
On this graph, we can define a uniform-flow operator 
in a simple explicit manner. Indeed, 
the operator $\calR$
defined by 
\begin{align}
	\calR j (z,z'):= j (z' - z) \in V 
		\, , \quad 
	(z,z')\in\cE 
\end{align}
has the desired properties, as follows by 
arguing as in the proof of Proposition~\ref{prop:J0}.

For any $\eps >0$, the rescaled graph $(\cX_\eps, \cE_\eps)$ corresponds to the symmetric grid $\cX_\eps= \Z_\eps^d \cap \overline{U}$ of size $\eps>0$ within $\overline{U}$ with its nearest neighbour graph structure $\cE_\eps:= \{ x,y\in \cX_\eps \, : \, |x-y| = \eps \}$.

In the case of edge-based energies over the integer lattice, in the computation of $f_{\hom}$ one can describe the set of all representatives by
\begin{align}
	\Rep_{\eps,\calR}(j,Q)
		=
        \hspace{-1mm}
	\left\{
		J \in V_a^{\cE_\eps} 
			: 
		\sum_{i=1}^d J(x,x+\eps e_i) = 0
			\, , \,  
		J(x,y) = j (y-x) 
			\, \text{ if } [x,y] \cap \partial Q \neq \emptyset
	\right\}.
\end{align}
Beyond the integer lattice, one could consider general $\Z^d$-periodic graph as well, we refer to  \cite{Gladbach-Kopfer-Maas-Portinale:2023} for a more detailed discussion.
\end{example}

\begin{example}(Scaling limits of $1$-homogenous energies)
	We consider a special subclass of the previously described examples, in which we additionally assume that the functions $f_\omega^{xy}$ satisfy the following scaling relations:
	\begin{align}
		f_\omega^{xy}(\lambda J) = |\lambda| f_\omega^{xy}(J)
			\, , \qquad 
		\forall J \in V
			\, , \,
		(x,y) \in \cE
			\, , \,  
		\omega \in \Omega 
			\, , \, 
		\lambda \in \R 
			\, .
	\end{align}
	In this case, the formula for $f_{\omega,\hom}(\cdot)$ further simplifies as
	\begin{align}
		f_{\omega,\hom}(j) 
			= 
		\lim_{\eps \to 0}
		\inf
			\left\{
				\frac12 
				\sum_{(x,y) \in \cE_\eps} 
					f_{\omega,\eps}^{xy}
					\big(
						J(x,y)
					\big) 
						\Hm^1([x,y]\cap Q) 
							\suchthat 
					J \in \Rep_{\eps,\calR}(j,Q)
			\right\}
				\, .
	\end{align}
	It is also clear from this equality that the same scaling properties are inherited to $f_{\omega,\hom}$, namely $f_{\omega,\hom}:V \otimes \R^d \to \R_+$ is a nonnegative div-quasiconvex function which satisfies
	\begin{align}
		f_{\omega,\hom}(\lambda j) = |\lambda| f_{\omega,\hom}(j)
			\, , \quad 
		\forall j \in V \otimes \R^d 
			\, , \, 
		\lambda \in \R 
			\, .
	\end{align}
	As a consequence of the linear growth and Lipschitz assumption on the discrete energies, we also know that $f_{\omega,\hom}$ has at least linear growth and it is Lipschitz, cfr. Lemma~\ref{lemma:prop_fhom}. In particular, arguing as in \cite[Corollary~5.3]{Gladbach-Kopfer-Maas-Portinale:2023}, one infers that when $V=\R$ (hence div-quasiconvexity reduces to the usual convexity) then the limit density $f_{\omega,\hom}=: \| \cdot \|_{\omega,\hom}$ is a norm on $\R \otimes \R^d \simeq \R^d$.
	
	Let $m_\eps^+$, $m_\eps^- \in \M(\cX_\eps)$ be measures on $\overline{U}$ of equal total mass, i.e. such that 
	\begin{align}
		\sum_{x \in \cX_\eps} m_\eps^+(x) 
			= 
		\sum_{x \in \cX_\eps} m_\eps^-(x) 
			\, , 		
	\end{align}
	and define $m_\eps:= m_\eps^+ - m_\eps^-\in \M_0(\cX_\eps ; V)$. We define the discrete vectorial $\bW_1$ cost between $m_\eps^+$ and $m_\eps^-$ as 
	\begin{align}
		\bW_{1,\eps}(m_\eps^+, m_\eps^-)	
			:= 
		\inf_J
		\left\{
			\frac12 
			\sum_{(x,y) \in \cE_\eps} 
				f_{\omega,\eps}^{xy}
				\big(
					J(x,y)
				\big) 
					\Hm^1([x,y]\cap Q) 
						\suchthat 
				\DIVE J = m_\eps	
		\right\}
			.
	\end{align}
	If $V=\R$, then $\bW_{1,\eps}$ indeed corresponds to the the Kantorovich-Wasserstein $\bW_1$ distance between probability measures on the graph $\cX_\eps$ with respect to a suitable distance $d_\eps$, see \cite{Portinale-Quattrocchi:2024} for a similar discussion. 
	The case with a multi-dimensional $V$ and convex energies corresponds to a discrete version of the vectorial $\bW_1$-distance studied in \cite{Ciosmak:2021}.
	
	As a corollary of our main theorem, in particular an application of Corollary~\ref{cor:minimisers}, we deduce that $\bW_{1,\eps}$ $\Gamma$-converges ($\mathbb{P}$-almost surely) as $\eps \to 0$ with respect to the weak topology of $\mathcal{P}(\overline U) \times \mathcal{P}(\overline{U})$ to the homogenised $\bW_1$-distance 
	\begin{align}
		\bW_{1,\hom}(\mu^+, \mu^-) 
			:=
		\inf_\xi
		\left\{
			\big\| \xi \big\|_{\TV, \omega, \hom}(\overline{U})
				\suchthat 
			\dive \xi = \mu^+ - \mu^- 
		\right\}
			\, ,
	\end{align}
	where $\| \cdot \|_{\TV,\omega,\hom}$ denotes the total variation of a measure computed with respect to the norm induced by $f_{\omega,\hom}$.
\end{example}

\begin{example}(Stationary Voronoi tessellations)
The next examples describes how to construct admissible random graphs starting from suitable point processes. 
Let $X_\omega\subset \R^d$ be a stationary point process satisfying (G1) and so that 
\begin{align}
    \inf
    \left\{
        |x-y| 
            \suchthat
        x,y \in \cX_\omega, \, x \neq y 
    \right\}
        > 0
            \, .
\end{align}
In order to define the graph structure, we consider the associated Voronoi tessellation $ \mathcal V(\cX_\omega) = \{ C_\omega(x) \}_{x \in \cX_\omega}$ given by 
\begin{align}
    C_\omega(x) 
        := 
    \left\{
        z \in \R^d 
            \suchthat 
        |z-x| \leq |z -y|, 
            \quad \forall y \in \cX_\omega
    \right\}
        \subset \R^d 
            \, , 
\end{align}
and declare an edge $[x,y]$ between $x,y \in \cX_\omega$ as soon as their Voronoi cells share an interface. In mathematical terms, we define 
\begin{align}
    \cE_\omega 
        := 
    \left\{
        (x,y) \in \cX_\omega \times \cX_\omega 
            \suchthat 
        \exists z \in \R^d 
            \text{ with }
        |x-z| = |y-z| = \text{dist}(z,\cX_\omega)
    \right\}
        \, .
\end{align}
It is then easy to see that $(\cX_\omega, \cE_\omega)$ is a stationary graph satisfying (G1),(G2),(G3). For a similar construction, and the property of the constructed graph, see e.g. \cite[Section~2.1]{Alicandro-Cicalese-Gloria:2011}.

Note that the Poisson point process is beyond the scope of this article, as there are almost surely holes of any size, i.e., for any $R>0$ there is $p\in\R^d$ with $B(p,R)\cap X_\omega = \emptyset$.
\end{example}

\section{Sketch of the proof of the main result}
\label{sec:sketch}

We provide a sketch of the proof of our main result,  Theorem \ref{thm:main}, to highlight the key ideas and most important steps. We start by discussing the proof of the upper bound. 
Then we move to the lower bound, which is the most involved part of the proof.
Recall the definitions of uniform-flow operators $\calR$, $\calR_\eps$, representatives $\Rep_{\eps,\calR}$, and the cell-formula $f_{\eps,\calR}$ introduced in Section~\ref{sec:uniform_repr_intro}. For simplicity we omit the dependence on $\omega \in \Omega$ throughout this section, as most of the arguments are purely deterministic.

We fix a uniform-flow operator $\calR$ throughout this section. As before, let $U \subset \R^d$ be either $\R^d$ or an open bounded domain with Lipschitz boundary.

\subsection{Upper bound}
Given $\xi \in \M(\R^d; V \otimes \R^d)$ with $\mu :=\dive \xi \in \M(\R^d; V)$, we seek discrete fluxes $J_\eps$ so that $\DIVE J_\eps = m_\eps$, $\iota_\eps J_\eps \to \xi$, and such that
\begin{align}
    \limsup_{\eps \to 0}
        F_\eps(J_\eps, \overline U)
            \leq
        \bF_{\hom}(\xi)
            \, .
\end{align}
We proceed in two main steps: first, we show that every $\xi \in \M(\R^d ; V \otimes \R^d)$ having finite energy can be suitably approximated (both as measure and in energy) with measures having smooth densities (with respect to $\Lm^d$) and compact support in $U$ (cfr. Lemma~\ref{lemma: smooth}). Second, we show the existence of a recovery sequence for such smooth measures, and conclude using a classical diagonal argument. Let us provide a few more details for the second part.

\begin{proof}[Existence of a recovery sequence for smooth fields] \
Take $\xi = j \Lm^d$ with $j \in C_c^\infty(U; V \otimes \R^d)$, and write $\mu:= \dive j \Lm^d$. Fix $\delta>0$ and consider a countable approximate cover of $\overline U$ with disjoint cubes of size $\delta$, namely
\begin{align*}
	\lim_{\delta \to 0}
	\Lm^d
	\Big(
		\overline U \setminus \bigcup_{i\in \N} Q_\delta(x_i)
	\Big) = 0
		\, , \quad x_i \in \overline U \, .
\end{align*}
We select optimal microstructures $J_\eps^{\delta,i}
\in \Rep_{\eps,\calR} \big( j(x_i),Q_\delta(x_i) \big) 
$ solving the cell problem on the cube $Q_\delta(x_i)$, i.e.,
\begin{align}
	\label{eq:cell}
	f_{\eps,\calR} \big( j(x_i),Q_\delta(x_i) \big)
		=
	F_\eps \big( J_\eps^{\delta,i} , Q_\delta(x_i) \big)
	 \, .
\end{align}
Using a partitions of unity we glue these fields together to obtain a global discrete vector field $J_\eps^\delta$.
We do this in such a way that the restriction of $J_\eps^\delta$ coincides with $J_\eps^{\delta,i}$ in the interior of each cube $Q_\delta(x_i)$, 
whilst near the intersection of the boundaries of two cubes $Q_\delta(x_i)$ and $Q_\delta(x_j)$, 
the vector field $J_\eps^\delta$ is obtained by averaging the values $j(x_i)$ and $j(x_j)$.
By continuity of $j$, these values are very close to each other when $\delta \to 0$. For more details we refer to Step 3 in the proof of the upper bound in Section~\ref{sec:ub}.

By Lipschitz continuity and the definition of $f_{\hom}$, we then estimate
\begin{align}
	\int_{\overline U} f_{\hom}(j(x))\de x
 \,
	&\stackrel{\delta \to 0}{\simeq}
     \,
		\sum_{i=1}^{N_\delta} \Lm^d\big(Q_\delta(x_i)\big) f_{\hom}\big(j(x_i)\big)
 \,
	\stackrel{\eps \to 0}{\simeq}
     \,
		\sum_{i=1}^{N_\delta} \eps^d f_{\eps,\calR} \big( j(x_i),  Q_{\delta/\eps}(x_i/\eps) \big)
\\
	&\stackrel{\eqref{eq:cell}}{=}
		\sum_{i=1}^{N_\delta} \eps^d
    F_\eps
		\big(
			J_\eps^{\delta,i},  Q_\delta(x_i)
		\big)
	\stackrel{(a)}{\simeq}
		\sum_{i=1}^{N_\delta} \eps^d
        F_\eps
		\big(
		  J_\eps^\delta,  Q_{\delta/\eps}(x_i/\eps)
		\big)
\\
	&\stackrel{(b)}{=}
		F_\eps
		\Big(
			J_\eps^\delta\, ,  \bigcup_{i=1}^{N_\delta} Q_\delta(x_i)
		\Big)
	\stackrel{\delta \to 0}{\simeq}
		F_\eps
		(J_\eps^\delta, \overline U
		)
    \, ,
\end{align}
where the validity of $(a)$ comes from the locality properties of $F$, while to get
$(b)$ we used its additivity.
Once made rigorous, these estimates would ensure that
\begin{align}
	\limsup_{\eps \to 0}
		F_\eps(J_\eps^{\delta_\eps}, \overline U)
	\leq \int_{\overline U} f_{\hom}&(j(x))\de x  \, .
\end{align}

Furthermore, using the definition of representatives and the properties of $\| \cdot \|_{\tKR}$ , it is not hard to show that with this procedure we can construct a flux so that $J_\eps^{\delta_\eps} \to \xi$  in $\tKR(\overline{U})$ as $\eps,\delta_\eps \to 0$ suitably.

We are left with one more step, since the divergence of the glued vector fields $J_\eps^{\delta}$ does not in general coincide with $m_\eps$.
Nevertheless, $\DIVE J_\eps^{\delta}$ is $\tKR(\overline{U})$-close to $m_\eps$ (cfr.~Step 4, Section~\ref{sec:ub}). 
Therefore, Proposition~\ref{prop:correctors} ensures that we can find a corrector $K_\eps^\delta$ of small total variation,
so that the corrected field 
	$\tilde J_\eps^\delta:= J_\eps^\delta + K_\eps^\delta$ 
has the desired property $\DIVE \tilde J_\eps^\delta = m_\eps$ (cfr. Step 5, Section~\ref{sec:ub}).  Putting all things together, we are finally able to show that $\tilde J_\eps^{\delta_\eps} \to \xi$ with
\begin{align}
	\limsup_{\eps \to 0}
		F_\eps(\tilde J_\eps^{\delta_\eps}, \overline U)
  \leq
	\limsup_{\eps \to 0}
		F_\eps(J_\eps^{\delta_\eps}, \overline U)
	\leq \int_{\overline U} f_{\hom} \big(j(x)\big)\de x  \, ,
\end{align}
which shows that $\tilde J_\eps^{\delta_\eps}$ is the sought recovery sequence of $\xi$.
\end{proof}

\subsection{Lower bound}
The main tool for the proof of the lower bound is given by the \textit{blow-up method} by Fonseca--M\"uller \cite{Fonseca-Mueller:1992}.

Let $J_\eps \in V_\rma^{\cE_\eps}$ be a sequence of discrete fluxes that converge vaguely, after embedding, to a limit measure $\xi \in \M(\overline U; V \otimes \R^d)$ and so that $\dive \iota_\eps J_\eps \to \mu$ vaguely, for some $\mu \in \M(\overline U; V)$. 
In particular, note that $\dive \xi = \mu$.   
The goal is to show that
\begin{align}
\label{eq:lb_intro}
    \liminf_{\eps \to 0}
        F_\eps(J_\eps, \overline{U})
            \geq
        \bF_{\hom}(\xi)
            \, .
\end{align}
Without loss of generality we assume that $\sup_\eps F_\eps(J_\eps, \overline{U}) < \infty$.
Consider now the positive Borel measures $\nu_\eps := F_\eps(J_\eps,\cdot) \in \M_+(\overline U)$. By the local compactness of $\M_+(\overline U)$ in the vague (narrow if $U$ is bounded) topology, we can assume that, up to extracting a subsequence, there exists $\nu\in \M_+(\overline U)$ such that
	\begin{equation}
		\begin{cases} \displaystyle
			\lim_{\eps \to 0} \nu_\eps(\overline U) = \liminf_{\eps \to 0} F_\eps(J_\eps; \overline{U}) \geq \nu(\overline U) \, , \\
			\nu_\eps \to \nu \ \text{ vaguely} \, , \\
			\iota_\eps J_\eps \to \xi \ \text{ vaguely} \, ,
		\end{cases}
	\end{equation}
	where the last convergence follows (up to extraction of a subsequence) by the linear growth assumption on $F$.
	To prove the sought bound \eqref{eq:lb_intro}, 
	it is not hard to see that it suffices to show the following inequalities:
	\begin{align}	\label{eq:lb_abs_cont_part_intro}
		f_{\hom}\left(\frac{\ddd \xi}{\ddd x}\right) \leq \frac{\ddd \nu}{\ddd x} \, ,  \qquad & \Lm^d\text{-a.e.~in }\overline U \, ,
		\\
			\label{eq:ub_blowup_recess_intro}
		f_{\hom }^\infty\left( \frac{\ddd \xi}{\ddd |\xi|}\right) \leq \frac{\ddd \nu}{\ddd |\xi|} \, ,  \qquad & |\xi|^s \text{-a.e.~in }\overline U \, .
	\end{align}
Here and below, 
for a vector-valued measure $\zeta \in \M(\overline{U}; W)$
with values in a 
finite-dimensional normed space $W$,
we write 
	$\frac{\ddd \zeta}{\ddd \sigma}$
for the density of its absolutely continuous part 
in the Lebesgue decomposition with respect to a given measure $\sigma \in \M_+(\overline U)$.
For every nonempty, bounded, open, convex subset of $R \subseteq \R^d$,
the Besicovitch differentiation theorem (cfr. Proposition~\ref{prop:Besicovitch}) ensures that
 \begin{align}
 \label{eq:density_intro}
     \frac
        {\ddd \zeta}
        {\ddd \sigma}
    (x)
        =
    \lim_{\delta \to 0}
        \frac
            {\zeta(x + \delta R)}
            {\sigma(x + \delta R)}
		\qquad \text{for $\sigma$-a.e.~$x \in \overline U$} \, .
 \end{align}
 As a corollary, using $\sigma := |\xi|$ and $\zeta:= \xi$ one obtains, for $|\xi|$-a.e.~$x \in \overline U$, the vague convergence as $\delta \to 0$ (after extracting a subsequence) of the rescaled measures
 \begin{align}
 \label{eq:def_tangent_intro}
    \xi_{\delta,x}:=
        \frac{1}{|\xi|(x + \delta R)}
    (\rho_{\delta,x})_{\#} \xi
        \quad \text{where} \quad
    \rho_{\delta,x}:\R^d \to \R^d
        \, , \quad
    \rho_{\delta,x}(y):= \frac{y-x}{\delta}
        \, ,
 \end{align}
to a tangent measure $\tau_x \in \M(\R^d; V\otimes \R^d)$ which satisfies
 \begin{align}
 \label{eq:structure_tangent_measures_intro}
    \dive \tau_x = 0
       \quad  \tand \quad
    \frac{\ddd \tau_x }{\ddd |\tau_x|}(y)
        =
    \frac{\ddd \xi }{\ddd |\xi|}(x)
        \, ,
 \end{align}
 for $|\tau_x|$-a.e.~$y \in \R^d$.
 See Lemmas~\ref{lemma:tangent_measures} and \ref{lemma:divergence_measures} below for the details.

 \smallskip
 \noindent
 \underline{Case 1} (the absolutely continuous part). \
 Applying \eqref{eq:density_intro} to $R := Q_1(0)$ and $\sigma := \Lm^d$, we deduce for $\Lm^d$-a.e. $x$ in $\overline{U}$,
 \begin{align}
 \label{eq:density_AC_intro}
    \frac
        {\ddd \nu}
        {\ddd \Lm^d}
    (x)
        =
     \lim_{\delta \to 0}
        \frac
            {\nu(Q_\delta(x))}
            {\Lm^d(Q_\delta(x))}
        \, , \qquad
     j
        :=
    \frac
        {\ddd \xi}
        {\ddd \Lm^d}
    (x)
        =
     \lim_{\delta \to 0}
        \frac
            {\xi(Q_\delta(x))}
            {\Lm^d(Q_\delta(x))}
        \, .
 \end{align}
 Note that $\nu(\partial Q_\delta(x)) = 0$ for all but countably many $\delta \in (0,1)$. In particular, up to taking a suitable subsequence in $\delta \to 0$, we have 
 \begin{align}
\label{eq:est_intro_1}
   \frac
        {\ddd \nu}
        {\ddd \Lm^d}
    (x)
        =
     \lim_{\delta \to 0}
     \lim_{\eps \to 0}
        \frac
            {\nu_\eps(Q_\delta(x))}
            {\Lm^d(Q_\delta(x))}
        =
     \lim_{\delta \to 0}
     \lim_{\eps \to 0}
       F_{\frac{\eps}{\delta}}\bigl(K_{\frac{\eps}{\delta}},Q_1(\delta^{-1} x ) \bigr)
        \, ,
 \end{align}
 where at last we simply used the definition of rescaled energy. Here, for $s=\frac\eps\delta$, the discrete vector field $K_\delta^\eps$  is obtained by rescaling $J_\eps$ as
 \begin{align}	\label{eq:def_K_s_AC_intro}
    K_\delta^\eps \in V_\rma^{\cE_s}
    \, , \quad
    K_\delta^\eps(x,y) := \delta^{1-d} J_\eps(\delta x, \delta y)
    \, , \quad
    \forall (x,y) \in \cE_s \, .
\end{align}
Note that, if $K_\delta^\eps \in \Rep_{\eps/\delta,\calR}\bigl(j,Q_1(\delta^{-1} x)\bigr)$, then we would be able to conclude that
 \begin{align}
   \frac
        {\ddd \nu}
        {\ddd \Lm^d}
    (x)
        \geq
     \lim_{\delta \to 0}
     \lim_{\eps \to 0}
       f_{\eps/\delta, \calR}\bigl(j,Q_1(\delta^{-1} x ) \bigr)
    =
        f_{\hom}(j)
        \, ,
 \end{align}
 by the very definition of $f_{\hom}(j)$.
However, one cannot generally ensure that $K_\delta^\eps$ is a competitor for the associated cell problem. Therefore, the next step is to show that we can find another sequence of vector fields $\tilde K_\delta^\eps$, suitably close to $K_\delta^\eps$, which does belong to the class of representatives and  is comparable in energy to $K_\delta^\eps$.

To get an intuition why this should be possible, we start by observing that  $K_\delta^\eps$  satisfies
\begin{align}	\label{eq:easy_comp_Ks_AC_intro}
	\big(
        \rho_{1,\frac{x}{\delta}}
    \big)_{\#}
    \big(
        \iota_{\eps/\delta} K_\delta^\eps
    \big)
        =
    \frac{1}{\delta^d} (\rho_{\delta,x})_{\#}(\iota_\eps J_\eps)
        \, ,
\end{align}
where the measure at the right-hand side converges vaguely, up to extracting a subsequence, for $\delta \to 0$ and $\eps=\eps(\delta) \to 0$ fast enough, 
to a tangent measure $\tau$ of $\xi$ in the point $x$, 
which thanks to \eqref{eq:structure_tangent_measures_intro} and \eqref{eq:density_AC_intro} is of constant density:  $\tau = j \Lm^d$.

This shows that the vector field $K_\delta^\eps$ on the translated cube $Q_1(\delta^{-1}x)$ is close, albeit in a weak sense, to the constant density measure  $\tau = j \Lm^d$, which is divergence-free, 
hence not far from being a representative for $j$ on the same cube $Q_1(\delta^{-1}x)$.
From this point on, we proceed in two corrections steps:

\smallskip
\noindent
\textit{Step 1} (boundary correction): For this purpose, 
for suitable $\eta \in (0,1)$ we consider a cutoff function $\psi_\eta$ which is zero outside the cube $Q_1(\delta^{-1}x)$, constant equal to $1$ on most of the interior of cube, and nonconstant on a cubical shell of width $\eta>0$. 
Then we consider
\begin{align}
    \tilde K_\delta^\eps
	:=
    \psi_\eta   K_\delta^\eps
	   +
	(1-\psi_\eta )
	\calR_s j
    \, ,
\end{align}
which by construction coincides with $\calR_s j$ at the boundary of $Q_1(\delta^{-1} x )=: Q^\delta$.

\smallskip
\noindent
\textit{Step 2} (divergence correction):
The second condition we need to enforce is that of being divergence free. Here we use the following existence result for correctors of the discrete divergence equation, that we prove in Proposition~\ref{prop:correctors}:
for given $m \in \M_0(\cX_s;V)$, we can find $J \in V_a^{\cE_s}$ so that $\DIVE J = m$ with $\supp(\iota_s J) \subset B_{Cs}(\conv(\supp(m)))$ and
\begin{align}
    \bigl|\iota_\eps J \bigr|(\R^d)
		\leq  C
		\left(
			\| \mu\|_{\tKR(\R^d)}
				+ \eps | \mu |(\R^d)
		\right)
    \, , \quad \text{where} \quad
    \mu := \iota_s m
        \, .
\end{align}
We apply the proposition to $m:= \DIVE \tilde K_\delta^\eps$ and find a vector field $C_s$ so that $\tilde K_\delta^\eps := \tilde K_\delta^\eps + C_s \in \Rep_{\eps/\delta,\calR}(j; Q^\delta)$.

\smallskip
\noindent
\textit{Step 3} (energy estimate):
Using the property of the corrector constructed in Step 2, and by choosing in a suitable optimal way the cutoff functions (cfr. the proof of Proposition~\ref{prop:asymptotic_cube_AC}  for a precise construction), we can show that
\begin{align}
    F_s (\tilde K_\delta^\eps, Q^\delta)
    -
		F_s(K_\delta^\eps,Q^\delta)
    \lesssim
     \frac1\eta \| \dive \iota_s K_\delta^\eps \|_{\tKR(\overline{Q^\delta})}
        +
    \frac1{\eta^2} \| \iota_s (K_\delta^\eps - \calR_s j) \|_{\tKR(\overline{Q^\delta})}
        +
    \sqrt\eta
        +
    \frac\eps\eta
        \, .
\end{align}
Note that $\iota_s K_\delta^\eps \to j \Lm^d$, and it is not hard to see that also $\dive \iota_s K_\delta^\eps \to \dive (j \Lm^d) = 0$. 
Together with \eqref{eq:est_intro_1} and the fact that $\tilde K_\delta^\eps$ is a competitor for the cell formula, we conclude that
\begin{align}
    \frac{\de \nu}{\de \Lm^d}(x)
        \geq
    f_{\hom}(j)
    - C
    \sqrt\eta
        \, .
\end{align}
Sending $\eta \to 0$, we conclude the proof.

\smallskip
\noindent
\underline{Case 2} (the singular part). \
In order to show \eqref{eq:ub_blowup_recess_intro}, we shall perform a blow-up around a singular point $x \in \R^d$. This time, we need to construct the set $R$ in a suitable geometric way, depending on the structure of the density of $\xi$. 
Let us choose $\sigma:=|\xi|$. 
From \eqref{eq:density_intro}
we obtain, for $|\xi|^s$-a.e. $x \in \R^d$,
\begin{align}
 \label{eq:density_sin_intro}
    \frac
        {\de \nu}
        {\de |\xi|}
    (x)
        =
     \lim_{\delta \to 0}
        \frac
            {\nu(R_\delta(x))}
            {|\xi|(R_\delta(x))}
        \, , \qquad
     j
        :=
    \frac
        {\de \xi}
        {\de |\xi|}
    (x)
        =
     \lim_{\delta \to 0}
        \frac
            {\xi(R_\delta(x))}
            {|\xi|(R_\delta(x))}
        \, ,
 \end{align}
 where $R_\delta(x):= x + \delta R$.
As in Case 1, we can also assume that for $|\xi|^s$-a.e. $x \in \R^d$,
 \begin{align}
\label{eq:est_intro_2}
   \frac
        {\de \nu}
        {\de |\xi|}
    (x)
        =
     \lim_{\delta \to 0}
     \lim_{\eps \to 0}
        \frac
            {\nu_\eps(R_\delta(x))}
            {|\xi|(R_\delta(x))}
        =
     \lim_{\delta \to 0}
     \lim_{\eps \to 0}
     \frac
        {\Lm^d(R_\delta(x))}
        {|\xi|(R_\delta(x))}
     \frac{F_{\eps/\delta}(K_\delta^\eps,\delta^{-1} R_\delta )}{\Lm^d(R)}
        \, ,
 \end{align}
 where $K_\delta^\eps$  is obtained by rescaling $J_\eps$ as in \eqref{eq:def_K_s_AC_intro}.

Moreover, we consider a tangent measure $\tau$ of $\xi$ around a singular point $x$: for $|\xi|^s$-a.e. $x \in \R^d$, the rescaled measure $\xi_{\delta,x}$ (cfr. \eqref{eq:def_tangent_intro}) converges (up to subsequence) to a measure $\tau$ that satisfies \eqref{eq:structure_tangent_measures_intro}. Around singular points, a tangent measure $\tau$ is not necessarily of the form $j \Lm^d$.  Nonetheless, we show in Proposition~\ref{prop:structure_tangent_measures} that there exists a measure $\kappa \in \M_+(\ker j)$ such that
\begin{align}
\label{eq:tau_intro}
    \tau = j \lambda \otimes \kappa
        \, ,
\end{align}
with respect to the decomposition $\R^d = (\ker j )^\perp \oplus \ker j$, and $\lambda \in \M_+((\ker j)^\perp)$ is the Hausdorff measure (of dimension $\dim((\ker j)^\perp)$).

Additionally, another consequence of Besicovitch's differentiation theorem ensures that $|\xi|^s$-a.e. $x \in \R^d$
\begin{align}
    \lim_{\delta \to 0}
        t_\delta
    = + \infty
        \, , \qquad
    \text{where} \quad
      t_\delta:=
     \frac
        {|\xi|(R_\delta(x))}
        {\Lm^d(R_\delta(x))}
        \, .
\end{align}
In particular, if we had that $K_\delta^\eps \in \Rep_{\eps/\delta,\calR}(t_\delta j,\delta^{-1} R_\delta)$, then we would conclude that
 \begin{align}
   \frac
        {\de \nu}
        {\de |\xi|}
    (x)
        \geq
     \lim_{\delta \to 0}
     \lim_{\eps \to 0}
     \frac1{t_\delta}
       \frac
       {f_{\frac\eps\delta, \calR}(t_\delta j,\delta^{-1} R_\delta )}{\Lm^d(R)}
    =
    \lim_{\delta \to 0}
    \frac1{t_\delta}
    f_{\hom}(t_\delta j)
    =
        f_{\hom}^\infty(j)
        \, ,
 \end{align}
 by the very definition of $f_{\hom}^\infty(j)$.

Once again,  one cannot generally ensure that $K_\delta^\eps$ is a competitor for the associated cell problem.
We shall proceed in a similar spirit as in the absolutely continuous part, and seek vector fields $\tilde K_\delta^\eps$, suitably close to $K_\delta^\eps$, which do belong to the class of representatives and whose energy is comparable to the one of $K_\delta^\eps$.
Recall \eqref{eq:easy_comp_Ks_AC_intro}. In particular, one has that, for suitable choices of $\delta, \eps=\eps(\delta) \to 0$,
\begin{align}
\label{eq:hyp_final_1_SIN_intro}
    \frac1{\Lm^d(R) }
    (\rho_{1,\frac{x_0}{\delta}})_{\#}
    \bigg(
        \frac{\iota_{\eps/\delta} K_\delta^\eps}{t_\delta}
    \bigg)
        \to  \tau
            \quad \text{vaguely in }
                \M(\R^d ; V \otimes \R^d)
        \, ,
\end{align}
where $\tau$ is of the form \eqref{eq:tau_intro}. Note that generally $\tau$ does not coincide with the Lebesgue measure on $\ker j$, hence $K_\delta^\eps \simeq t_\delta j$ (in a weak sense) only in the direction of  $(\ker j)^\perp$. This is why the choice of the set $R$ is crucial around singular points: in particular, we pick $R$ being a rectangle with sides parallel to $(\ker j)^\perp$ os size $1$ and sides parallel to $\ker j$ of size $\alpha \ll 1$, where indeed we have no information on the shape of $\tau$.  We then proceed as before: we first fix the right boundary conditions, and then find suitable correctors to obtain a divergence free field which is a competitor on $\delta^{-1} R_\delta$.
Thank to the specific choice of the set $R$ we show we can perform these corrections paying an error in the energy of the form, as $\delta \to 0$,
\begin{align}
    \err_{\eta,\alpha}
        \lesssim
        C_{\alpha,\eta}  \,  o_\eps(1) +
    \eps + \sqrt{\eta} + \sqrt{\alpha} + \frac{\sqrt{\alpha}}{\eta}
        \, ,
\end{align}
for some $o_\eps(1) \to 0$ as $\delta, \eps=\eps(\delta) \to 0$. This is the content of Proposition~\ref{prop:asymptotic_strip}, one of the most important and complex proofs of this work.

Sending first $\delta \to 0$, then $\alpha \to 0$, and finally $\eta \to 0$, we are able to conclude.

\section{Existence of uniform flows}
\label{sec:existence}
With assumptions (G1) and (G2) we can show the existence of a linear embedding $\calR \in \mathrm{Lin}(V\otimes \R^d; V_\rma^{\cE})$ with the following properties.

\begin{defi}[Uniform-flow operator]
	\label{def:admissible}
	A \emph{uniform-flow operator} for a graph $(\cX,\cE)$ embedded in $\R^d$ is a bounded linear operator $\calR \in \mathrm{Lin}(V\otimes \R^d; V_\rma^\cE)$ so that
\begin{enumerate}
	\item (divergence free) $\DIVE \calR  j = 0$ for all $j\in V\otimes \R^d$.
	\item (convergence) the rescaling $\calR_\eps\in \mathrm{Lin}(V\otimes \R^d;V_\rma^{\cE_\eps})$ defined by $\calR_\eps j(\eps x, \eps y) := \eps^{d-1} \calR  j(x,y)$ is so that
\begin{align}
\label{eq:J0 properties_1}
	\iota_\eps \calR_\eps j \to j\Lm^d \quad \text{ vaguely as } \eps \to 0
        \, , \qquad
    \forall j \in V \otimes \R^d
        \, .
\end{align}
	\item (boundedness) there is a constant $C>0$ such that, for every $\eps >0$,
	\begin{align}
\label{eq:J0 properties_2}
	|\iota_\eps \calR_\eps j|(Q)\leq C |j| \Lm^d(Q)
\end{align}
whenever $Q$ is an orthotope containing a cube of side-length $\eps$.
\end{enumerate}
\end{defi}
\begin{rem}[The cartesian grid $(\cX, \cE) = (\Z^d, \E^d)$]
	In the simplest case
	\begin{align*}
		\cX=\Z^d
			\tand
		\cE = \E^d := \{(z,z')\in \Z^d\times \Z^d\,:\,|z-z'|=1\}\,,
	\end{align*}
	one easily checks that a uniform-flow operator is given by  
		$\calR  j(z,z') 
		:= j(z'-z)\in V$.
	To show the existence of a uniform-flow operator on more general graphs,
	the idea is to construct a grid of paths that behaves like the cartesian grid $(\Z^d,\E^d)$.
	Proving the desired properties (in particular, the convergence to the constant density measure) is more involved
	due to the possibly nontrivial geometry of the graph.
\end{rem}

\begin{rem}[Orthotopes]
\label{rem:vague_J0_cubes}
	If $\calR$ is a uniform-flow operator,
    then $\iota_\eps \calR_\eps j(Q) \to j\Lm^d(Q)$ for
	every orthotope $Q \subset \R^d$,
	as a consequence of (2) and Portmanteau theorem.
	Moreover,
	for
	every set $A \in \calB(\R^d)$
	that is a countable, disjoint union of orthotopes, each containing a cube of side-length $\eps > 0$,
	we have
    \begin{align}
        |\iota_\eps \calR_\eps j|(A)\leq C |j| \Lm^d(A)
            \, ,
    \end{align}
	by $\sigma$-additivity of the measure
	$|\iota_\eps \calR_\eps j|$,
	where $C$ is the constant in \eqref{eq:J0 properties_2}.
\end{rem}

The main result of this section is the following existence result.
\begin{prop}\label{prop:J0}
	Every graph $(\cX,\cE)$ embedded in $\Z^d$ satisfying (G1) and (G2) admits a uniform-flow operator $\calR \in \mathrm{Lin}(V\otimes \R^d; V_\rma^\cE)$.
\end{prop}

\begin{proof}
For every $z \in \Z^d$, we use the assumption (G1) to pick a nearby vertex $x_z \in \cX$
with
$|x_z - z| \leq R_1$.
For all $z,z'\in \Z^d$ with $|z-z'| = 1$,
note that
	$|x_z - x_{z'}| \leq 2 R_1 + 1$
by the triangle inequality.
Hence, using assumption (G2),
we may choose, for all $z \in \Z^d$ and $i \in \{1, \ldots, d\}$,
a simple path
	$P_{z,i}
	\in \fP$
	of Euclidean length at most
	$\ell := R_2 ( 2 R_1 + 1) + 1$
	connecting $x_z$ and $x_{z + e_i}$ in
	the graph $(\cX,\cE)$.
Recalling
	Definition~\ref{def:vectoralongpath}
we define
	$\calR  \in
		\Lin(V\otimes \R^d; V_\rma^\cE)$
by
\begin{align}
\label{eq:def_calR_eps}
	\calR j:=
			\sum_{z\in \Z^d}
			\sum_{i=1}^d
				J_{P_{z,i}} j e_i
			\in  V_\rma^\cE
		 \qquad  \text{for }  j \in V \otimes \R^d
	\, ,
\end{align}
see Figure~\ref{fig:unit_flows}.
\begin{figure}[h]
\includegraphics{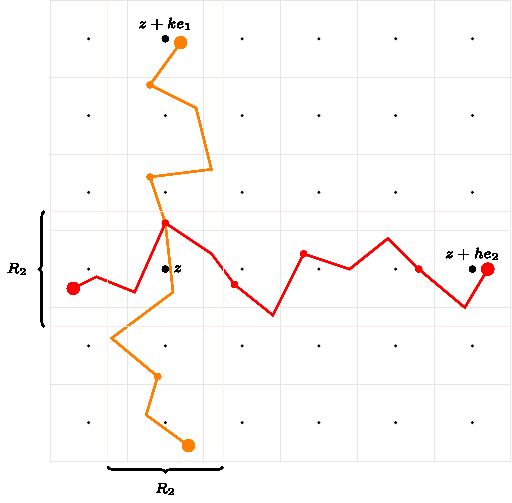}
\caption{In red (resp. in orange), a concatenation of paths on $(\cX,\cE)$ of the form $P_{z,1}$ (resp. $P_{z,2}$) with $z \in \Z^2$. Along the red paths $\calR j = j e_2$, whilst along the orange path $\calR j = j e_1$.}
\label{fig:unit_flows}
\end{figure}
For each $(x,y)\in \cE$,
only finitely many summands in the definition of
$\bigl(\calR j\bigr)(x,y)$ are nonzero, so that
$\calR j$ is well-defined. It is then clear that
$\calR\in \mathrm{Lin}(V\otimes \R^d;V_\rma^\cE)$.

It remains to check that $\calR$ has the desired properties.

\smallskip
\noindent
\textit{(1) \ Divergence free.} \
For $i =1, \dots, d$, Lemma~\ref{lem:J_Pq} yields
\begin{align*}
	\DIVE
	\bigg(
		\sum_{z\in \Z^d} J_{P_{z, i}}
	\bigg)
		=
	\sum_{z \in \Z^d}
	\big(
		\one_{\{x_z\}} - \one_{\{x_{z+e_i}\}}
	\big)
		= 0
	 \, ,
\end{align*}
which implies that $\calR j$ is divergence free

\smallskip
\noindent
\textit{(3) \ Boundedness.} \
We show the statement for $\eps = 1$. The corresponding estimate for $\eps>0$ follows using the same argument together with the scaling properties \eqref{eq:J-scaling}.

Let us first assume that $Q$ is a unit cube.
Using the definitions, it follows that for all $z \in \Z^d$ and all $i$,
	$|\iota_1 J_{P_{z, i}}|(Q) \leq
	\length (P_{z, i} \cap Q)
	$,
where we slighly abuse notation by viewing
the path $P_{z, z + e_i}$ as a union of line segments in $\R^d$.
Since
	$\length (P_{z, i}) \leq \ell$
and  each unit cube $Q$ intersects at most
	$K := K(d,R_1, R_2)$ of the paths $P_{z, i}$,
it follows that
\begin{align}
	|\iota_1 \calR j|(Q)
		&\leq
	\| j \|
		\sum_{ z \in \Z^d}
		\sum_{i=1}^d
		\text{length}
		\left(
			P_{z, i} \cap Q
		\right)
		\leq
	\| j \| K \ell
		\, .
\end{align}

In general, every orthotope $Q$ containing a unit cube can be covered by $C \Lm^d(Q)$-many unit cubes, where $C > 0$ depends only on $d$.
Therefore the claimed inequality \eqref{eq:J0 properties_2} follows by subadditivity of $|\iota_1 \calR j|$.

\smallskip
\noindent
\textit{(2) Convergence.} \
Fix $j \in V \otimes \R^d$.
For every $\varphi \in C_c^1(\R^d)$ we will show the convergence
\begin{align}
\label{eq:test_C1}
    \lim_{\eps \to 0}
    \int_{\R^d} \varphi \dd (\iota_\eps \calR_\eps j)
        =
    \int_{\R^d} \varphi j \dd \Lm^d  \, .
\end{align}
The claimed vague convergence then follows from this together with the boundedness proved above, via a classical compactness argument.
Namely, boundedness implies that $\sup_\eps |\iota_\eps \calR_\eps j|(K)<\infty$ for every compact set $K\subset \R^d$, which implies that, up to a non-relabeled subsequence, $\iota_\eps \calR_\eps j \to \mu \in \M(\R^d; V \otimes \R^d)$ vaguely as $\eps \to 0$. Due to the fact that $C_c^1(\R^d)$ is dense in $C_c(\R^d)$ in the uniform topology, from \eqref{eq:test_C1} we infer that $\mu = j \Lm^d$, and therefore the claimed convergence.

Fix $\varphi \in C_c^1(\R^d)$.
Without loss of generality,
we assume that $j = v \otimes e_d$ for some $v \in V$,
as the general statement follows by linearity.
For this particular $j$, the vector field
$\calR_\eps j$ naturally splits
as a sum of vector fields induced by paths $\eps P_{z'}$
for $z' \in \Z^{d-1}$,
where $P_{z'}$ is an approximately straight path
passing through $(z',0) \in \Z^d$
in the approximate direction $e_d$.
Precisely,
\begin{align}
	\label{eq:formula_J0eps_paths}
		\calR_\eps j
			=
		\eps^{d-1}
		\biggl(\sum_{z' \in \Z^{d-1}}
		J_{\eps P_{z'}
		}
		\biggr)
		v
			\, ,
	\end{align}
where
$P_{z'}$ is the concatenation of the paths
	$\{ P_{(z',h), d}
		\suchthat h \in \Z\}$ (in Figure~\ref{fig:unit_flows} corresponding to the red paths),
so that
\begin{align}
	J_{\eps P_{z'}}
		=
	\sum_{h \in \Z}
		J_{\eps P_{(z',h), d}}
		\, .
\end{align}

Fix a mesoscopic length scale $\eps \ll \delta \ll 1$
such that $\delta/\eps \in \N$.
We then decompose each path $P_{z'}$ in pieces of mesoscopic size, namely
\begin{align*}
	J_{\eps P_{z'}}
		=
	\sum_{m \in \Z}
		J_{\eps P_{z'}^m}
	\quad
		\text{where}
	\quad
		J_{\eps P_{z'}^m}
		=
	 \sum_{h = m {\delta/\eps}}
	 		^{ (m+1) {\delta/\eps}-1}
		J_{\eps P_{(z',h), d}} \, .
\end{align*}
Note that $\eps P_{z'}^m$ connects
	$\eps x_{z_m'}$ with $\eps x_{z_{m+1}'}$,
	where $z_m' := (z', m\delta/\eps)$, see Figure~\ref{fig:mesoscopic}.
\begin{figure}[h]
    \centering

\includegraphics[scale=1.1]{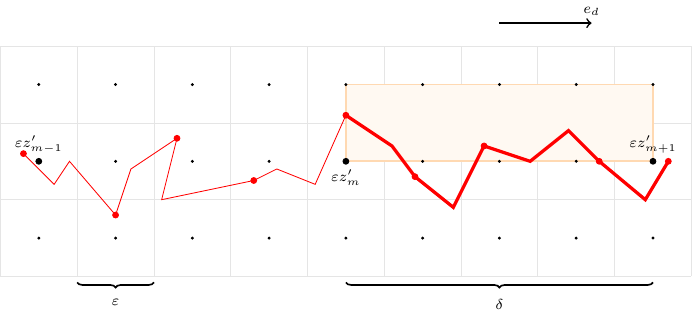}
    \caption{A two-scale decomposition of the paths in the direction $e_d$. Highlighted the subcomponent given by $J_{\eps P_{z'}^m}$. In light orange, a representation of the set $R_{z_m'}$. On the mesoscopic scale $\delta \gg \eps$, the path $J_{\eps P_{z'}^m}$ is approximately the segment joining $z_m'$ and $z_{m+1}'$, hence $\iota_\eps J_{\eps P_{z'}^m}$ has an approximate average orientation given by $e_d$, see \eqref{eq:convergence_est1}.}
    \label{fig:mesoscopic}
\end{figure}
    
Fix $z'\in \eps \Z^{d-1}$. We claim that
\begin{align}
\label{eq:claim1_convergence}
	\bigg|
	\int \varphi \dd \iota_\eps (J_{\eps P_{z'}^m})
	-
	\varphi(\eps z_m')
		\delta e_d
\bigg|
\leq
C (\eps + \delta^2)
\end{align}
for some $C > 0$ depending on $\varphi$, $R_1$, and $R_2$, which may change from line to line below.
The proof of this statement crucially uses that $\varphi$ belongs to $C_c^1$, and not merely to $C_c$.
Indeed, noticing that
	$\int \dd  \iota_\eps (J_{\eps P_{z,i}})
			= \eps (x_{z + e_i} - x_z)$,
we have the telescopic series
\begin{align*}
	\int \dd \iota_\eps (J_{\eps P_{z'}^m})
	=
		\sum_{h = m {\delta/\eps}}
				^{ (m+1) {\delta/\eps}-1}
				\eps \bigl(x_{(z', h+1)} - x_{(z', h)}\bigr)
	= \eps \bigl( x_{z_{m+1}'} - x_{z_m'} \bigr) \, .
\end{align*}
Using the triangle inequality we find
\begin{align*}
	\Bigl| x_{z_{m+1}'} - x_{z_m'} - \frac{\delta}{\eps} e_d \Bigr|
	& = \bigl| x_{z_{m+1}'} - z_{m+1}' +  z_m' - x_{z_m'}  \bigr|
	\\& \leq \bigl | x_{z_{m+1}'} - z_{m+1}'  \bigr| + \bigl|  z_m' - x_{z_m'}  \bigr|
	\leq 2( 2R_1 + 1)\,,
\end{align*}
hence
\begin{align}
	\label{eq:convergence_est1}
	\bigg|
		\int \dd \iota_\eps (J_{\eps P_{z'}^m})
	  -
	  	\delta e_d
	\bigg|
	\leq  2( 2R_1 + 1) \eps \, .
\end{align}
Moreover, for $x$ in the support of $J_{\eps P_{z'}^m}$,
$(G1)$ and $(G2)$ yield
	$|\varphi(x) - \varphi(\eps z_m')|
	\leq C \| \nabla \varphi\|_\infty \delta
	\leq C \delta$.
Consequently,
since $\length(\eps P_{z'}^m) \leq C \delta$,
\begin{align}
	\label{eq:convergence_est2}
	\bigg|
        \int \bigl( \varphi - \varphi( \eps z_m') \bigr) \dd \iota_\eps ( J_{\eps P_{z'}^m})
    \bigg|
    &\leq
	C
		\delta
	\length(\eps P_{z'}^m)
	\leq
	C
		\delta^2\, .
\end{align}
Combining
\eqref{eq:convergence_est1}
and
\eqref{eq:convergence_est2}
yields the claimed bound
\eqref{eq:claim1_convergence}.

Now, consider the orthotope (cfr. Figure~\ref{fig:mesoscopic}) 
\begin{align*}
	R_{z_m'}
		=
	(\eps z_m', 0)
	+
	[0,\eps)^{d-1} \times [0,\delta)\;,
\end{align*}
and observe that, since $\varphi \in C_c^1(\R^d)$,
\begin{align*}
	\bigg|
	\varphi( \eps z_m')
	-
	\frac{1}{ \eps^{d-1} \delta}
		\int_{R_{z_m'}} \varphi
	\bigg|
	\leq
	C \delta \, .
\end{align*}
Combined with \eqref{eq:claim1_convergence} we find
\begin{align}
		\label{eq:zmbound}
	\bigg|
	\int \varphi \dd \iota_\eps (J_{\eps P_{z'}^m})
	-
	\frac{1}{ \eps^{d-1}}
		\biggl(\int_{R_{z_m'}} \varphi\biggr)
		e_d
\bigg|
\leq
C (\eps + \delta^2) \, .
\end{align}
Since $\varphi$ is compactly supported, we notice using assumptions $(G1)$ and $(G2)$ that both integrals vanish, except for $(z', m) \in \calS_{\eps,\delta} \subseteq \Z^d$,
where $\calS_{\eps,\delta}$ is a set whose cardinality $\# \calS_{\eps,\delta}$ can be bounded by $\frac{C}{\eps^{d-1} \delta}$.
Therefore, summation over $(z', m) \in \calS_{\eps,\delta}$ yields, using \eqref{eq:zmbound},
\begin{align*}
	\int \varphi \dd \iota_\eps (\calR_\eps j)
	&		=
		\eps^{d-1}
		\sum_{
			(z', m) \in
			\calS_{\eps,\delta} }
			v
			\otimes
		\biggl(
			\int \varphi
		\dd
		\iota (J_{\eps P_{z'}^m})
		\biggr)
		\\& =
		\eps^{d-1}
		\sum_{
			(z', m) \in
			\calS_{\eps,\delta} }
			\bigg[
			\frac{1}{ \eps^{d-1}}
			\biggl(\int_{R_{z_m'}} \varphi\biggr)
			v
			\otimes
			e_d
			+ O(\eps + \delta^2)
			\bigg]
		\\& =
			\biggl(
				\int_{\R^d} \varphi
			\biggr)
			v
			\otimes
			e_d
			+
			\eps^{d-1}
			\bigl(\# \calS_{\eps, \delta}\bigr)
			O\Bigl( \eps + \delta^2\Bigr)
		\\& =
		\biggl(
				\int_{\R^d} \varphi
		\biggr)
			j
			+
			O\Bigl( \frac{\eps}{\delta} + \delta\Bigr) \, .
\end{align*}
Taking for instance $\delta = \eps \big\lfloor \frac{1}{\sqrt{\eps}}\big\rfloor$ (so that $\delta/\eps \in \N$),
we obtain the desired convergence \eqref{eq:test_C1}.
\end{proof}

\begin{rem}[Absolute continuity of limit points for $\calR_\eps j$]
\label{rem:abs_cont_Reps}
    The boundedness property of $\calR_\eps$ can be used to show that, for every $j \in V \otimes \R^d$, every accumulation point $\lambda \in \M_+(\R^d)$ of the sequence of measures $\{ |\iota_\eps \calR_\eps j| \}_\eps$ must necessarily be absolutely continuous with respect to the Lebesgue measure with bounded density. Indeed, assume that (up to subsequence) $|\iota_\eps \calR_\eps| \to \zeta$ vaguely. Let $D \subset \R^d$ be a closed bounded set, and define the set
    \begin{align}
        Q(D ; \eps)
            :=
        \bigcup
        \big\{
            Q_z = z + [0,\eps)^d
                \suchthat
            z \in \eps \Z^d
                \, , \quad
            Q_z \cap D \neq \emptyset
        \big\}
            \, ,
    \end{align}
    for which we have the trivial inclusion $D \subset Q(D;\eps)$, for every $\eps>0$. Denote by $Q^\circ(D,\eps)$ the interior of $Q(D,\eps)$. Note that $Q(D,\eps)$ is a countable, disjoint union of cubes of size $\eps$. Therefore, by Remark~\ref{rem:vague_J0_cubes}, boundedness of $\calR_\eps$ and the vague convergence to $\lambda$ ensure, for every $\eps_0 >0$,
    \begin{align}
        \zeta(D)
            \leq
        \zeta( Q^\circ(D;\eps_0) )
            &\leq
        \liminf_{\eps \to 0}
            |\iota_\eps \calR_\eps|(Q^\circ(D,\eps_0))
    \\
            &\leq
        \liminf_{\eps \to 0}
            |\iota_\eps \calR_\eps|
                (Q(D,\eps_0))
            \leq
        C  \| j \|  \Lm^d(Q(D,\eps_0))
            \, .
    \end{align}
    Taking the limit of as $\eps_0 \to 0$, we obtain that
    \begin{align}
        \zeta(D) \leq C \| j \|  \Lm^d(D)
            \, , \quad
        \forall D \subset \R^d \, \text{ closed}
            \, .
    \end{align}
    By inner regularity of $\zeta$ and $\Lm^d$, we then conclude that the previous inequality holds for all Borel sets $D \subset \R^d$, hence the claimed absolute continuity.
\end{rem}

\section{The multi-cell formula and homogenized limit}
	\label{sec:multi-cell}

In this section we define the homogenized energy density
	$f_{\hom} :V\otimes \R^d \to \R$
and study some of its main properties.

Recall that $R_\Lip$ denotes the radius of nonlocality in assumption (F1), and $R_3$ is the maximal edge-length in assumption (G3).

\begin{defi}[Representative]
	\label{def:rep}
	 Let $\eps \in (0,1]$. 
	 Let $\calR$ be a uniform-flow operator.
	A discrete vector field $J \in V_\rma^{\cE_\eps}$ is said to be an
	$(\eps, \calR)$-representative of
	a tensor
	$j \in V \otimes \R^d$ in a Borel set $A \in \calB(\R^d)$ if
	\begin{enumerate}
	\item[(i)] $\DIVE J = 0$;
	\item[(ii)]
		$J(x,y) = \calR_\eps j(x,y)$
		for all $(x,y) \in \cE_\eps$ with $\dist\bigl([x,y],\R^d\setminus A\bigr) \leq \eps R_\partial$,
	\end{enumerate}
 where  $R_\partial := \max \{ R_\Lip, R_3 \}$. The set of all $\eps$-representatives of $j$ in $A$  will be denoted by $\Rep_{\eps,\calR}(j;A) $. We use the notation $\Rep_\calR := \Rep_{1,\calR}$.
\end{defi}

\begin{figure}[h]
    \centering
\includegraphics{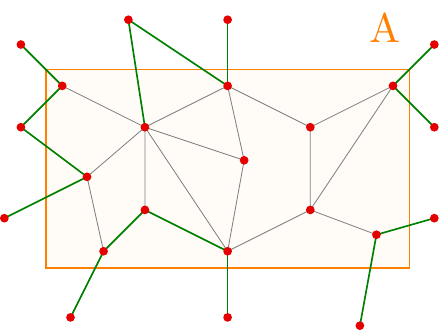}
    \caption{The set of $\calR$-representatives of $j$ in $A$ must coincide with $\calR j$ on all the edges highlighted in green, which at distance of order $1$ from $\partial A$. For $(\eps,\calR)$ representatives, the boundary conditions must be satisfied for edges at distance of order $\eps$ from $\partial A$.}
    \label{fig:cellformula}
\end{figure}

In other words, the set $\Rep_{\eps,\calR}(j,A)$ contains all the divergence-free discrete vector fields which coincide with the uniform flow $\calR_\eps j$ at distance ~$\eps$ from the boundary $\partial A$ of $A$, see Figure~\ref{fig:cellformula}. 
Note that $\Rep_{\eps,\calR}(j;A)$ is non-empty for all $(j,A)$,
as this set contains the canonical representative $\calR_\eps j$. 
The need of imposing Dirichlet-type boundary conditions up to $\eps R_\partial$  comes from the mild nonlocality of the graph (hence the dependence on $R_3$) and the energy (hence the dependence on $R_{\Lip}$). In the limit as $\eps \to 0$, this formally becomes a pure Dirichlet boundary condition on the boundary of the set $A$, which is coherent with the cell-formula typically appearing in stochastic homogenisation results for continuous energies. For a similar representation formula, see e.g. \cite{Alicandro-Cicalese-Gloria:2011} (in the setting of curl-free minimisation problems).

The following simple scaling lemma follows from the definition of $\iota_\eps$ and $\calR_\eps$.

\begin{lemma}[Scaling]
	\label{lem:j-and-J}
	Let $\eps \in (0,1]$ and let $\calR$ be a uniform-flow operator.
	For any $j \in V \otimes \R^d$ and $A \in \calB(\R^d)$
	the following properties hold:
		\begin{align}\label{eq:J-scaling}
			& \big(\iota_\eps \calR_\eps j \big)( \eps A)
				=
			\eps^d
			\big(\iota_1 \calR j\big)( A)
							\, ,
		\\
			& J \in \Rep_{\calR}(j,A) \quad{ \text{iff} } \quad
			\eps^{d-1} J(\cdot /\eps) \in \Rep_{\eps, \calR}(j, \eps A)  \, .
		\end{align}
\end{lemma}

While $\Rep_{\eps,\calR}(j;A)$ may contain many elements,
the next result shows that all of them
assign the same value to $A$.

\begin{lemma}
	\label{lem:equal-mass}
	Let $\eps \in (0,1]$ and let $\calR$ be a uniform-flow operator.
	Let $j \in V \otimes \R^d$ and let $A \in \cB(\R^d)$.
	For all $J \in \Rep_{\eps,\calR}(j,A)$ we have
	\begin{align}
		\label{eq:equal-int}
		\iota_\eps J (A)
		= \iota_\eps \calR_\eps j (A)
		\, .
	\end{align}
\end{lemma}

		The reason why this result holds is that
		the integral of any
		divergence-free vector field over a domain is completely determined by its boundary values.
		In the continuous setting,
		this is an elementary consequence of Stokes' theorem: Indeed, let
		$j$
		be a smooth divergence-free vector field
		on a bounded domain $\Omega \subseteq \R^d$.
		Let $y \in \R^d$ be arbitrary and consider the linear map $\ell_y : \Omega \to \R$ with slope $y$, i.e.,
		$\ell_y(x) := \ip{x, y}$ for $x \in \Omega$.
		An application of Stokes' theorem yields,
		\begin{align}
			\Bip{
				\int_\Omega j \dd x
					,
				y
			}
				=
			\int_\Omega
			\ip{
				j
					,
				\nabla \ell_y
			}
				\dd x
				&=
			\int_{\partial \Omega}
				\ip {j, \ext }
				\ell_y
				\dd \Hm^{d-1}
			- \int_\Omega (\nabla \cdot j) \ell_y \dd x
                \, , 
		\end{align}
		where $\ext$ denotes the outward unit normal on $\partial \Omega$.
		Since $\nabla \cdot j \equiv 0$ by assumption, the latter expression is fully determined by the boundary values of $j$, hence the same holds for $\int_\Omega j \dd x$.
	The following proof contains a discrete version of this argument.

	\begin{proof}[Proof of Lemma \ref{lem:equal-mass}]
		Without loss of generality, we can assume that $\eps=1$.
		For $J \in \Rep_\calR(j,A)$
		it follows from the definitions that
	\begin{align}
		\iota_1 J (A)
		= \frac12
		\sum_{(x,y) \in \cE^A}
		\lambda_{xy}
		J(x,y) \otimes (y-x)
			\, ,
	\end{align}
	where $\lambda_{xy} := \Hm^1([x,y] \cap A) / | y - x|$.
	Note first that, by anti-symmetry of $\lambda_{xy}
	J(x,y)$,
	\begin{align*}
		- \frac12 \sum_{(x,y) \in \cE^A}
			\lambda_{xy}
			J(x,y) \otimes (y-x)
				=
		\sum_{(x,y) \in \cE^A}
				\lambda_{xy}
				J(x,y) \otimes x \, .
	\end{align*}
	We will distinguish two cases in the sum on the right-hand side,
	depending on whether or not
		$x$ belongs to
		the interior of $\cX$ in $A$,
		which we define by
	\begin{align*}
		(\cX^A)^\circ :=
		\left\{
			x \in \cX \cap A
			\suchthat
			[x, y] \subseteq A
			\text{ for every }y \sim x
			\right\}.
		\end{align*}

		First, if $ x \in (\cX^A)^\circ$ and
	$(x,y) \in \cE^A$,
	we note that $[x,y] \subseteq A$, hence $\lambda_{xy} = 1$.

		Second, if $x \in \cX \setminus (\cX^A)^\circ$ and $(x,y) \in \cE^A$,
	then $x$ has a neighbour $z \sim x$ such that $[x,z] \not\subseteq A$. Consequently,
	$\dist([x,y], A^c)
	\leq \dist(x, A^c)
	\leq |x-z|
		\leq R_3 \leq R_\partial$.
	Since $J \in \Rep_\calR (j,A)$,
	this implies that $J(x,y) = \calR j(x,y)$.

	Splitting the sum above, we obtain
	\begin{align*}
		& - \frac12
		\sum_{(x,y) \in \cE^A}
			\lambda_{xy}
			J(x,y) \otimes (y-x)
		\\&
		=
		\sum_{\substack{
			x \in (\cX^A)^\circ ,\; y \in \cX
			\suchthat
			\\
		(x,y) \in \cE^A   }}
				\lambda_{xy}
				J(x,y) \otimes x
		+
		\sum_{\substack{
			x \in \cX \setminus (\cX^A)^\circ ,\; y \in \cX  \suchthat
			\\
		(x,y) \in \cE^A   }
		}
				\lambda_{xy}
				J(x,y) \otimes x
		\\ &		=
				\sum_{ x \in (\cX^A)^\circ }
						\DIVE J(x) \otimes x
				+
				\sum_{\substack{
					x \in \cX \setminus (\cX^A)^\circ ,\;  y \in \cX  \suchthat
			\\
		(x,y) \in \cE^A
				}}
						\lambda_{xy}
						\calR j(x,y) \otimes x
				\, .
	\end{align*}
	Since $\DIVE J(x) = 0$ for all $x \in \cX$,
	the latter expression does not depend on the particular representative $J \in \Rep_\calR(j,A)$.
	This yields the result.
\end{proof}

The multi-cell formula then reads as follows.

\begin{defi}
	Let $\eps \in (0,1]$ and let $\calR$ be a uniform-flow operator.
	The $\eps$-rescaled cell-problem functional
	$f_{\omega,\eps,\calR} :V \otimes \R^d \times \calB(\R^d) \to \R$
	is defined by
	\begin{align}
		f_{\omega,\eps,\calR}(j,A) :=
		 \inf
		\left\{
			F_{\omega,\eps}(J,A)
		\suchthat
			J \in \Rep_{\eps,\calR}(j;A)
		\right\} \, .
	\end{align}
	We use the notation $f_{\omega,\calR} := f_{\omega,1,\calR}$.
\end{defi}

Note that $f_{\omega,\eps,\calR}(j,A) < \infty$ for all $(j,A)$ as above, since the canonical representative $\calR_\eps j$ is a competitor, as observed after Definition \ref{def:rep}.

\begin{lemma}[Scaling]
	\label{lem:scaling-again}
	Let $\eps \in (0,1]$ and let $\calR$ be a uniform-flow operator.
	For any $j \in V \otimes \R^d$ and $A \in \calB(\R^d)$
	we have
	\begin{align}
		\label{eq:scaling_prop_fomega}
		f_{\omega,\eps,\calR}(j,A)
        =
		\eps^d
		f_{\omega,\calR}\Bigl(j,\frac{A}{\eps}\Bigr)\, .
	\end{align}
\end{lemma}

\begin{proof}
		Using the definition of the rescaled energy $F_\eps$ from \eqref{eq:rescaled-energy}
		and Lemma \ref{lem:j-and-J},
		we obtain
	\begin{align*}
		f_{\omega,\eps,\calR}(j,A)
	& =
			\inf
			\big\{
				F_{\omega,\eps}(J,A)
			\suchthat
				J \in \Rep_{\eps,\calR}(j;A)
			\big\}
	\\&	=
			\inf
			\bigg\{
				\eps^d
				F_\omega \bigg(
							 \frac{J(\eps \cdot)}{\eps^{d-1}},
						  	 \frac{A}{\eps} \bigg)
			\suchthat
				\frac{J(\eps \cdot)}{\eps^{d-1}} \in
						\Rep_\calR\bigg(  j;\frac{A}{\eps}\bigg)
			\bigg\}
	\\&	=
			\inf
			\bigg\{
				\eps^d
				F_\omega \Big(
							 J,
						  	 \frac{A}{\eps} \Big)
			\suchthat
				J \in
						\Rep_\calR\bigg( j;\frac{A}{\eps}\bigg)
			\bigg\}
	= \eps^d f_{\omega,\calR}\bigg( j, \frac{A}{\eps}\bigg)\,,
	\end{align*}
	as desired.
\end{proof}

\begin{lemma}[Properties of $f_{\omega,\eps,\calR}$]\label{lemma: F_omega}
	There exist constants $c,C>0$ such that the following assertions hold:
	\begin{enumerate}
		\item [(i)]
		\label{it:linear-growth}
		(linear growth) \
		For any bounded set
			$A \in \calB(\R^d)$
		with $\Lm^d(\partial A) = 0$
		we have
		\begin{align}
                \qquad
                c|j| \Lm^d(A)
                    \leq
                \liminf_{\eps \to 0}
				f_{\omega,\eps,\calR}(j,A)
			\leq
			\limsup_{\eps \to 0}
				f_{\omega,\eps,\calR}(j,A)
			\leq C \Lm^d(A) \bigl(|j|+1\bigr)
		\end{align}
		for all $j \in V \otimes \R^d$.
		\item [(ii)]
        (Lipschitz property) \
		For any bounded set $A\in \calB(\R^d)$
		we have
		\begin{align*}
			|f_{\omega,\eps,\calR}(j,A) - f_{\omega,\eps,\calR}(j',A)|
				\leq
			C\Lm^d
				\big( B(A,\eps \tilde R) \big)
				|j-j'|
		\end{align*}
		for all $j, j' \in V \otimes \R^d$, where $\tilde R := R_\Lip + \sqrt{d}$.
		\item [(iii)] (subadditivity) \
		For any pairwise disjoint collection of Borel sets $\{A_1, \ldots, A_N\}$ we have \begin{align}
			A = \bigcup_{i =1}^N A_i
				\quad \Longrightarrow \quad
			f_{\omega,\calR}(j,A) \leq \sum_{i=1}^N f_{\omega,\calR}(j,A_i) \, .
		\end{align}
	\end{enumerate}
\end{lemma}

\begin{proof}
	\emph{(i)}: \
To prove the lower bound, we fix $j \in V \otimes \R^d$ and let $A \subseteq \R^d$ be relatively compact with
	$\Lm^d(\partial A) = 0$.
For any $J \in \Rep_\calR(j,A)$
we obtain
using Assumption (F2),
\eqref{eq:total_var_iota},
and Lemma \ref{lem:equal-mass},
\begin{align}
\label{eq:lower_bound_rep}
    \frac{1}{2 c_2}
    F(J,A)
\geq
	\frac{1}{2}
    \sum_{(x,y) \in \cE}
        \Hm^1\bigl([x,y] \cap A\bigr) |J(x,y)|_V
=
     |\iota_1 J|(A)
\geq
     |\iota_1 J(A)|
=
     |\iota_1 \calR j(A)|
        \, .
\end{align}
Minimising over $J$, we infer that
\begin{align*}
    \frac{1}{2 c_2}
    f_{\omega,\calR}(j,A)
    \geq
    |\iota_1 \calR j(A)|
    \, .
\end{align*}
Applying this bound to $A/\eps$,
and using
Lemma \ref{lem:scaling-again}
and
\eqref{eq:J-scaling},
we find
\begin{align}
	\label{eq:low-tem}
	\frac{1}{2 c_2}
    f_{\omega,\eps, \calR}(j,A)
	=
    \frac{\eps^d}{2 c_2}
    f_{\omega,\calR}\Bigl( j,\frac{A}{\eps}\Bigr)
    \geq
    \eps^d
	\Bigl|\iota_1 \calR j\Bigl(\frac{A}{\eps}\Bigr) \Bigr|
    =
    |\iota_\eps \calR_\eps j (A)| \, .
\end{align}

	Next we claim that
	$\iota_\eps \calR_\eps j (A) \to j \Lm^d(A)$.
	To prove this, we first note that
	 $\iota_\eps \calR_\eps j$ converges vaguely to
	 $j \Lm^d$ since $\calR$ is a uniform-flow operator,
	and we recall that
	$A$ is relatively compact with $\Lm^d(\partial A) = 0$.
	The claim therefore follows by Remark~\ref{rem:abs_cont_Reps} and
	\cite[Propostion~1.62]{Ambrosio-Fusco-Pallara:2000} .

	Using the claim, we let $\eps \to 0$ in \eqref{eq:low-tem} to obtain
	\begin{align}
	\frac{1}{2 c_2}
	\liminf_{\eps \to 0}
		f_{\omega,\eps,\calR}(j,A)
    \geq
    	|j \Lm^d(A)|
	=
		|j|  \, \Lm^d(A) ,
\end{align}
as desired.

We are left with the proof of the upper bound, which follows from the assumption $F_\omega(0,A) \leq C_\omega \Lm^d(A)$, (F1), a similar argument as in Remark~\ref{rem:abs_cont_Reps}. Indeed, for every bounded $A \in \calB(\R^d)$, simply insert the representative $\calR_\eps j \in \Rep_{\eps,\calR}(j,A)$ and by the Lipschitz property \eqref{eq:Lipschitz-F-eps} of $F_{\omega,\eps}$
\begin{align}
\label{eq:proof_ub_f}
		f_{\omega,\eps,\calR}(j, A)
			\leq
		F_{\omega,\eps}(\calR_\eps j, A)
                &\leq
		F_{\omega,\eps}(0,A)
                +
            |\iota_\eps  \calR_\eps j |
             \Big( B(A,\eps R_\Lip) \Big)
\\
                &\leq
            C_\omega \Lm^d(A)
                +
            |\iota_\eps  \calR_\eps j |
             \Big( B(A,\eps R_\Lip) \Big)
                \, .
\end{align}
In order to use the boundedness of $\calR$, valid only for orthotopes, we observe that
\begin{align}
\label{eq:proof_lip_2}
    B(A,\eps R_\Lip)
        \subset
    \tilde B :=
    \bigcup
    \big\{
        Q_z = z + [0,\eps)^d
            \suchthat
        z \in \eps \Z^d
            \, , \quad
        Q_z \cap B(A,\eps R_\Lip) \neq \emptyset
    \big\}
        \, ,
\end{align}
where the set on the right-hand side is a finite ($A$ being bounded) union of disjoint cubes of side-length $\eps$.  Therefore, we can apply boundedness of $\calR$ (Remark~\ref{rem:vague_J0_cubes}) and obtain
\begin{align}
    |\calR_\eps(j'-j)|(\tilde B)
        \leq
    C |j-j'| \Lm^d(\tilde B)
        \leq
    C |j-j'| \Lm^d(B(A,\eps (R_\Lip+\sqrt{d}))
        \,  ,
\end{align}
where at last we used the trivial inclusion $\tilde B \subset B(A,\eps (R_\Lip+\sqrt{d}))$. The latter inequality, together with \eqref{eq:proof_ub_f} yields
 \begin{align}
 \label{eq:bounds_balls_uniform-flow}
		f_{\omega,\eps,\calR}(j, A)
			\leq
		C_2\Lm^d(A)
			+
		C \Lm^d(B(A,\eps \tilde R)) |j|
            \, ,
\end{align}
 which shows the claimed upper bound, thanks to the fact that $\Lm^d(B(A,\eps \tilde R)) \to \Lm^d(A)$ as $\eps \to 0$ if (and only if) $\Lm^d(\partial A)=0$.

	\medskip

	\emph{(ii)}: \
	To show (ii), we start from the lower bound. Pick
	$J \in \Rep_{\eps,\calR}(j, A)$ and define $J' := J + \calR_\eps (j'-j)$.
	It is readily checked that $J' \in \Rep_{\eps,\calR}(j', A)$.
	By the  Lipschitz properties of $F_{\omega,\eps}$ \eqref{eq:Lipschitz-F-eps}, we obtain
	\begin{align}
 \label{eq:proof_lip_1}
		f_{\omega,\eps,\calR}(j',A)
		\leq
		F_{\omega,\eps}(J',A)
		& \leq
	F_{\omega,\eps}(J,A) +
       c_2 |\calR_\eps(j'-j)|(B(A,\eps R_\Lip))
            \, .
	\end{align}
Arguing as in \eqref{eq:proof_lip_2}, from \eqref{eq:proof_lip_1} we infer that
        \begin{align}
            f_{\omega,\eps,\calR}(j',A)
                \leq
            F_{\omega,\eps}(J,A)
                +
            C c_2 |j-j'| \Lm^d(B(A,\eps \tilde R ))
                \, .
        \end{align}
	Minimising over all admissible $J \in \Rep_{\calR,\eps}(j, A)$, these bounds together yield
		$f_{\omega,\eps,\calR}(j',A) \leq f_{\omega,\eps, \calR}(j,A) + Cc_2 \Lm^d\big( B(A,\eps \tilde R) \big) |j'-j|$.
	The other bound follows by exchanging the roles of $j$ and $j'$.

	\medskip

	\emph{(iii)}: \
	Take near-optimal competitors $J_i$ in every $A_i$, namely for a given $\delta>0$, let $J_i \in \Rep_\calR(j,A_i)$ be such that $F_\omega(J_i,A_i) \leq f_{\omega,\calR} (j,A_i) + \delta$.
	Consider the glued field
	\[
	J(x,y):= \begin{cases}J_i(x,y) &\text{ if }[x,y]\subset A_i\\
		\calR j(x,y)&\text{ otherwise.}
	\end{cases}
	\]

	Note that for any edges $(x,y)\in \cE$ near the boundaries of the $A_i$ we have $J_i(x,y) = \calR j(x,y)$. The glued field $J$ is thus divergence-free and has the right boundary values in $A$, or in other words $J \in \Rep_\calR(j,A)$.
	Therefore, using the additivity assumption (F3) we obtain
	\[
	f_{\omega,\calR}(j,A) \leq  F_\omega(J,A) = \sum_{i=1}^N  F_\omega(J,A_i) \leq \sum_{i=1}^N f_{\omega,\calR}(j,A_i)+\delta.
	\]
	Since $\delta$ is arbitrary, this shows subadditivity.
\end{proof}

We now use the subadditive ergodic theorem to define the homogenized energy density, see e.g. \cite[Theorem~4.1]{Licht-Michaille:2002}.

\begin{defi}	\label{def:f_hom}
	We define $f_{\omega,\hom}:V\otimes \R^d \to \R$
	as the limit
	\begin{align*}
		f_{\omega,\hom}(j):= \lim_{\eps \to 0} \frac{f_{\omega,\calR}(j,A/\eps)} {\Lm^d(A/\eps)} \, , \quad \forall j \in V \otimes \R^d \, ,
	\end{align*}
where $A \subset \R^d$
is an arbitrary nonempty, open, convex, bounded set.
\end{defi}

By definition, the homogenised density $f_{\omega,\hom}(j)$ is obtained by taking a set $A$ and computing the energy density on the blown-up sets $A/\eps$, while the underlying graph remains fixed.
The following simple lemma asserts that one may equivalently keep the size of the set $A$ fixed and shrink the underlying graph so that the edge lengths become of order $\eps$.
This point of view will be convenient in the sequel.

\begin{lemma}[Equivalent formula for the homogenised density]
	\label{lemma:fhom_formula}
	For $j_0 \in V \otimes \R^d$ and $A \in \calA$
	we have
	\begin{align}
	\label{eq:fhom_formula}
		f_{\omega,\hom}(j)
			=
		\lim_{\eps \to 0}
			\frac
			{f_{\omega,\eps,\calR}
				\big(
				 j, A
				\big)
			}
			{\Lm^d(A)}.
	\end{align}
	\end{lemma}

	\begin{proof}
	Lemma \ref{lem:scaling-again} yields
	\begin{align*}
		\frac{f_{\omega,\calR} (
						j,
						A/\eps
					)}
			{\Lm^d(A/\eps)}
		= \frac{ f_{\omega,\eps,\calR}(
				        j, A)}
			{\Lm^d(A)}
			,
	\end{align*}
	hence the result follows by passing to the limit $\eps \to 0$.
	\end{proof}

\begin{lemma}	\label{lemma:prop_fhom}
	The function $f_{\omega,\hom}:V\otimes \R^d \to \R$ exists almost surely. Moreover,
	\begin{itemize}
		\item [(i)] $f_{\omega,\hom}$ has at least linear growth: $f_{\omega,\hom}(j) \geq c|j|$.
		\item [(ii)]$f_{\omega,\hom}$ is Lipschitz: $|f_{\omega,\hom}(j) - f_{\omega,\hom}(j')| \leq C|j-j'|$.
		\item [(iii)] For all $x\in \R^d$ and almost every $\omega\in \Omega$ we have $f_{\hom,\tau_x \omega} = f_{\omega,\hom}$. If $P$ is ergodic, then $f_{\omega,\hom}$ is independent of $\omega$.
	\end{itemize}
\end{lemma}

\begin{proof}
	Points (i) and (ii) follow from Lemma \ref{lemma: F_omega}. Point (iii) follows from the properties of the random variables and the ergodic theorem (see \cite[Prop~1]{DalMaso-Modica:1986}) and \cite[Proposition~2.3]{Braides-Maslennikov-Sigalotti:2008}).
\end{proof}

We note that one could show now that $f_{\omega,\hom}:V\otimes \R^d\to \R$ is $\dive$-quasiconvex, and if $V=\R$ actually $f_{\omega,\hom}:\R\otimes \R^d \sim \R^d \to \R$ is convex. However, it is not necessary to prove the $\dive$-quasiconvexity property for $f_{\omega,\hom}$, as it arises as a natural consequence of the $\Gamma$-convergence and the lower semicontinuity of  the any $\Gamma$-limit, see e.g. \cite[Theorem~1.2]{Baia-Chermisi-Matias-Santos:2013}.

\section{Correctors to the divergence equation}
\label{sec:correctors}
In this section, we discuss the existence and properties of correctors to the discrete divergence equation, in a similar spirit as in \cite{Gladbach-Kopfer-Maas-Portinale:2023}. A similar result in $L^p$ in a continuous and periodic setting is provided by \cite[Lemma~2.14]{Fonseca-Muller:1999} (no locality property of the correctors is therein discussed). This is a crucial tool when performing corrections in the proof of the lower bound in our main result.

We first recall \cite[Def~4.4]{Gladbach-Kopfer-Maas-Portinale:2023} the definition of vector field associated to a	simple directed path $P \in \fP$ on $(\cX,\cE)$.
For an edge
		$e =(x,y) \in \cE$,
	the corresponding reversed edge will be denoted by
		$\overline e = (y,x) \in \cE$.

	\begin{defi}[Unit flux through a path]	\label{def:vectoralongpath}
		Let $P := (x_i)_{i=0}^m \in \fP$ be a simple path in $(\cX, \cE)$,
		thus $e_i = (x_{i-1}, x_i) \in \cE$
		for $i = 1, \ldots, m$,
		and $x_i \neq x_k$ for $i \neq k$.
		The \emph{unit flux through $P$}
		is the discrete vector field
			$J_{P} \in \R^{\cE}_a$
		given by
		\begin{align}	\label{eq:divergence_J_Pq}
			J_P(e) =
		\begin{cases}
		1  & \text{if } e = e_i \text{ for some } i \, , \\
		-1 & \text{if } e = \overline e_i  \text{ for some } i\, ,\\
		0  & \text{otherwise} \, .
		\end{cases}
		\end{align}
	\end{defi}

The next lemma collects some of the key properties of these vector fields.
	\begin{lemma}[\cite{Gladbach-Kopfer-Maas-Portinale:2023}; Properties of $J_P$]	\label{lem:J_Pq}
		Let $P := (x_i)_{i=0}^m$ be a simple path in $(\cX, \cE)$.
				The discrete divergence of the associated unit flux
					$J_P : \cE \to \R$
				is given by
				\begin{align}
					\dive J_P
					=
					\one_{\{x_0\}} - \one_{\{x_m\}}.
				\end{align}
	\end{lemma}
The next proposition shows a self-strengthening of the condition (G2) under the validity of (G1), which plays an important role in this section.

\begin{lemma}[Localisation of (G2)]
\label{lemma:localisation_G}
	Let $(\cX,\cE)$ be a graph satisfying \emph{(G1)} and \emph{(G2)}. Then there exists  $C<\infty$ which only depends on $R_1$, $R_2$ so that for every $\eps \in (0,1)$,  $x,y \in \cX_\eps$, there exists a path $P_{xy} = (x_i)_{i=1}^m \in \tP$ on $(\cX_\eps,\cE_\eps)$ such that
	\begin{align}
	\label{eq:localisation_statement}
		\emph{length}(P_{xy}) \leq R' \big( |x-y | + \eps)
			\tand
		\sup_{i \in \{1, \dots, m\} }
		\sup_{z \in [x_i, x_{i+1}]}
			\dist
			\big(
				z , [x,y]
			\big)
		\leq C \eps
			\, ,
	\end{align}
	for some $R'< \infty$ which only depends on $R_1$, $R_2$.
\end{lemma}

\begin{proof}
The intuition behind the proof stems from the fact that, thanks to Assumptions (G1) and (G2), we are able to construct a path from $x,y$ gluing together paths of length of order $\eps$, while ensuring that we do not drift too far away from the segment $[x,y]$. Fix $x,y \in \cX_\eps$.

\smallskip
\noindent
\textit{Step 1}: \
If $|x-y| \leq \eps R_1$, the statement trivially follows from (G2).

\smallskip
\noindent
\textit{Step 2}: \
Assume that $|x-y| > \eps R_1$ and consider the sequence of points $x=:z_0, z_1, \dots, z_{m-1}$, $z_m:=~y$, for $m:= [(\eps R_1)^{-1}|x-y|] \in \N$, defined by
\begin{align}
	z_{i+1} \in [x,y]
		\tand
	 | z_{i+1} - z_i | = \eps R_1
		\, , \qquad
	\forall i = 0, \dots, m-2
		\, .
\end{align}

\begin{figure}[h]
\centering 
\includegraphics[scale=1.2]{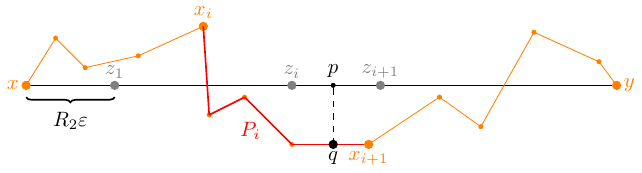}
\caption{Construction of good paths.}
	\label{fig:localisation}
\end{figure}
For every $i$, we use Assumption~(G1) and find a point $x_i \in \cX_\eps$ such that $x_i \in B_{\eps R_1}(z_i)$ (we set $x_0:= x$ and $x_m:= y$). Note that by traingle inequality, we must necessarily have $|x_{i+1} - x_i| \leq 2\eps R_1$, for every $i = 0, \dots, m-1$. Therefore, by Assumption~(G2) we can find a path $P_i \in \tP$ on $(\cX_\eps,\cE_\eps)$ which connects $x_i$ to $x_{i+1}$ of Euclidean length
\begin{align}
\label{eq:length_localised_paths}
	\text{length}(P_i) \leq
        R_2
	\big(
		|x_{i+1} - x_i| + \eps
	\big)
		\leq
	R_2
	\big(
		2\eps R_1 + \eps
	\big)
		\leq
	\eps R_2(2R_1 + 1)
		\, ,
\end{align}
see Figure~\ref{fig:localisation}.

We claim that the path $P=(P_0, \dots, P_{m-1})$ obtained gluing them together satisfies the claimed properties. Should the path P not being simple, we can simply eliminate the loops and obtain a simple path who satisfies the sought properties. To this end, pick any $q \in P_i$ and denote by $p \in [x,y]$ its projection onto $[x,y]$ (cfr. Figure~\ref{fig:localisation}). In particular,
\begin{align}
	|q-p| = \dist(q, [x,y]) \geq |q- z_j|
		\, , \qquad
	\forall j =0, \dots, m
		\, .
\end{align}
Therefore, by \eqref{eq:length_localised_paths} and triangle inequality we find that
\begin{align}
	\eps R_2(2R_1 + 1)
		&\geq
	\text{length}(P_i)
		\geq
	|x_i - q| + |q-x_{i+1}|
\\
		&\geq
	|z_i - q| - |x_i - z_i| + |q- z_{i+1}|  -  |z_{i+1} - x_{i+1}|
\\
		&\geq
	2 \dist
	\big(
		q, [x,y]
	\big)
		- 2 \eps R_1
		\, ,
\end{align}
which yeilds $\dist(q, [x,y]) \leq C \eps$, for $C:= \eps R_2(R_1 + 1/2) + R_1$. This shows the second part of \eqref{eq:localisation_statement}. The bound on the total length simply follows from \eqref{eq:length_localised_paths} and the fact that $m = \big[ |x-y| / \eps R_1 \big] \leq  |x-y| / \eps R_1$, which implies
\begin{align}
	\text{length}(P)
		=
	\sum_{i=0}^{m-1} \text{length}(P_i)
		\leq
	 m \eps R_2(2R_1 + 1)
	 	\leq
	  (2 R_2 + R_2R_1^{-1} )  |x-y|
	 	\, ,
\end{align}
and thus concludes the proof.
\end{proof}

We are ready to prove the main result of this section. In what follows, $\diam$ denotes the Euclidean diameter.

\begin{prop}[Correctors: existence and estimates]\label{prop:correctors}
	Let $m \in \calM_0(\cX_\eps; V)$ and set $\mu := \iota_\eps m$.
	Then there exists a discrete vector field $J\in V^{\cE_\eps}_a$ with the following properties: there exists $C=C(R_2)<\infty$ independent of $\eps$, $m$ so that
	\begin{enumerate}
		\item\label{item:c1} $J$ satisfies the discrete divergence equation  $\DIVE J = m$.
		\item\label{item:c2} We have the estimate
		$ \bigl|\iota_\eps J \bigr|(\R^d)
		\leq  C
		\left(
			\| \mu\|_{\tKR(\R^d)}
				+ \eps | \mu |(\R^d)
		\right)
		$.
	\end{enumerate}
	Assume additionally that $\supp(\mu) \subset Q$ where $Q$ is a convex set. Then  there exists a constant $C=C(d,R_1,R_2)< \infty$ independent of $\eps$ and $m$ so that
	\begin{enumerate}
		\setcounter{enumi}{2}
	\item\label{item:c3} $J$ can be chosen supported at distance at most $\eps$ from $Q$, in the sense that
	\begin{align}
		\supp(\iota_\eps J) \subset B_{C\eps}(Q)
		\, .
	\end{align}
	\end{enumerate}
\end{prop}
\begin{rem}[Dependence on the total variation of $\mu$]
\label{rem:dep_TV}
	The presence of the total variation in (2) is a consequence of the general assumptions on the graph, in particular (G2).
	If one assumes the stronger condition that
		for every $x,y\in \cX$
	there exists a path $P$ in $(\cX,\cE)$ with Euclidean length
		$(P)\leq R_2 |x-y|$,
	then one can get rid of the term $\eps \|\mu\|_{\TV(\R^d)}$.
	In particular, this holds for the cartesian grid $(\Z^d, \E^d)$.
\end{rem}

	Before proving Proposition \ref{prop:correctors},
	it is instructive to state and prove the following continuous counterpart.

\begin{prop}[Correctors at the continuous level]
	\label{prop:corrector_Euclidean}
	For every $\mu \in \M_0(\R^d;V)$ there exists
		$\nu \in \M(\R^d ; V \otimes \R^d)$
		such that
	\begin{align}
	\label{eq:claim_corrector_cont}
		\nabla \cdot \nu = \mu
			\; \tand \;
		\| \nu \|_{\TV(\R^d;V)}
			\leq
		\dim(V)
		\| \mu \|_{\tKR(\R^d;V)}
			\, .
	\end{align}
\end{prop}

\begin{proof}
	Fix a basis $(e_i)_i$ of $V$ and let $(e_i^*)_i$ be the corresponding dual basis of $V^*$.
	Let $\mu_i = \ip{\mu, e_i^*} \in \M_0(\R^d)$, so that
		$\mu = \sum_{i} \mu_i \otimes e_i$.
	Then we can write
		$\mu_i = \mu_i^+ - \mu_i^-$,
		so that $\| \mu_i \|_{\tKR(\R^d)} = \mathbb W_1(\mu_i^-, \mu_i^+)$ by
		Kantorovich duality \eqref{eq:remark_KR_T1}
		for the scalar optimal transport problem.
	Let $\pi_i \in \M_+(\R^d \times \R^d)$ an optimal coupling for $\mathbb W_1(\mu_i^-, \mu_i^+)$, i.e.,
	\begin{align}
		\| \mu_i \|_{\tKR(\R^d)}
			=
		\int_{\R^d \times \R^d} | x-y | \dd \pi_i(x,y)
			\; \;  \text{and} \; \;
		(\tP_1)_{\#} \pi_i = \mu_i^-
			\, , \,
		(\tP_2)_{\#} \pi_i = \mu_i^+
			\, .
	\end{align}
	We set $\pi:= \sum_i \pi_i \otimes e_i \in \M(\R^d \times \R^d ;V)$ and observe that $\pi$ is a coupling for $\mu$, in the sense that $(\tP_2)_{\#} \pi - (\tP_1)_{\#} \pi = \mu$.

	We claim that the coordinate projections decrease the $\tKR$-norm, in the sense that
	\begin{align}
	\label{eq:claim_decrease}
		\| \mu_i \|_{\tKR(\R^d)}
		\leq \| \mu \|_{\tKR(\R^d; V)}
			\, , \qquad
		\forall i = 1, \dots, n
			\, .
	\end{align}
	To show \eqref{eq:claim_decrease}, fix $x_0 \in\R^d$ from the definition of the $\tKR$-norm
	and pick a Lipschitz test function
	$\varphi : \R^d \to \R$ with $\varphi(x_0) = 0$ and
	$\Lip(\varphi) \leq 1$.
	Set $\Phi_i := \varphi \otimes e_i^* : \R^d \to V^*$. Then $\Lip(\Phi_i) = \Lip(\varphi)\leq 1$ and
	\begin{align*}
		\int_{\R^d} \varphi \dd \mu_i
			=
		\int_{R^d}
		\ip{
			\Phi_i , \dd \mu
		}
			\leq
		\| \mu \|_{\tKR(\R^d; V)}
			\, .
	\end{align*}
	Taking the supremum over $\varphi$ we obtain \eqref{eq:claim_decrease}.

	Finally, we define $\nu \in \M(\R^d ; V \otimes \R^d)$ as
	\begin{align*}
		\nu := \int_{\R^d \times \R^d}
		\frac
			{\dd \pi}
			{\dd |\pi|_V}(x,y)
		\otimes
		\nu^{xy} \dd |\pi|_V(x,y)
			\, , \where
		\nu^{xy} :=
		\frac
			{y-x}{|y-x|}
			\left(
				\Hm^1 \res [x,y]
			\right)
				\, . \quad
	\end{align*}
	We are left to prove that $\nu$ satisfies \eqref{eq:claim_corrector_cont}.
	For the divergence condition we use the fact that $\nabla \cdot \nu^{xy} = \delta_y - \delta_x$
	to obtain
	\begin{align}
		\nabla \cdot \nu
			=
		\int_{\R^d \times \R^d}
			\nabla \cdot \nu^{xy} \dd \pi(x,y)
			=
		\int_{\R^d \times \R^d}
			\big( \delta_y - \delta_x  \big)
				\dd \pi(x,y)
			=
		(\tP_2)_{\#} \pi - (\tP_1)_{\#} \pi = \mu
			\, .
	\end{align}
	Moreover, since
		$\| \nu^{xy} \|_{\TV(\R^d)} = | x-y | $
	and
		$|\pi|_V \leq \sum_i \pi_i$
	we infer that
	\begin{align}
	\label{eq:tv_cont}
	\begin{aligned}
		\| \nu \|_{\TV(\R^d)}
			&\leq
		\int_{\R^d \times \R^d}
			|x-y| \dd |\pi|_V(x,y)
			\leq
		\sum_i
			\int_{\R^d \times \R^d}
				|x-y| \dd \pi_i(x,y)
\\
			&=
		\sum_i
			\| \mu_i\|_{\tKR(\R^d)}
			\leq
		\dim(V) \| \mu \|_{\tKR(\R^d;V)}
			\, ,
	\end{aligned}
	\end{align}
	where at last we used \eqref{eq:claim_decrease}.
\end{proof}

\begin{rem}
	It is worth noting that, one could have proved the same bound (in fact with a better constant $C'$) if we chose $\pi \in \M(\R^d, \R^d)$ to be any solution to $T_1(\mu)$ in the vectorial sense.  Unfortunately, it seems nontrivial to show the existence of an optimal transport plan in the vectorial formulation. Similarly, properties of (quasi)-optimal transport plans (for example that $\supp(\pi) \subset \supp(\mu) \times \supp(\mu)$, which is crucial for the proof of (3) in Proposition~\ref{prop:correctors})  are nontrivial in the vectorial case. This explains the reason of working with the scalar components of $\mu$.
\end{rem}

Inspired by the proof on $\R^d$, we are now ready to prove Proposition~\ref{prop:correctors}.

\begin{proof}[Proof of Proposition~\ref{prop:correctors}]
Given $\mu=\iota_\eps m$,
we use a similar construction as in $\R^d$,
adapting the proof to ensure that the constructed $\nu$ is of the form $\nu = \iota_\eps J_\eps$ for some $J_\eps \in \R_a^{\cE_\eps}$.
For this purpose, let $\pi = (\pi_1, \dots , \pi_n) \in \M(\R^d\times \R^d)$ be the associated vector of optimal transport plans associated to $\mu$ as in the proof of \eqref{eq:claim_corrector_cont}.
As $\mu$ is an atomic measure, the same holds for $\mu_i$ and $\pi_i$, for every $i$.
Moreover, for all $i$ we have by construction,
\begin{align}
	\supp(\pi_i) \subset \supp(\mu_i^-) \times  \supp(\mu_i^-) 	\subset \supp(\mu_i)  \times \supp(\mu_i) \subset \supp(\mu)  \times \supp(\mu)
		\, ,
\end{align}
which implies that
	$\supp(\pi) \subset \supp(\mu) \times \supp(\mu)$.

The only modification we have to make to construct a compatible $\nu$ is to replace, for every $x,y \in \supp(m) \subset \cX_\eps$, the measure $\nu^{xy}$ in \eqref{eq:def_nu} with suitable discrete measures along paths on $(\cX_\eps,\cE_\eps)$ in the sense of Definition~\ref{def:vectoralongpath}.
Precisely: for every couple $(x,y)\in \supp \pi$, we pick an optimal (w.r.t. the discrete distance structure on $(\cX_\eps, \cE_\eps)$) path $P_{xy}= (z_i^{xy} \in \cX_\eps)_{i=0}^{m_{xy}}$ connecting $x$ and $y$ in $(\cX_\eps,\cE_\eps)$. We then replace $\nu^{xy}$ in \eqref{eq:def_nu} with $\iota_\eps J_{P_{xy}} \in \M(\R^d ; \R^d)$, where $J_P$ is defined in Definition~\ref{def:vectoralongpath}.
In explicit formulas, we set
\begin{align}
	\nu_\eps :=
	\int_{\R^d \times \R^d}
	\frac
		{\dd \pi(x,y)}
		{\dd |\pi|(x,y)}
	\otimes
	\iota_\eps J_{P_{xy}} \dd |\pi|(x,y)
		\in \M(\R^d ; V \otimes \R^d)
			\, .
\end{align}
Writing $\pi = \sum_{x,y} p(x,y) \delta_{(x,y)}$, we can further write
\begin{align}
	\nu_\eps
		=
	\sum_{(x,y)\in\cE_\eps}
		p(x,y) \otimes \iota_\eps J_{P_{xy}}
	 	=
 \iota_\eps J_\eps
	 	\, , \where
 J_\eps:=
 		\sum_{(x,y)\in\cE_\eps}
 			p(x,y) J_{P_{xy}}
 		\, .
\end{align}
Thanks to what showed in Remark~\ref{prop:corrector_Euclidean}, we know that $\nabla \cdot \nu_\eps = \mu$, which by Lemma~\ref{lemma:divergence_discr_cont} is equivalent to $\dive J_\eps = m$. Arguing as in Proposition~\ref{prop:corrector_Euclidean}, in particular in \eqref{eq:tv_cont}, we control the total variation as
\begin{align}
\| \nu_\eps \|_{\TV(\R^d)}
		&\leq
	\int_{\R^d \times \R^d}
		\| \iota_\eps J_{P_{xy}} \|_{\TV(\R^d)} \dd |\pi|(x,y)
		\leq
	C \sum_{i=1}^n
		\int_{\R^d \times \R^d}
			\| \iota_\eps J_{P_{xy}} \|_{\TV(\R^d)} \dd \pi_i(x,y)
 .
\end{align}
In order to bound the total variation of $\iota_\eps J_{P_{xy}}$, we employ Assumption~(G2) to infer that
\begin{align}
	\| \iota_\eps J_{P_{xy}} \|_{\TV(\R^d)}
		=
	\frac12 \sum_{(z,w) \in \cE_{\eps}}
		|J_{P_{xy}}(z,w)| \| x-y \|
		=
	L(P_{xy})
		\leq
	R \big( |x-y| + \eps )
		\, ,
\end{align}
which together with the previous inequality ensures that
\begin{align}
	\frac1R \| \nu_\eps \|_{\TV(\R^d)}
		&\leq
	C  \sum_{i=1}^n
		\int_{\R^d \times \R^d}
			\big( |x-y| + \eps \big) \dd \pi_i(x,y)
		=
	C
	\sum_{i=1}^n
	\Big(
		\| \mu_i \|_{\tKR(\R^d)}
			+
		\frac\eps 2 | \mu_i |(\R^d)
	\Big)
\\
		&\leq
	C  n
	\Big(
		\| \mu \|_{\tKR(\R^d)}
			+
		\frac12 \eps \| \mu \|_{\TV(\R^d)}
	\Big)
		\, ,
\end{align}
where we at last we used \eqref{eq:claim_decrease} and that $|\mu_i|(\R^d) \leq |\mu|(\R^d)$, for every $i=1, \dots, n$.
This proves \eqref{item:c2}.

We are left to prove (3), assuming in addition that $\supp(\mu) \subset Q$ for $Q$ convex. To this end, we construct $J_\eps$ as above, but thanks to Lemma~\ref{lemma:localisation_G} we choose the paths $P_{xy}=(z_i^{xy} \in \cX_\eps)_{i=0}^{m_{xy}}$ in such a way that
\begin{align}
	\dist \big( z, [x,y] \big)
		\leq C' R \eps
	\, , \quad \forall z \in \big[ z_i^{xy}, z_{i+1}^{xy} \big]
	\, , \quad \forall i = 1 , \dots, m_{xy}
		\, .
\end{align}
In this way we ensure that $\supp(J_{P_{xy}}) \subset B_{C'R \eps}([x,y])\subset B_{C'R\eps}(Q)$ for every $x,y \in \supp(m)$, which clearly implies that $J_\eps$ satisfies (3).
\end{proof}

\section{Tangent measures and blow-up of divergence measures}
\label{sec:tangent}
In this section we discuss the notion of \emph{tangent measure} for a $W$-valued measure,
where $W$ is a finite-dimensional normed vector space.
Furthermore, we show that any tangent measure whose divergence is a measure is almost everywhere divergence-free.

We shall take advantage of the following general version of \textit{Besicovitch differentiation theorem}, see e.g.,~\cite[Proposition~2.2]{Ambrosio-DalMaso:1992} and \cite[Theorem~2.2]{Ambrosio-Fusco-Pallara:2000}.

\begin{prop}[Besicovitch differentiation theorem]
\label{prop:Besicovitch}
Let $\nu\in\M(U;W)$ be a Radon measure on a Borel set $U \subset \R^d$, and let $\xi\in\M_+(U)$.  Then there exists a Borel set $E \subset U$, with $\xi(E) = 0$, so that for every $x \in \supp(\xi) \setminus E$ and  for every bounded, convex, open set $C$ containing the origin,
\begin{align}
\label{eq:limit_Besic}
    L(x) := \lim_{r \to 0^+}
    \frac
        {\nu(x + r C)}
        {\xi(x + r C)}
           \quad \text{exists and does not depend on $C$}
            \, .
\end{align}
Moreover, the identity
	$L = \frac{\dd \nu}{\dd \xi}$
holds
	$\xi$-a.e., 
where $\frac{\dd \nu}{\dd \xi}$ denotes the density of the absolutely continuous part in the Lebesgue decomposition $\nu = \frac{\dd \nu}{\dd \xi} \xi + \nu^s$.
Finally, we have
$\nu_s = \nu \res \tilde E$, where
\begin{align}
    \tilde E = \big( U \setminus \supp \nu \big)
        \cup
    \left\{
        x \in \supp(\nu)
            \suchthat
        \lim_{r \to 0^+}
        \frac
            {|\nu|(x + r B_1)}
            {\xi(x + r B_1)}
                = +\infty
    \right\} \, . 
\end{align}
\end{prop}

It is crucial for our application that the exceptional set $E$ does not depend on the set $C$.
Indeed, in the proof of the lower bound 
we perform a blow-up procedure around singular points. In this application, the set $C$ will be a strip whose orientation depends on the point itself.

As a consequence of the previous result, we have the following corollary.
\begin{cor}[Lebesgue's points]
\label{cor:Lebesgue_points}
    Let $\xi \in \M_+(U)$ be a nonnegative measure on $U \subset \R^d$, and let $f \in L^1(\xi)$. Then there exists a set $E \subset U$ with $\xi(E)=0$ so that, for very $x \in \supp(\xi)\setminus E$, we have
    \begin{align}
        \lim_{r \to 0^+}
            \frac1{\xi(x+rC)}
            \int_{x+rC}
            \left|
                f(y) - f(x)
            \right|
                \dd \xi(y)
            = 0
                \, ,
    \end{align}
    for every bounded, convex, open set $C$ containing the origin.
\end{cor}
The proof follows the same line of \cite[Corollary~2.23]{Ambrosio-DalMaso:1992} and it is direct consequence of Proposition~\ref{prop:Besicovitch}.

We refer to the complement of $E$ as the set of \textit{Lebesgue points} of $f$ relative to $\xi$.

For $\delta > 0 $ and $x \in \R^d$,
we will consider the rescaling function
	$\rho_{\delta,x}:\R^d \to \R^d$
given by $\rho_{\delta,x}(y):= (y-x)/\delta$.

\begin{defi}[Tangent measure]
Let $\nu \in \M(\R^d; W)$ be a Radon measure
and let $x \in \R^d$.
Given an open bounded convex set
	$C \subset \R^d$ and $\delta > 0$,
we consider the rescaled measures
\begin{align}
\label{eq:def_jdelta}
	\nu_{\delta,x}:=
	\frac{1}{|\nu|(\delta (C-x))}
		(\rho_{\delta,x})_{\#} \nu
	\in \M(\R^d; W)
		\, .
\end{align}
Any accumulation point of $\nu_{\delta,x}$ as $\delta \to 0$ in the vague topology of $\M(\R^d; W)$ is called a
	\emph{($C$-)tangent measure}
of $\nu$ at $x$.
The set of all $C$-tangent measures to $\nu$ at $x$ is denoted by
$\Tan_C(j,x)$.
\end{defi}

In the setting of the definition above, it is well known that $\Tan_C(\nu,x) \neq \emptyset$; see, e.g., \cite[Prop.~3.4]{DeLellis:2006}). In fact, thanks to Corollary~\ref{cor:Lebesgue_points} we can say much more about tangent measures, as the next Lemma provides.

\begin{rem}[Mass of a tangent measure]
	\label{rem:mass_tangent}
	Let
	$C \subset \R^d$
	be an
	open bounded convex set
	and let $\nu \in \M(\R^d; W)$.
	Then,
	$|\tau|(C)\leq 1$
	for all tangent measures $\tau \in \Tan_C(\nu, x)$
	and all $x \in \R^d$.
	This follows from the
	vague lower semicontinuity of the total variation,
	since
	$|\nu_{\delta,x}|(Q) = 1$ for all $\delta > 0$. In fact, we can always pick a tangent measure whose total mass over $C$ is equal to $1$, see the Lemma below.
\end{rem}

\begin{lemma}[Properties of tangent measures I]
	\label{lemma:tangent_measures}
	Let $\nu \in \M(\R^d; W)$ be a vector-valued measure. Then there exists an $|\nu|$-exceptional set $E \subset \R^d$, i.e.  $|\nu|(E)=0$, so that, for every $x_0 \in \supp(|\nu|) \setminus E$, the following properties hold: for every convex, bounded, open set $C\subset \R^d$,
	\begin{enumerate}
    \item  \label{it:lemma_tang_0}
        The set $\emph{Tan}_C(\nu,x_0)$ contains a tangent measure $\tau$ which satisfies $|\tau|(\partial Q)=0$ as well as $|\tau|(Q)=1$.
	\item	\label{it:lemma_tang_1}
	Every tangent measure $\tau \in \emph{Tan}_C(\nu,x_0)$
		has constant density with respect to its variation
	$|\tau|\in\M_+(\R^d)$, namely
	\begin{align}
		\frac{\de \tau}{\de |\tau|}(y)
			=
		\frac{\de \nu}{\de |\nu|}(x_0)
			\, , \quad
		\text{for $|\tau|$-a.e.~}y \in \R^d
			 \, .
	\end{align}
	\item	\label{it:lemma_tang_2}
	If $\nu_{\delta_m,x_0} \to \tau \in \Tan_C(\nu,x_0)$ vaguely in $\M(\R^d; W)$
	for some null-sequence $(\delta_m)_m$, then $|\nu_{\delta_m,x_0}| \to |\tau|$ vaguely in $\M_+(\R^d)$.
	Consequently,
	\begin{align}
		\Tan_C(\nu,x_0)
			=
		\frac{\de \nu}{\de |\nu|}(x_0) \Tan_C\big(|\nu|,x_0\big)
			\, .
	\end{align}
    \end{enumerate}
    Finally, for $\Lm^d$-a.e.~ $x_0 \in \supp(|\nu|) \setminus E$, we have
	For $\Lm^d$-a.e.~ $x_0 \in \R^d$ we have
	\begin{align}
    \label{eq:lemma_tang_3}
		\Tan_C(\nu,x_0) = \big\{ j \Lm^d \big\} \,
	\end{align}
	where $j
	= \frac{\de \nu}{\de |\nu|}(x_0)
	= \frac{\de \nu}{\de \Lm^d}(x_0) $.
	\end{lemma}

	\begin{proof}
  For \eqref{it:lemma_tang_0} see, e.g., \cite[Lemma~2.5]{Baia-Chermisi-Matias-Santos:2013}. Properties \eqref{it:lemma_tang_1} and \eqref{it:lemma_tang_2} hold for every $x_0$ Lebesgue point of $\frac{\dd \nu}{\dd |\nu|}$ relative to $|\nu|$
  \cite[Theorem~2.44]{Ambrosio-Fusco-Pallara:2000}, which thanks to Corollary~\ref{cor:Lebesgue_points} are all but an exceptional set (independent of the choice of the set $C$). Finally, \eqref{eq:lemma_tang_3} can be found in, 
  e.g., 
  \cite[Example~2.41]{Ambrosio-Fusco-Pallara:2000}.
	\end{proof}

In this work, particular attention will be devoted to tangent measures to Radon measures whose distributional divergence is a Radon measure as well.
\begin{defi}[Distributional divergence]
	The \emph{distributional divergence} of a
	vector-valued measure
		$\nu \in \M(\R^d; V \otimes \R^d)$
	is the distribution
		$\dive \nu \in \mathcal D'(\R^d; V)$
	defined by
\begin{align}
	\ip {\varphi, \dive \nu }
		:=
	- \int_{\R^d}
	\ip{ \nabla \varphi , \de \nu }
		\qquad
	\forall \varphi \in C_c^\infty(\R^d; V^*)
		\, .
\end{align}
Here the dual pairing on the right-hand side is between
$V^* \otimes \R^d$ and $V \otimes \R^d$.
\end{defi}

We consistently use the canonical identifcation between
$V \otimes \R^d$ and $\calL(\R^d;V)$, namely,
$v \otimes x \in V \otimes \R^d$ will be identified with the linear map
	$\R^d \ni y \mapsto \ip{x,y} v \in V$.

Recall that we always identify $\R^d$ (but not $V$) with its dual space in the canonical way.

\begin{defi}[Divergence measures]
\label{def:div_meas}
A vector-valued measure $\nu \in \M(\R^d; V \otimes \R^d)$ is called \emph{divergence measure}
if
its distributional divergence coincides with an element of
	$\M(\R^d; V)$,
in the sense that
there exists a measure $\sigma_\nu \in \M(\R^d; V)$ so that
\begin{align}
	\ip{\varphi,\dive \nu}
		=
	\int_{\R^d}
	\ip{
		\varphi, \de \sigma_\nu
	}
		 \qquad
	\forall \varphi \in C_c^\infty(\R^d; V^*)
		\, .
\end{align}
In this case, we write $\sigma = \dive \nu$.
\end{defi}

For $C\subset \R^d$, $\delta >0$, and $x_0 \in \R^d$, we use the short-hand notation $C_{\delta,x_0} := \delta (C- x_0)$.

\begin{lemma}[Properties of tangent measures II]
\label{lemma:divergence_measures}
Let $C\subset \R^d$ be an open bounded convex set containing the origin,
and let $\nu \in \M(\R^d; V \otimes \R^d)$ be a divergence measure. Denote by $\mu_\nu := \dive \nu \in \M(\R^d ; V)$.
The following assertions hold:
\begin{enumerate}
\setcounter{enumi}{3}
    \item	\label{it:lemma_tang_4}
    $\displaystyle \emph{rank}\Big( \frac{\de \nu}{\de |\nu|}(x_0) \Big) \leq n-1$ for $|\nu|^s$-a.e. $x_0 \in \R^d$.
    \item   \label{it:lemma_tang_5}
    Recall the definition of $\nu_{\delta,x}$ in \eqref{eq:def_jdelta}. Then for $|\nu|$-a.e. $x_0 \in \R^d$, we have that, if
    $\nu_{\delta_m,x_0} \to \tau \in \Tan_C(\nu,x_0)$ vaguely as $\delta_m \to 0$, then
    \begin{align}
    \label{eq:tv_bounds_div_rescale}
        \sup_{m \in \N}
        \left|
            \frac
    		  { (\rho_{\delta,x_0})_{\#} \mu_\nu}
    		  {|\nu|(C_\delta(x_0))}
        \right|(B)
            < \infty
                \, ,
    \end{align}
    for every bounded set $B \subset \R^d$.
    \item	\label{it:lemma_tang_6}
    $\dive \tau = 0$ for every $\tau \in \Tan_C(\nu,x_0)$
    for $|\nu|$-a.e. $x_0 \in \R^d$.
\end{enumerate}
\end{lemma}

\begin{proof}

\smallskip
\noindent
\eqref{it:lemma_tang_4}: \ This follows from Remark~\ref{rem:existence_recess}.

\smallskip
\noindent
\eqref{it:lemma_tang_5}. \
We know from Lemma~\ref{lemma:tangent_measures} that $|\nu|$-a.e. $x_0 \in \R^d$, there exists $\tau \in \Tan_C(\nu,x_0)$ so that $\nu_{\delta_m,x_0} \to \tau$ for some $\delta_m \to 0$. In this case, we also know that $|\nu_{\delta_m,x_0}| \to |\tau|$ vaguely in $\M(\R^d; V \otimes \R^d)$ as $m \to \infty$. Thanks to Proposition~\ref{prop:Besicovitch}, we also know that $|\nu|$-a.e. $x_0 \in \R^d$, the density of $\mu_\nu$ with respect to $|\nu|$ exists, i.e.
\begin{align}
    \lim_{m \to +\infty}
        \frac
            {\mu_\nu(B_{\delta_m}(x_0))}
            {|\nu|(B_{\delta_m}(x_0))}
    =
        \frac
            {\dd \mu_\nu}
            {\dd |\nu|}(x_0)
                \in \R_+
                    \, .
\end{align}
Then, for every such $x_0 \in \R^d$, for every closed ball $B_R:=B_R(0) \subset \R^d$ of radius $R>0$, we have that
\begin{align}
	 \limsup_{m \to \infty}
        \left|
            \frac
    		  { (\rho_{\delta,x_0})_{\#} \mu_\nu}
    		  {|\nu|(C_\delta(x_0))}
        \right|(B_R)
	&= \limsup_{m \to \infty}
		\frac
			{|\mu_\nu| ( B_{\delta_m R}(x_0) )}
			{|\nu| ( B_{\delta_m R}(x_0) )}
		\frac
			{|\nu| ( B_{\delta_m R}(x_0) ) }
			{|\nu|(C_\delta(x_0))}
\\
	&=
		\frac{\de |\mu_\nu|}{\de |\nu|}(x_0)
		\limsup_{m \to \infty}
			|\nu_{\delta_m,x_0}|(B_R)
	\leq
		\frac{\de |\mu_\nu|}{\de |\nu|}(x_0)
			|\tau|(B_R) < \infty
		\, ,
\end{align}
where at last we used $|\nu_{\delta_m,x_0}| \to |\tau|$ with the fact that $B_R$ is closed.

\smallskip
\noindent
\eqref{it:lemma_tang_6}. \  Let us compute the divergence of the rescaled measure $\nu_{\delta,x_0}$: for a given function $\varphi \in C_c(\R^d;V)$, define $\varphi_\delta(\cdot):= \delta \varphi(\cdot-x_0/\delta)$. By definition, we have that
\begin{align}
	\langle
		\dive \nu_{\delta,x_0} , \varphi
	\rangle
&=
	- \int
	\langle
		\nabla \varphi , \de \nu_{\delta,x_0}
	\rangle_{V \otimes \R^d}
=
	- \frac1{|\nu|(C_\delta(x_0))}
	\int
	\left \langle
		(\nabla \varphi)\Big( \frac{\cdot-x_0}\delta\Big) , \de \nu
	\right \rangle_{V \otimes \R^d}
\\
&=
	- \frac1{|\nu|(C_\delta(x_0))}
	\int
	\left \langle
		\nabla \varphi_\delta, \de \nu
	\right \rangle_{V \otimes \R^d}
=
	\frac1{|\nu|(C_\delta(x_0))}
	\int
	\left \langle
	 	\varphi_\delta, \de \mu_\nu
	\right \rangle_V
		\, ,
\end{align}
where at last we used that $\dive \nu = \mu_\nu \in \M(\R^d ; V)$. We continue with the definition of push forward and obtain
\begin{align}
	\langle
		\dive \nu_{\delta,x_0} , \varphi
	\rangle
=
	\frac\delta{|\nu|(Q_\delta(x_0))}
	\int
	\left \langle
	 	\varphi\Big(\frac{\cdot-x_0}\delta\Big), \de \mu_\nu
	\right \rangle_V
=
	\delta \int
	\left\langle
		\varphi , \frac
		{\de (\rho_{\delta,x_0})_{\#} \mu_\nu}
		{|\nu|(Q_\delta(x_0))}
	\right\rangle_V
		\, .
\end{align}
In other words, this shows that $\nu_\delta$ is also a divergence measure, with
\begin{align}	\label{eq:divergence_rescaled_formula}
\mu_{\delta,x_0} :=
	\dive \nu_{\delta,x_0} =
	\delta
	\frac
		{ (\rho_{\delta,x_0})_{\#} \mu_\nu}
		{|\nu|(Q_\delta(x_0))}
	\in \M(\R^d; V)
		\, .
\end{align}

A direct application of \eqref{it:lemma_tang_4} together with \eqref{eq:divergence_rescaled_formula} shows that, $|\nu|$-a.e. $x_0 \in \R^d$, for every $\tau \in \Tan_C(\nu,x_0)$ so that $\nu_{\delta_m,x_0} \to \tau$, we must have $\dive \nu_{\delta_m,x_0} \to 0$ vaguely in $\M(\R^d;V)$ as $m \to \infty$. Due to the fact that the distributional divergence commutes with the vague convergence of measures, it follows that $\dive \tau =0$, for every $\tau \in \Tan_C(\nu,x_0)$.
\end{proof}
\begin{rem}
Arguing in a similar way as in the previous proof, one could prove a stronger statement, namely that for  $|\nu|$-almost every $x_0 \in \R^d$,
\begin{align}
	\frac
	{ (\rho_{\delta_m,x_0})_{\#} \mu_\nu}
	{|\nu|(Q_{\delta_m}(x_0))}
	\to \hat \mu_{x_0} := \frac{\de \mu_\nu}{\de |\nu|}(x_0) |\tau|
	\quad
	\text{vaguely in }
	\M(\R^d; V)
	\text{ as }m \to \infty
	\, .
\end{align}
\end{rem}

We have shown in
Lemma \ref{lemma:tangent_measures}\eqref{it:lemma_tang_1}
and
Lemma \ref{lemma:divergence_measures}\eqref{it:lemma_tang_6}
that tangent measures of divergence measures
are unidirectional and divergence-free.
The next proposition shows such measures have a special structure: they are "constant" in the direction orthogonal to the kernel of the density. This will be crucial in the proof of the lower bound \eqref{eq:lb_singular} when performing blow-ups around singular points.

\begin{prop}[Unidirectional divergence-free measures]
\label{prop:structure_tangent_measures}
	Let $\tau\in \M(\R^d;V\otimes \R^d)$ be such that
	\begin{align}
		\dive \tau = 0
		\tand
		\frac{d\tau}{d|\tau|}(x) \equiv j
			\quad |\tau|\text{-a.e. }x \in \R^d
	\end{align}
	for some $j \in V \otimes \R^d$.
	Write $\R^d = (\ker j)^\perp \oplus \ker j$ and let $\lambda\in \M_+((\ker j)^\perp)$ denote the Hausdorff measure restricted to $(\ker j)^\perp$. Then there exists a measure $\kappa\in \M_+(\ker j)$ such that
	\[
	\tau = j \lambda \otimes \kappa.
	\]
\end{prop}
\begin{proof}
	Write $S := (\ker j)^\perp \subseteq \R^d$ for brevity
	and set $s := \dim(S)$.
	We claim that the result follows if we could show that
	$\tau$ does not depend on $S$, in the sense that for any $h \in S$,
	\begin{align}
		\label{eq:claim_partial_tau}
		\partial_{h} |\tau| = 0
		\quad \text{in } \, \mathcal D'(\R^d ; V) \, .
	\end{align}
	Indeed, to prove the claim,
	we observe that
	\eqref{eq:claim_partial_tau} implies that, for all test functions $\Psi \in C_c^\infty(\R^d ; V^*)$ and $h \in S$,
	\begin{align*}
		\int_{\R^d}
			\Psi(z)
		\dd \tau(z)
		=
		\int_{\R^d}
			\Psi(z + h)
			\dd \tau(z).
	\end{align*}
	Writing $S_h(z) = z + h$ for the translation map in the direction of $h$,
	this means that
		$(S_h)_\# \tau = \tau$,
	i.e., $\tau$ is invariant with respect to translations in $S$.
	Consequently, $\tau(A \times B) = \tau(A' \times B)$
	for all Borel sets $B \subseteq S^\perp$
	whenever $A, A' \subseteq S$ are translates of each other.
	Using this fact and the finite additivity of $\tau$, we find that
	\begin{align}
		\label{eq:cubeQB}
		\tau(Q \times B) = |Q| \tau\big( [0,1)^s \times B\big)
	\end{align}
	for all cartesian cubes $Q \subseteq S$
	of the form $Q = \prod_{i=1}^s [\alpha_i, \beta_i)$ in $S$ with $\alpha_i, \beta_i \in \Q$.
	By approximation, \eqref{eq:cubeQB} also holds for all such cubes $Q$ with $\alpha_i, \beta_i \in \R$, and therefore it holds for all Borel sets $Q \subseteq S$. This means that
	\begin{align*}
		\tau(Q \times B)
		= |Q| \, K(B),
		\quad \text{where} \quad
		K(B)
		:=
		\tau\big( [0,1)^s \times B\big),
	\end{align*}
	hence $\tau = \lambda \otimes K$.
	Combining this with the assumption
		$\tau = |\tau| j$,
	we infer that
		$K = |K| j$
	which yields the desired result with $\kappa = |K|$.

	To prove \eqref{eq:claim_partial_tau},
	fix $\phi \in C_c^\infty(\R^d)$
	and $h \in S$.
	Our goal is to show that
	$
	\bip{ \partial_h \phi, |\tau| } = 0
	$
	for every $\phi \in C_c^\infty(\R^d)$.
	Let $(e_i)_{i=1}^d$ be an orthonormal basis of $\R^d$
	such that $e_1, \ldots, e_s \in S$
	and $e_1 = h$.
	Let us write $v_i := j e_i \in V$ for $i = 1, \ldots, d$,
	and note that $v_1 \notin \spane\{ v_2, \ldots,  v_d \}$,
	since $v_{s + 1} = \ldots = v_d = 0$ and the restriction of $j$ to $S$ is injective.
	Therefore, we can pick $v \in V^*$ such that
	\begin{align}
		\label{eq:delta}
		\ip{v_k, v^*} = \delta_{1k}
		\quad \text{for }
		k = 1, \ldots d \, .
	\end{align}
	Define $\Psi \in C_c^\infty(\R^d; V^*)$ by
		$\Psi := \phi \otimes v^*$.
	As $\tau$ is divergence-free,
	we obtain
	using the identity $\tau = j |\tau|$,
	the definition of $v_i$,
	the indentity \eqref{eq:delta},
	and the fact that $e_1 = h$,
	\begin{align*}
		0
		=
		\ip{\nabla \Psi, \tau}
		=
		\Bip{\sum_{i=1}^d\partial_i \phi \otimes e_i \otimes v^*,
			 j |\tau|}
		 =
		 \sum_{i=1}^d
		\ip{v_i, v^*}
				  \bip{\partial_i \phi,
				   |\tau|}
		=
				 \bip{\partial_h \phi,
				  |\tau|}\, ,
	\end{align*}
	which is the desired identity.
\end{proof}

\section{The upper bound}
\label{sec:ub}
In the section, we prove the upper bound estimate for the $\Gamma$-convergence result. As usual, we will omit the $\omega$-dependence everywhere, as our result is of deterministic nature. The proof itself is mostly deterministic, the only random feature being the existence of certain limit, $\mathbb P$-a.e., due to the subadditive ergodic theorem. We fix $\mu \in \M(\overline \Omega;V)$ and  a sequence of bounded measures $m_\eps \in \M(\cX_\eps;V)$ so that $m_\eps \to \mu$ narrowly in $\M(\overline \Omega;V)$. We want to show that, for every tensor field $\nu\in \M(\overline \Omega;V\otimes \R^d)$ with $\dive \nu = \mu \in \M(\overline \Omega;V)$, we can find $J_\eps : \cE_\eps \stackrel{\rma}{\to} \R$ with $\DIVE J_\eps = m_\eps$,  $\iota_\eps J_\eps \to \nu$ narrowly  $\mP$-almost surely, and $F_\eps(J_\eps, \overline U) \to \bF_{\hom}(\nu,\overline U)$.

In the first part of this chapter, we show that continuous energy of a given flux $\nu$ can be approximated using smooth and compactly supported approximations of $\nu$. In the second part, we show how to construct recovery sequences for smooth vector fields and how to use the approximation result provided in the first part to show the existence of a recovery sequence in full generality.

\subsection{Smooth vector fields are dense in energy}

In the first step, we provide an approximation result for the continuous energy functional, which enables us to work with measures having a smooth and compactly supported density.

\begin{lemma}
\label{lemma: smooth}
	Let $U \subset \R^d$ be open, bounded, with Lipschitz boundary. Let $\nu\in \M(\overline U; V\otimes \R^d)$ and $\dive \nu = \mu \in \M(\overline U; V)$. Then there is a sequence $j_\rho \in C_c^\infty(U;V\otimes \R^d)$, $\dive j_\rho\in C_c^\infty(U;V)$ such that $j_\rho \Lm^d \to \nu$ narrowly in $\M(\overline{U}; V \otimes\R^d)$ and $\dive j_\rho \Lm^d \to \mu$ narrowly in $\M(\overline{U}; V)$ as $\rho \to 0$, as well as 
	\begin{align}	\label{eq:approx_cont_lemma}
		\lim_{\rho \to 0} \int_{U} f_{\hom}(j_\rho) \,dx = \bF_{\hom}(\nu)
            \, .
	\end{align}
\end{lemma}

\begin{rem}[General energy densities]
    The approximation described in Lemma~\ref{lemma: smooth} does not depend on the special form of $f_{\hom}$, but holds in general if one replaces $f_{\hom}$ with any energy density $f:V \otimes \R^d \to \R$ which satisfies the same properties (linear growth and Lipschitz continuity) as in Lemma~\ref{lemma:prop_fhom}.
\end{rem}

\begin{proof}
	The proof proceeds in two parts. First we replace $\nu$ with a possibly singular measure $\tilde \nu_\rho\in \M(U;V\otimes \R^d)$ with compact support in $U$. Then we mollify $\tilde \nu_\rho$ using convolutions to arrive at the sought measure with smooth density $j_\rho\in C_c^\infty(U;V\otimes \R^d)$.

    \smallskip
	\noindent
	\emph{Step 1 (reduction to compact support)}. \ 
    First we need to choose a suitable, locally bi-Lipschitz 
    $\Phi_\rho: \overline U \to U$ so that $\Phi_\rho(\overline U)$ is compact and $\Phi_\rho$ is sufficiently close to  the identity. Once given, we define 
    \begin{align}
		\tilde \nu_\rho := \big(\Phi_\rho)_{\#} (\nu \tD\Phi_\rho^T)\in \M(U;V\otimes \R^d)
			\, ,
	\end{align}
    or, in other words, the measure defined for every $\varphi \in C_b(\overline{U}; (V\otimes \R^d)^*)$ as
	\begin{align}	\label{eq:def_tildej}
		\int \varphi \de \tilde \nu_\rho
			=
		\int \big( \varphi \circ \Phi_\rho \big) \tD \Phi_\rho \de \nu
			\, .
	\end{align}
   The goal is to construct the map $\Phi_\rho$ in such a way that 
    \begin{itemize}
		\setlength{\itemsep}{5pt}
		\item [$(i)$] $\displaystyle \tilde \nu_\rho \to  \nu$ narrowly in $\M(\overline{U};V \otimes \R^d)$ as $\rho \to 0$.
		\item [$(ii)$] $\displaystyle \dive \tilde \nu_\rho \to  \mu$  narrowly in $\M(\overline{U};V)$ as $\rho \to 0$.
		\item [$(iii)$] We have the energy bound $\displaystyle \limsup_{\rho \to 0} \bF_{\hom}(\tilde \nu_\rho) \leq \bF_{\hom}(\nu)$.
	\end{itemize}
    	Note that $\nu$ is a divergence measure (in the sense of Definition~\ref{def:div_meas}) with $\dive \nu = \mu \in \M(\overline U; V)$, then also $\tilde \nu_\rho$ is a divergence measure and its divergence is given by $\dive \tilde \nu_\rho= (\Phi_\rho)_{\#} \mu \in \M(\overline U; V)$. Indeed, for any test function $\varphi\in C_c^\infty(\R^d;V^*)$, we have
	\begin{align}
		\langle \dive \tilde \nu_\rho, \varphi \rangle
			&=
		- \langle \tilde \nu_\rho, \tD \varphi \rangle
			=
		- \langle \nu, (\tD\varphi \circ \Phi_\rho) \tD\Phi_\rho \rangle
			\\ &=
		- \langle \nu, \tD(\varphi \circ \Phi_\rho) \rangle
			=
		\langle \dive \nu, \varphi \circ \Phi_\rho \rangle
			=
		\langle (\Phi_\rho)_{\#} \mu, \varphi  \rangle
	\end{align}
	which is the claimed equality.
    Moreover, we use the Radon--Nikodym decomposition $\nu = \frac{\de \nu}{\de x} \Lm^d + \frac{\de \nu}{\de |\nu|} |\nu|^s$ to then obtain
	\begin{align}	\label{eq:RN_decomp_tildej}
		\tilde \nu_\rho =
			\Big[
			\Big(
				\frac{\de \nu}{\de x} \tD \Phi_\rho^T
			\Big) \circ \Phi_\rho^{-1}
			\Big]
			(\Phi_\rho)_{\#} \Lm^d
		+
			\Big[
			\Big(
				\frac{\de \nu}{\de |\nu|} \tD \Phi_\rho^T
			\Big)  \circ \Phi_\rho^{-1}
			\Big]
			(\Phi_\rho)_{\#} |\nu|^s
				\, ,
	\end{align}
 Using the fact that $\Phi_\rho$ is locally bi-Lipschitz, we ensure that
    \begin{align}
    \label{eq:preservation_decomposition}
        (\Phi_\rho)_{\#} \Lm^d \ll \Lm^d
            \tand
        (\Phi_\rho)_{\#} |\nu|^s \perp (\Phi_\rho)_{\#} \Lm^d
            \, .
    \end{align}
    In particular, \eqref{eq:preservation_decomposition} ensures that \eqref{eq:RN_decomp_tildej} is a Lebesgue--Radon--Nikodym decomposition for the measure $\tilde \nu_\rho$.

   By the change of variables formula and by \eqref{eq:def_tildej} we can then compute
	\begin{align}
		\frac{\de \tilde \nu_\rho}{\de x}(\Phi_\rho(x)) = (\det \tD \Phi_\rho)^{-1}(x) \frac{\de \nu}{\de x}(x) \tD \Phi_\rho^T(x)
			\, , \qquad  \Lm^d\text{-a.e. } x \in \overline U \, ,
	\end{align}
	whereas in general (including the singular part), we simply obtain that
	\begin{align} 	\label{eq:density_tildej_singular}
		\frac{\de \tilde \nu_\rho}{\de |\tilde \nu_\rho|}
			\circ \Phi_\rho
			=
		\frac1
		{
		\Big|
			\frac{\de \nu}{\de |\nu|} \tD \Phi_\rho^T
		\Big|
		}
		\Big(
			\frac{\de \nu}{\de |\nu|} \tD \Phi_\rho^T
		\Big)
			\, .
	\end{align}

    We shall now construct a suitable $\Phi_\rho$: an application of the Lemma \ref{lemma: deformation} below provides the existence of a bi-Lipschitz deformation $\Phi_\rho:\overline U \to U$ which satisfies $\|\Phi_\rho(x)-x\|_\infty \leq \rho$ and $\|\tD\Phi_\rho - \mathrm{id}\|_\infty\leq \rho$.
        We then define 
	\begin{align}
		\tilde \nu_\rho := \big(\Phi_\rho)_{\#} (\nu \tD\Phi_\rho^T)\in \M(U;V\otimes \R^d)
			\, ,
	\end{align}
	or, in other words, the measure defined for every $\varphi \in C_b(U; V \otimes \R^d)$ as
	\begin{align}	
		\int \varphi \de \tilde \nu_\rho
			=
		\int \big( \varphi \circ \Phi_\rho \big) \tD \Phi_\rho \de \nu
			\, .
	\end{align}

	We claim that the newly obtained measure $\tilde \nu_\rho$ is quantitatively close to $\nu$, in the sense there exists a constant $C=C(U)\in\R_+$ such that, for every $\rho>0$,
	\begin{itemize}
		\setlength{\itemsep}{5pt}
		\item [$(i)'$] $\displaystyle \|\tilde \nu_\rho - \nu\|_{\KR(\overline U;V \otimes \R^d)} \leq C\rho|j|(\overline U)$,
		\qquad $(ii)'$ $\displaystyle \|\dive \tilde \nu_\rho - \mu\|_{\KR(\overline U;V)} \leq C\rho|\mu|(\overline U)$,
	\end{itemize}
    and that it satisfies the energy bound $(iii)$.
    
	To show $(i)'$, we take a $1$-Lipschitz test function $\varphi\in C(\overline U;V\otimes \R^d)$ so that $\| \varphi \|_\infty \leq 1$ and by means of a simple triangle inequality, from the very definition of $\tilde \nu_\rho$ we obtain (as usual, we use the duality notation between measures and continuous functions)
	\begin{align}	\label{eq:error_tildej}
		\begin{aligned}
			|\langle \tilde \nu_\rho - \nu, \varphi \rangle| = & |\langle \nu, (\varphi \circ \Phi_\rho) \tD\Phi_\rho - \varphi \rangle|\\
			 \leq & |\nu|(\overline U)(\|\varphi\circ \Phi_\rho - \varphi\|_\infty + \|\varphi\|_\infty\|\tD\Phi_\rho - \id\|_\infty)\\
			 \leq &
			 	|\nu|(\overline U)(\| \Phi_\rho - \id\|_\infty + \|\tD\Phi_\rho - \id\|_\infty)
			 	\leq  C \rho |\nu|(\overline U) \, ,
		\end{aligned}
	\end{align}
	where at last we used the property of $\Phi_\rho$. By taking the supremum over the test functions we obtained the sought bound in $\KR(\overline U)$.

	We take advantage of this, and for any $1$-Lipschitz test function $\varphi\in C_c^1(\R^d;V)$ with $\| \varphi\|_\infty \leq 1$, we write
	\begin{align} \label{eq:change_of_var}
		\begin{aligned}
			|\langle \dive \tilde \nu_\rho - \dive \nu, \varphi \rangle|
			 = & |\langle  \mu , \varphi \circ \Phi_\rho - \varphi \rangle|
			 	\leq
			 C\rho |\mu|(\overline U)
				\, ,
		\end{aligned}
	\end{align}
	where in the last inequality we proceeded exactly as in \eqref{eq:error_tildej}. This shows $(ii)'$.
	We are left to show $(iii)$. By the expansion formula of the determinant in terms of the trace, it is easy to see that $\|\tD\Phi_\rho - \id \|_\infty \leq \rho$ implies for every $\rho \in (0,1)$,
	\begin{align}	\label{eq:expansion_determ}
		\big|  \det \tD \Phi_\rho - 1 \big| \leq C \rho
			\, , \quad
		C=C(d) \in \R_+ \, .
	\end{align}
	Integrating \eqref{eq:change_of_var}, by \eqref{eq:expansion_determ}, and using the Lipschitz continuity of $f_{\hom}$ from Lemma \ref{lemma:prop_fhom}, we estimate the energy of absolutely continuous part  of $\tilde \nu_\rho$ as
	\begin{align}	
			 \int_U
			 	f_{\hom}
			 &	
                \Big(
			 	\frac{\de \tilde\nu_\rho}{\de y}
			 	\Big)
			 		\de y
			 =
			 \int_{\Phi_\rho(U)}
			 	f_{\hom}
			 	\Big(
			 		\Big(
				\frac{\de \nu}{\de x} \tD \Phi_\rho^T
			\Big) \circ \Phi_\rho^{-1}
			 	\Big)
			 		\dd y
                    +
            f_{\hom}(0) \Lm^d( U \setminus \Phi_\rho(U) )
      \\
				&=
			\int_U
				(\det \tD\Phi_\rho)
				f_{\hom}
				\Big(
					(\det \tD\Phi_\rho)^{-1}\frac{\de \nu}{\de x}D\Phi_\rho^T
				\Big)
					\de x
                   +
            f_{\hom}(0) \Lm^d( U \setminus \Phi_\rho(U) )
        \\
        \label{eq:energy_est_1}
		&\leq  \int_U ( 1 + C \rho )
			\Big(
				f_{\hom}
				\Big(\frac{\de \nu}{\de x}\Big)
					+
				C	\rho
				\Big|
					\frac{\de \nu}{\de x}
				\Big|
			\Big)	\de x
                + C f_{\hom}(0) \rho
				\, ,
	\end{align}
	for some constant $C=C(d,U) \in \R_+$.
	Now for the singular part: using the formula provided in \eqref{eq:RN_decomp_tildej} and \eqref{eq:density_tildej_singular} and the homogeneity and the Lipschitz continuity of $f_{\hom}^\infty$ (Lemma~\ref{lemma:prop_fhom}), we obtain
	\begin{align}	\label{eq:energy_est_2}
	\begin{aligned}
		\int_{\overline U} f_{\hom}^\infty
		\Big(
			\frac{\de \tilde \nu_\rho}{\de |\tilde \nu_\rho|}
		\Big)
			\de|\tilde \nu_\rho|^s
	=&
		\int_{\overline U} f_{\hom}^\infty
		\Big(
			\frac{\de \tilde \nu_\rho}{\de |\tilde \nu_\rho|}(\Phi_\rho(x))
		\Big)
		\Big|
			\frac{\de \nu}{\de |\nu|} \tD \Phi_\rho^T
		\Big|
	 		\de |\nu|^s\\
	=&
		\int_{\overline U} f_{\hom}^\infty
		\Big(
			\frac{\de \nu}{\de |\nu|} \tD \Phi_\rho^T
		\Big)
	 		\de |\nu|^s\\
	\leq&
		\int_{\overline U}
		\Big(
			f_{\hom}^\infty
		\Big(
			\frac{\de \nu}{\de |\nu|}
		\Big)
			+ C \rho
		\Big)
			\de |\nu|^s
		\, .
	\end{aligned}
	\end{align}
	All in all, \eqref{eq:energy_est_1} together with \eqref{eq:energy_est_2}, provides the upper bound
	\begin{align}	\label{eq:approx_proof}
		\bF_{\hom}(\tilde \nu_\rho)
			\leq
		(1+ C \rho) \bF_{\hom}(\nu)
			+ C \rho | \nu |(\overline U)
		\, .
	\end{align}
	This completes Step 1 of the proof.

	\smallskip
	\noindent
	\emph{Step 2 (regularisation)}: \
	We now replace $\tilde \nu_\rho\in \M(U;V\otimes \R^d)$, which is supported in a compact subset $K := \Phi(\rho(\overline U))$ of $U$ with $\rho_0:= \dist(K,\partial U)>0$, with its mollification $\tilde j_{\rho,{\rho'}} := \tilde \nu_\rho \ast \psi_{\rho'}\in C_c^\infty(U;V\otimes \R^d)$. Here $\psi_{\rho'}\in C_c^\infty(B(0,{\rho'}))$ is a standard mollifier and ${\rho'}\in(0,\rho_0)$, so that $\supp \tilde j_{\rho,{\rho'}}\in B(K,{\rho'}) \subset U$.

	As ${\rho'}\to 0$, we clearly have $\tilde j_{\rho,{\rho'}} \Lm^d \to \tilde \nu_\rho$ narrowly and $\dive \tilde j_{\rho,{\rho'}} \Lm^d \to \dive \tilde \nu_\rho$ narrowly.  
   Finally, the energy bound follows by classical continuity argument with respect to the convolution, see e.g. \cite[Corollary~2.11]{Baia-Chermisi-Matias-Santos:2013}, which yields $\bF_{\hom}(\tilde j_{\rho,{\rho'}} \Lm^d) \to \bF_{\hom}(\tilde \nu_\rho)$ as ${\rho'} \to 0$. 

\smallskip
	\noindent
	\emph{Step 3 (conclusion)}: \
	Finally, we conclude the proof of the Lemma by taking a suitable diagonal sequence $j_\rho := \tilde  j_{\rho,{\rho'}(\rho)}$ so that $j_\rho \Lm^d \to \nu$ narrowly and, by \eqref{eq:approx_proof}, so that \eqref{eq:approx_cont_lemma} is satisfied.
\end{proof}

\begin{lemma}[Deformation of domains with Lipschitz boundary]\label{lemma: deformation}
    Let $U\subset \R^d$ be open and bounded with Lipschitz boundary. Then for any $\rho>0$ there is a smooth bi-Lipschitz deformation $\Phi_\rho: \overline U \to U$ with $\|\Phi_\rho-\mathrm{id}\|_\infty \leq \rho$ and $\|\emph{D}\Phi_\rho - \mathrm{id}\|_\infty\leq \rho$.
\end{lemma}

\begin{proof}
Using compactness, cover $\overline U$ with finitely many rotated open cubes $(Q_i)_{i\in I}$ such that every cube is either contained in $U$ or $R_i^T(Q_i \cap U)-x_i = \{(x',x_d)\in Q_i'\times \R\,:\,0<x_d<h_i(x')\}$ for some rotation $R_i\in SO(d)$, some translation vector $x_i\in\R^d$ and a Lipschitz function $h_i:\R^{d-1}\to\R$. The outer unit normals $n_i = R_ie_d$ then have the additional property that $x-\alpha n_i\in U$ for all $x\in \overline U \cap B(Q_i,r_i)$ and all $\alpha\in[0,r_i]$ for some $r_i>0$.

Pick a partition of unity $\eta_i\in C_c^\infty(Q_i)$ such that $\sum_{i\in I}\eta_i(x) = 1$ for all $x\in\overline U$. Define the global deformation $\Phi_\rho\in C^\infty(\R^d;\R^d)$,
\[
\Phi_\rho(x) := x - \rho \sum_{i\in I'} \eta_i(x) n_i,
\]
where $I'\subseteq I$ is the index set of cubes intersecting $\partial U$. First note that for $\rho\leq \max_{i\in I'} r_i / |I'|$, we have $\Phi_\rho(x)\in U$ for all $x\in\overline U$.

By construction,
\begin{equation}\label{eq: L infty bound}
|\Phi_\rho(x)-x| \leq \rho \sum_{i\in I'}\eta_i(x)\leq \rho,
\end{equation}
and
\begin{equation}\label{eq: L infty bound derivative}
|\tD\Phi_\rho(x)- \mathrm{id}| = |\rho \sum_{i\in I} R_ie_d \otimes \nabla \eta_i(x)| \leq C\rho.
\end{equation}

Using \eqref{eq: L infty bound derivative}, we see that
\begin{equation}
|\Phi_\rho(x) - \Phi_\rho(y)| \geq |x-y| - C\rho|x-y|,
\end{equation}
which implies global injectivity of $\Phi_\rho$ and Lipschitz continuity of its smooth inverse as long as $C\rho<1$.
\end{proof}

\subsection{Proof of the upper bound}
In this section we take advantage of the approximation result provided by Lemma~\ref{lemma: smooth} to show the validity of the limsup inequality in our main theorem.

\begin{proof}
Given $\nu\in \M(\overline U;V\otimes \R^d)$ with divergence $\dive \nu= \mu\in M(\overline U;V)$ and discrete measures $m_\eps\in \M(\cX_\eps;V)$ with 
	$m_\eps \to \mu$
narrowly in $\M(\overline{U};V)$
we wish to find $J_\eps\in V_\rma^{\cE_\eps}$ with $\DIVE J_\eps = \mu_\eps$ such that $\iota_\eps J_\eps \to \nu$ narrowly to $\M(\overline{U};V \otimes \R^d)$ 
and $F_\eps(J_\eps,\overline U)\to \bF_{\hom}(\nu,\overline U)$.

Let us describe first how to construct a recovery sequence in the case when $U$ is a bounded, Lipschitz domain.

We shall proceed in five steps:
\begin{itemize}
	\item [Step 1:] We replace $\nu$ by a smooth $j_\rho\in C_c^\infty(U;V\otimes \R^d)$ that is energy-close to $\nu$. 
	\item [Step 2:] We discretize $j_\rho$ in dyadic cubes at scale $\delta\in (0,\dist(\supp j_\rho, \partial U))$ to obtain a piecewise-constant $j_{\rho,\delta}:U\to V\otimes \R^d$ that is energy-divergence close to $j_\rho$ and thus $\nu$.
	\item [Step 3:] From $j_{\rho,\delta}$, we construct and glue optimal microstructures on each cube of size $\delta$ to build $\tilde J_{\rho,\delta,\eps}\in V_\rma^{\cE_\eps}$ that is divergence-close to $j_{\rho,\delta}$ and has $\limsup_{\eps\to 0} F_\eps(\tilde J_{\rho,\delta,\eps},\overline{U})\leq \bF_{\hom}(j_{\rho,\delta} \Lm^d,\overline{U})$.
	\item [Step 4:] We find a corrector $K_{\rho,\delta,\eps}\in V_\rma^{E_\eps}$ solving $\DIVE K_{\rho,\delta,\eps} = \mu_\eps - \DIVE \tilde J_{\rho,\delta,\eps}$ with total variation $|K_{\rho,\delta,\eps}| \leq C\rho + C(\rho)\delta + C(\rho,\delta)\eps$.
	\item [Step 5:] Closing arguments: we show that $J_{\rho,\delta,\eps} := \tilde J_{\rho,\delta,\eps} + K_{\rho,\delta,\eps}$ has the right properties and choose a diagonal sequence $\delta(\eps),\rho(\eps)\to 0$.
\end{itemize}

\noindent
\emph{Step 1}: \
For every $\rho>0$, we seek  $j_\rho\in C_c^\infty(U;V\otimes \R^d)$ satisfying, with $\nu_\rho:= j_\rho \Lm^d$, the properties
\begin{equation}\label{eq: mollification estimates}
	\begin{cases}
		\|\nu_\rho  - \nu \|_{\tKR(\overline U)} \leq \rho
                \, ,
  \\
		\|\dive \nu_\rho - \mu \|_{\tKR(\overline U)} \leq \rho
                \, ,
  \\
		\bF_{\hom}(\nu_\rho, \overline U) \leq \bF_{\hom}(\nu ,\overline U) + \rho
            \, ,
	\end{cases}
\end{equation}
This is done by applying Lemma \ref{lemma: smooth}, and use that narrow convergence is equivalent to $\KR$ convergence on compact sets.

\noindent
\emph{Step 2}: \
We fix $\delta\in(0,\frac1{2\sqrt{d}}\dist(\supp j_\rho, \partial U))$ 
and cover the the domain $U$ with finitely many cubes disjoint cubes 
	$Q_z := z+[0,\delta)^d$, where $z\in \delta\Z^d$. 
We now consider any piecewise constant $j_{\rho,\delta}:U\to V\otimes \R^d$ satisfying
\begin{align}
\label{eq:piecewise}
    j_{\rho,\delta}(x)  = j_{\rho,\delta,z}
            \quad\text{ for }x \in Q_z
                \, , \quad \text{so that} \quad
            \sup_{x \in Q_z}
            \| j_\rho(x) - j_{\rho,\delta,z} \|
                \leq
            C(\rho) \delta
                \, .
\end{align}
For example, $j_\rho$ being Lipschitz, we can also take $j_{\rho,\delta,z}:= j_\rho(z)$, for every $z \in \Z^d$. From now on, the constant $C(\rho)$ might change line by line.

We set $\nu_{\rho,\delta}:= j_{\rho,\delta} \Lm^d \in \calM(\overline{U}; V \otimes \R^d)$. Note that $\supp j_{\rho,\delta}\subset U$ by our choice of $\delta$. We claim that
\begin{equation}\label{eq: discretization estimates}
	\begin{cases}
		\|\nu_{\rho,\delta} - \nu_\rho\|_{\tKR(\overline U)} \leq C(\rho)\delta
			\, ,
	\\
 \displaystyle
		\|\dive \nu_{\rho,\delta} - \dive \nu_\rho\|_{\tKR(\overline U)} \leq C(\rho)\delta
			\,  ,
	\\
 \displaystyle
		\bF_{\hom}(\nu_{\rho,\delta} , \overline U) \leq \bF_{\hom}(\nu_\rho , \overline U) + C(\rho)\delta
			\, .
	\end{cases}
\end{equation}
The first inequality trivially holds by construction.
Concerning the second bound in \eqref{eq: discretization estimates}, we fix $\psi \in C^1(\overline U ; V^*)$ with 
	$\Lip(\psi) \leq 1$. 
Note that, due to the fact that $j_\rho$ is compactly supported, 
there exists $\tilde \psi \in C_c^1(\overline U ; V^*)$ 
with $\tilde \psi = \psi$ on $\Lip(\tilde \psi) \leq  \tilde  C(\rho)$ 
on $\supp j_\rho \cup \supp j_{\rho,\delta}$. 
As a consequence, we obtain
\begin{align}
    \big|
    \langle
        \dive j_{\rho,\delta} - \dive j_\rho
            ,
        \psi
    \rangle
    \big|
        =
    \int
    \big|
    \langle
        \nabla \tilde \psi
            ,
        j_{\rho,\delta} - j_\rho
    \rangle
    \big|
    \dd x
        \leq
    C(\rho) \tilde C(\rho) \delta
        \, ,
\end{align}
where we used $\| j_{\rho,\delta} - j_\rho \|_\infty \leq C(\rho) \delta$, uniformly in $\psi$. This shows the second inequality in  \eqref{eq: discretization estimates}. The third inequality directly follows as well from the $L^\infty$-bound and the Lipschitz property of $f_{\hom}$.

\smallskip
\noindent
\emph{Step 3}: \
The next step is to define a global competitor by \textit{gluing} together near-optimal microstructures. This procedure necessarily creates extra divergence, that we shall control thanks to the property of the uniform-flow operator. It is useful to introduce a layer, of size $\eta>0$, between the cubes $\{Q_z\}_z$ and perform a continuous interpolation between the near-optimal microstructures. A similar argument will also appear in the proof of the lower bound, see in particular the proof of Step 1 in Proposition~\ref{prop:asymptotic_strip}.

For every $\eta \in (0,\delta/4)$, we consider the smaller cubes $Q_{\eta,z} := \{ x \in Q_z \suchthat d(x, \partial Q_z) \geq \eta \} \subset Q_z$. Let $\{ 0 \leq \psi_{z,\eta} \in C_c^\infty(\R^d) \, : \, z \in \delta \Z^d \}$ be a family of smooth functions satisfying the following properties:
\begin{itemize}
    \item $\psi_{z,\eta} = 1$ on $Q_{2\eta,z}$ and $\psi_{z,\eta} \leq 1$ everywhere.
    \item $\psi_{z,\eta} =0$ on the complement of the larger cube
    \begin{align}
        \tilde Q_{\eta,z} := B_{\| \cdot \|_\infty}(Q_z, \eta) =
        \left\{
            x \in \R^d \suchthat \exists y \in Q_z\, , \, \| x - y \|_\infty \leq \eta
        \right\}
            \, .
    \end{align}
    In particular, $\psi_{z,\eta} =0$ on $Q_{2\eta,z'}$, for every $z'\neq z$.
    \item The family $\{ \psi_{z,\eta} \}_z$ is a partition of unity on $\R^d$, namely
    \begin{align}
    \label{eq:partition_unity}
        \sum_{z \in \delta \Z^d}
            \psi_{z,\eta}(x)
        = 1
            \, , \qquad
        \forall x \in \R^d
            \, , \quad
        \forall \eta \in (0, \delta/4)
                \, .
    \end{align}
    \item The gradients are bounded by 
		$\| \nabla \psi_{\eta,z} \|_\infty \leq C \eta^{-1}$ 
		for every $z \in \delta \Z^d$. Moreover, as $\eta \to 0$, 
		we have the convergence
    \begin{align}
    \label{eq:cutoff_deriv_convergence}
        \nabla \psi_{\eta,z} \Lm^d
            \to
        \tD \1_{Q_z} = \ext \Hm^{d-1} \res (\partial Q_z)
            \qquad \text{narrowly in }\calM(\R^d ; \R^d)
                \, ,
    \end{align}
	where here $\tD$ denotes the distributional derivative.
\end{itemize}
An example of a family which does the job is given by
\begin{align}
\label{eq:possible_choice_cutoff_ub}
    \psi_{\eta,z} := \rho_\eta * \1_{Q_z}
        \, , \quad
    \forall z \in \delta \Z^d
        \, ,
\end{align}
where $\rho_\eta \in C_c(B_\eta(0))$ is a smooth mollifier.
Since $\omega$ was chosen so that $f_{\hom}$ exists, 
we can find for every $z\in \delta\Z^d$ a sequence of 
near-optimal admissible microstructures 
	$J_{\rho,\delta,z,\eps}\in \Rep_{\calR ,\eps} 
		(j_{\rho,\delta,z}, Q_{2\eta,z})$ 
such that
\begin{equation}
\label{eq:ub_step3_optimalmicro}
    \lim_{\eps \to 0}
        F_\eps(J_{\rho,\delta,z,\eps},Q_{2\eta,z})
            =
        \Lm^d(Q_{2\eta,z}) f_{\hom}(j_{\rho,\delta, z})
            =
        \bF_{\hom}(\nu_{\rho,\delta},  Q_{2\eta,z})
            \, .
\end{equation}
Recall the notation introduced in \eqref{eq:def_hat_0}. 
We define the global competitor 
	$\tilde J_{\rho,\delta,\eps} \in V_\rma^{E_\eps}$ as
\begin{align}
    \tilde J_{\rho,\delta,\eps} 
	:=
    \sum_{z \in I_\delta}
        \hat \psi_{z,\eta} \cdot J_{\rho,\delta,z,\eps}
    \, , \qquad \text{where} \quad
        I_\delta 
		:= 
		\{ z \in \delta \Z^d \, : \, Q_z \subset U \}
         \, .
\end{align}
By construction, we have that $ \tilde J_{\rho,\delta,\eps}$ is supported in $U$, and
it coincides with $J_{\rho,\delta,z,\eps}$ on $Q_{2\eta,z}$, for every $z \in \delta \Z^d$.
We claim that we have
\begin{align}\label{eq: micro estimates}
	\begin{cases}
 \displaystyle
		\limsup_{\eps \to 0}\|\iota_\eps \tilde J_{\rho,\delta,\eps} - \nu_{\rho,\delta}\|_{\tKR(\overline U)} \leq C(\rho)(\delta + \eta)
            \, ,
\\
  \displaystyle
		\limsup_{\eps \to 0}\|\dive \iota_\eps  \tilde J_{\rho,\delta,\eps} - \dive \nu_{\rho,\delta}\|_{\tKR(\overline U)}
      = g_{\delta,\rho}(\eta)
    \, ,
\\
  \displaystyle
		\limsup_{\eps \to 0} F_\eps(\tilde J_{\rho,\delta,\eps}, \overline U) \leq \bF_{\hom}(\nu_{\rho,\delta}, \overline U) + C \delta +   C(\rho) \eta 
            \, ,
    \end{cases}
\end{align}
where $g_{\delta,\rho}(\eta) \to 0$ (possibly depending on $\delta$ 
and $\rho$
\footnote{If one chooses the cutoff functions 
as in \eqref{eq:possible_choice_cutoff_ub}, then we can take 
	$g_{\delta,\rho} = C(\rho) \eta/\delta$.}) as $\eta \to 0$.
To show the first point, we employ Lemma~\ref{lem:equal-mass} and Remark~\ref{rem:vague_J0_cubes} to obtain that
\begin{align}\label{eq: reduced cube mass}
    \iota_\eps J_{\rho,\delta,z,\eps}(Q_{2\eta,z})
        =
    \iota_\eps \calR_\eps (j_{\rho,\delta,z})(Q_{2\eta,z})
        \xrightarrow[\eps \to 0]{}
    j_{\rho,\delta,z} \Lm^d(Q_{2\eta,z})
        =
    \nu_{\rho,\delta}(Q_{2\eta,z})
        \, ,
\end{align}
for every $z \in \delta \Z^d$.
Using that $J_{\rho,\delta,z,\eps} \in \Rep_{\eps,\calR}(j_{\rho,\delta,z}, Q_{2\eta, z})$, we write
\begin{align}
    \tilde J_{\rho,\delta,\eps}
        &=
    \sum_{z \in I_\delta}
        \1_{Q_{2\eta,z}}  J_{\rho,\delta,z,\eps}
            +
    \sum_{z \in I_\delta}
        \big(  \psi_{z,\eta} - \1_{Q_{2\eta,z}}  \big)  J_{\rho,\delta,z,\eps}
\\
\label{eq:ub_step3_split}
        &=
    \sum_{z \in I_\delta}
        \1_{Q_{2\eta,z}}  J_{\rho,\delta,z,\eps}
            +
    \sum_{z \in I_\delta}
        \big(  \psi_{z,\eta} - \1_{Q_{2\eta,z}}  \big)  \calR_\eps (j_{\rho,\delta,z})
            \, .
\end{align}
Note that $x \mapsto \psi_{z,\eta}(x) - \1_{Q_{2\eta,z}}(x)$ 
is a compactly supported function whose set of discontinuities 
has $\Lm^d$-measure zero. 
As a consequence, by Remark~\ref{rem:abs_cont_Reps}, Remark~\ref{rem:localisation_weakconv}, and Proposition~\ref{prop:equivalence_convergence_KR},
we have that as $\eps \to 0$,
\begin{align}
\label{eq:ub_step3_KR1}
    \Bigl\| 
		\iota_\eps
    \big(  \psi_{z,\eta} - \1_{Q_{2\eta,z}}  \big)  \calR_\eps (j_{\rho,\delta,z})
        -
     \big(  \psi_{z,\eta} - \1_{Q_{2\eta,z}}  \big)
        j_{\rho,\delta,z} \Lm^d
	\Bigr\|_{\tKR(\overline U)}
	\to 0
        \, .
\end{align}
On the other hand, for every $1$-Lipschitz test function $\phi\in C(\overline U; V^* \otimes \R^d)$ with $\phi(0)=0$,
\begin{align}
\label{eq:ub_step3_test1}
    \Big\langle
    \sum_{z \in I_\delta}
        \1_{Q_{2\eta,z}}
             \iota_\eps J_{\rho,\delta,z,\eps}
     - j_{\rho,\delta}
        ,
    \phi
    \Big\rangle
        =
     \sum_{z \in I_\delta}
        \big\langle
            \iota_\eps J_{\rho,\delta,z,\eps} - j_{\rho,\delta,z}
                ,
            \phi   \1_{Q_{2\eta,z}}
        \big\rangle
            -
         \langle
            \1_{S_\eta^c} j_{\rho,\delta} , \phi
        \rangle
            \,  , \quad
\end{align}
where we used the notation $S_\eta$ to denotes the subset of $\R^d$ given by the union of all the sets $\{Q_{2\eta,z} \suchthat z \in I_\delta\}$.
We estimate the two terms in \eqref{eq:ub_step3_test1} one at a time: concerning the first one, we observe that it coincides with
\begin{align}
     \sum_{z \in I_\delta}
        \big(
            \iota_\eps J_{\rho,\delta,z,\eps}(Q_{2\eta,z}) - \nu_{\rho,\delta}(Q_{2\eta,z})
        \big)
        \cdot
            \phi(z)
            +
     \sum_{z \in I_\delta}
        \langle
            \iota_\eps J_{\rho,\delta,z,\eps} - j_{\rho,\delta,z}
                ,
            \big( \phi - \phi(z) \big)    \1_{Q_{2\eta,z}}
        \rangle
            \, ,
\end{align}
where in the first equality we used that 
$j_{\rho,\delta} = 0$ on $Q_{2\eta,z}$ for every $z \notin I_\delta$, 
which follows from the fact that 
$\sqrt{d} \delta < \dist(\supp(j_{\rho,\delta}), \partial U)$. 
Using \eqref{eq: reduced cube mass}, we see that 
the first term in the right-hand side above goes to zero in $\eps \to 0$ 
uniformly over the test functions $\phi$, 
whereas the test functions $\phi-\phi(z)\mathds{1}_{Q_{2\eta,z}}$ 
are uniformly bounded by $C\delta$.
Therefore,
\begin{align}
    \limsup_{\eps \to 0}
    \Big\|
        \sum_{z \in I_\delta}
        \1_{Q_{2\eta,z}}
             \iota_\eps J_{\rho,\delta,z,\eps}
        &- j_{\rho,\delta} \1_{S_\eta}
    \Big\|_{\KR(\overline U)}
        \leq
    C \delta \limsup_{\eps \to 0}
    \sum_{z \in I_\delta}
            \big|
                \iota_\eps J_{\rho,\delta,z,\eps} - j_{\rho,\delta,z}
            \big|(Q_{2\eta,z})
\\
        &\leq
    C \delta \limsup_{\eps \to 0}
        \sum_{z \in I_\delta}
        \big|
            \iota_\eps J_{\rho,\delta,z,\eps}
        \big|(Q_{2\eta,z})
            +
        \big|
            \nu_{\rho,\delta}(\overline{U})
        \big|
\\
        &\leq
    C c_2 \delta
    \big(
        \sum_{z \in I_\delta}
            \bF_{\hom}(\nu_{\rho,\delta}, Q_{2\eta,z})
        + \nu_{\rho,\delta}(\overline{U})
    \big)
\\
\label{eq:ub_step3_KR2}
        &\leq
    C c_2\delta
    \big(
        \bF_{\hom}(\nu_{\rho,\delta}, \overline{U})
            +
        \nu_{\rho,\delta}(\overline{U})
    \big)
                \, ,
\end{align}
where for going from the second to the third line we used (F2) as well as \eqref{eq:ub_step3_optimalmicro}.

Collecting the estimates in \eqref{eq:ub_step3_test1}, \eqref{eq:ub_step3_KR1}, \eqref{eq:ub_step3_split}, and finally \eqref{eq:ub_step3_KR2} we obtain
\begin{align}
    \limsup_{\eps \to 0}
        \|
            \iota_\eps \tilde J_{\rho,\delta,\eps} - j_{\rho,\delta}
        \|_{\tKR(\overline U)}
            &\leq
        C(\rho) \delta +
        \Big\|
            \sum_{z \in I_\delta}
            \big(
                \psi_{z,\eta} - \1_{Q_{2\eta,z}}
            \big)
                j_{\rho,\delta,z} \Lm^d
            - \1_{S_\eta^c} j_{\rho,\delta}
        \Big\|_{\tKR(\overline U)}
\\
            &=
        C(\rho) \delta +
        \Big\|
            \sum_{z \in I_\delta}
            \big(
                \psi_{z,\eta} - 1
            \big)
                j_{\rho,\delta,z} \Lm^d
        \Big\|_{\tKR(\overline U)}
\\
            &\leq C(\rho) (\delta + \eta)
    \, ,
\end{align}
which shows the first inequality in \eqref{eq: micro estimates}.

We now turn our attention to the second inequality in \eqref{eq: micro estimates}, involving the divergence $\tilde J_{\rho,\delta,\eps}$: using the notation from Section~\ref{sec:preliminaries_derivatives}, by the discrete Leibniz rule  \eqref{eq:Leibniz} we have
\begin{align}
\label{eq:DIVE_tildeJ_ub}
    \DIVE \tilde J_{\rho,\delta,\eps}
        =
    \sum_{z \in I_\delta}
        \grad \psi_{z,\eta}
            \star
         J_{\rho,\delta,z,\eps}
            \, ,
\end{align}
where we also used that $\DIVE J_{\rho,\delta,z,\eps} \equiv 0$ on $\cX_\eps$, for every $z \in \delta \Z^d$.
Using that, for every $z \in \delta \Z^d$, 
the total variation $\{ \nabla \psi_{\eta,z} \cdot (\iota_\eps J_{\rho,\delta,z,\eps}) \}_\eps$ is bounded uniformly in $\eps$, an application of Lemma~\ref{lemma:gradients} shows that
\begin{align}
\label{eq:ub_step3_dive}
    \limsup_{\eps \to 0}
    \Big\|
        \iota_\eps
        \big(
            \grad \psi_{z,\eta}
                \star
             J_{\rho,\delta,z,\eps}
        \big)
            -
        \nabla \psi_{\eta,z}
            \cdot
        ( \iota_\eps J_{\rho,\delta,z,\eps} )
    \Big\|_{\tKR(\overline{U})}
        = 0
        \, .
\end{align}
On the other hand, by construction, on the set where $\nabla \psi_{\eta,z}$ does not vanish, then $J_{\rho,\delta,z,\eps}$ coincides with $\calR_\eps j_{\rho,\delta,z}$. 
Using the convergence of the uniform-flow operator 
(note that $\nabla \psi_{\eta,z}$ is smooth and compactly supported), 
from \eqref{eq:ub_step3_dive} and \eqref{eq:DIVE_tildeJ_ub} we conclude that
\begin{align}
\label{}
    \limsup_{\eps \to 0}
    \Big\|
        \dive \iota_\eps \tilde J_{\rho,\delta,\eps}
            -
        \sum_{z \in I_\delta}
            \nabla \psi_{\eta,z}
                \cdot
            j_{\rho,\delta,z} \Lm^d
    \Big\|_{\tKR(\overline{U})}
        = 0
        \, .
\end{align}
In order to conclude the proof of the claimed inequality, it is enough to observe that \eqref{eq:cutoff_deriv_convergence} yields
\begin{align}
     \sum_{z \in I_\delta}
            \nabla \psi_{\eta,z}
                \cdot
            j_{\rho,\delta,z} \Lm^d
        \xrightarrow[\eta \to 0]{}
    \sum_{z \in I_\delta}
       \ext  \cdot  j_{\rho,\delta,z}
            \Hm^{d-1} \res_{Q_z}
                =
            \dive j_{\rho,\delta}
        \, ,
\end{align}
narrowly in $\calM(\overline{U}; V)$.

Finally, we estimate the energy: 
using the additivity of the energy we split into two contributions, 
the bulk terms and a boundary one, as
\begin{align}
\label{eq:ub_step3_energy_split}
    F_\eps(\tilde J_{\rho,\delta,\eps}, \overline U)
        =
    \sum_{z \in I_\delta}
    F_\eps(\tilde J_{\rho,\delta,\eps}, Q_{2\eta,z})
        +
    F_\eps
    \big(
        \tilde J_{\rho,\delta,\eps}, D_{\delta,\eta}
    \big)
        \, ,
\end{align}
where for simplicity we used the notation $D_{\delta,\eta}:= \overline U \setminus \bigcup_{z \in I_\delta} Q_{2\eta,z}$.
Concerning the bulk, using that $\{\psi_{\eta,z}\}_z$ is a partition of unity and the very definition of $\tilde J_{\rho,\delta,\eps}$ we observe that
\begin{align}
\label{eq:diff_micro_global}
    \tilde J_{\rho,\delta,\eps}
        -
    J_{\rho,\delta,z,\eps}
        =
    \begin{cases}
    \displaystyle
        0
            &\text{on } Q_{2\eta,z} \, ,
    \\
    \displaystyle
        \sum_{z' \in I_\delta}
            \psi_{\eta,z'}
            \calR_\eps (j_{\rho,\delta,z'} - j_{\rho,\delta,z})
                &\text{on }
                    Q_z\setminus Q_{2\eta,z} \, .
    \end{cases}
\end{align}
Note also that on $Q_z \setminus Q_{2\eta,z}$, all $\psi_{\eta,z'}=0$ for all $z'$ but the ones satisfying $\|z' - z\|_{\ell_\infty} = \delta$ (which are finitely many). In particular, using the boundedness of $\calR_\eps$ we have that
\begin{align}
    \Big|
        \sum_{z' \in I_\delta}
            \psi_{\eta,z'}
            \calR_\eps (j_{\rho,\delta,z'} - j_{\rho,\delta,z})
    \Big|
    &
    \big(
       Q_z \setminus Q_{2\eta,z}
    \big)
        \leq
        \hspace{-3mm}
    \sum_{
        \substack{ z' \in I_\delta \\  
				\|z' - z\|_{\ell_\infty} =\, \delta }
        }
        \hspace{-4mm}
        \psi_{\eta,z'}
        \big|
            \calR_\eps (j_{\rho,\delta,z'} - j_{\rho,\delta,z})
        \big|
        \big(
           Q_z \setminus Q_{2\eta,z}
        \big)
\\
        &\leq
        C         
		 \hspace{-3mm}
    \sum_{
        \substack{ z' \in I_\delta \\ 
				 \|z' - z\|_{\ell_\infty} =\, \delta }
        }
        \hspace{-4mm}
        \psi_{\eta,z'}
        | j_{\rho,\delta,z'} - j_{\rho,\delta,z} |
        \Lm^d
        \big(
            Q_z \setminus Q_{2\eta,z}
        \big)
        \leq
            C(\rho) \eta \delta^d
                \, .
\end{align}
for $\eps >0$ small enough. Consequently, from this estimate, \eqref{eq:diff_micro_global}, the Lipschitz properties \eqref{eq:Lipschitz-F-eps} of $F_\eps$, and the fact that, for $\eps>0$ small enough $B(Q_{2\eta,z},\eps R_\Lip) \subset Q_z \setminus Q_{2\eta,z}$, we infer that
\begin{align}
    \sum_{z \in I_\delta}
        F_\eps(\tilde J_{\rho,\delta,\eps}, Q_{2\eta,z})
        &\leq
    \sum_{z \in I_\delta}
        F_\eps(J_{\rho,\delta,z,\eps}, Q_{2\eta,z})
            +
        \big|
            \iota_\eps (\tilde J_{\rho,\delta,\eps} - J_{\rho,\delta,z,\eps})
        \big|( B(Q_z,\eps R_\Lip) )
\\
        &\leq
    \bigg(
    \sum_{z \in I_\delta}
        F_\eps(J_{\rho,\delta,z,\eps}, Q_{2\eta,z})
    \bigg)
            +
        C(\rho) \eta
            \, .
\end{align}
Taking the limsup in $\eps \to 0$ and using the property \eqref{eq:ub_step3_optimalmicro} of the microstructures, we find that
\begin{align}
    \limsup_{\eps \to 0}
    \sum_{z \in I_\delta}
        F_\eps(\tilde J_{\rho,\delta,\eps}, Q_{2\eta,z})
    &\leq
        \bF_{\hom}
        \bigg(
            \nu_{\rho,\delta}
                ,
            \bigcup_{z \in I_\delta} Q_{2\eta,z}
        \bigg)
            +C(\rho) \eta
\\
\label{eq:ub_energy_bulk}
    &\leq
        \bF_{\hom}(\nu_{\rho,\delta}, \overline U)
            +C(\rho) \eta
        \, ,
\end{align}
where at last we also used that $\bF_{\hom}$ is nonnegative. Concerning the boundary contribution in \eqref{eq:ub_step3_energy_split}, we note that by construction
\begin{align}
\label{eq:diff_global_boundary}
    \tilde J_{\rho,\delta,\eps}
        =      0
    \qquad \text{on} \qquad
        B
        \big(
            D_{\delta,\eta}
                \, , \,
            \eps R_\Lip
        \big)
            \, ,
\end{align}
for $\eps$ small enough (here we also used that $\calR_\eps (0) = 0$). Using this fact, by locality (Remark~\ref{rem:locality}) and the linear growth condition (F1) we see that
\begin{align}
    F_\eps ( \tilde J_{\rho,\delta,\eps}, D_{\delta,\eta})
        =
    F_\eps ( 0, D_{\delta,\eta} )
        \leq
    C \Lm^d(D_{\delta,\eta})
        \leq
    C \delta
        \, .
\end{align}

Putting this together with \eqref{eq:ub_step3_energy_split} and \eqref{eq:ub_energy_bulk} we conclude the proof of \eqref{eq: micro estimates}.

\smallskip
\noindent
\emph{Step 4}: \
Combining \eqref{eq: mollification estimates}, \eqref{eq: discretization estimates}, and \eqref{eq: micro estimates}, together with $m_\eps \to \mu$ in $\tKR$, we see that
\begin{align}
\label{eq:ub_step5_pre_corrector_bound}
    \limsup_{\eps \to 0}
    \|
        \mu_\eps - \dive \iota_\eps \tilde J_{\rho,\delta,\eps}
    \|_{\tKR(\overline U)}
        \leq
    \rho + C(\rho)\delta + g_{\delta,\rho}(\eta)
        \, .
\end{align}
We apply Proposition~\ref{prop:correctors} to $m:=m_\eps -  \DIVE \tilde J_{\rho,\delta,\eps}$ and find correctors $K_{\rho,\delta,\eps}\in V_\rma^{\cE_\eps}$ so that $\DIVE K_{\rho,\delta,\eps} = \mu_\eps - \DIVE \tilde J_{\rho,\delta,\eps}$ and
\begin{equation}
\label{eq:ub_step5_corrector_bound}
    |\iota_\eps J_{\rho,\delta,\eps}|(\overline U)
        \leq
    C
    \Big(
        \|\mu_\eps - \dive \iota_\eps \tilde J_{\rho,\delta,\eps}\|_{\tKR(\overline U)}
            +
        \eps
        \big|
            \mu_\eps - \dive \iota_\eps \tilde J_{\rho,\delta,\eps}
        \big|(\overline U)
    \Big)
        \, .
\end{equation}
We finally define the candidate recovery sequence as
\begin{align}
    J_{\rho,\delta,\eps}:=
        \tilde J_{\rho,\delta,\eps}
            +
        K_{\rho,\delta,\eps}
            \in V_\rma^{E_\eps}
        \, .
\end{align}
Note that the second term on the right-hand side of \eqref{eq:ub_step5_corrector_bound} is vanishing as $\eps \to 0$, due to the fact that
\begin{align}
    \sup_{\eps >0 }
    \big|
        \mu_\eps - \dive \iota_\eps \tilde J_{\rho,\delta,\eps}
    \big|(\overline U)
        \leq
    \sup_{\eps >0 }
    \Big(
    \big|
        \mu_\eps
    \big|(\overline U)
        +
    \big|
        \dive \iota_\eps \tilde J_{\rho,\delta,\eps}
    \big|(\overline U)
    \Big)
        \leq C(\rho, \eta)
            \, ,
\end{align}
which follows from the fact that $\mu_\eps \to \mu$ narrowly and from the explicit form of the divergence of $\tilde J_{\rho,\delta,\eps} $ given in \eqref{eq:DIVE_tildeJ_ub}\footnote{In fact, it is possible to show that the constant $C(\rho, \eta)$ can be chosen independent of $\eta$.}.
In particular, from \eqref{eq:ub_step5_corrector_bound} and \eqref{eq:ub_step5_pre_corrector_bound} we can control the total variation of the correctors as
\begin{align}
    \limsup_{\eps \to 0}
        |\iota_\eps J_{\rho,\delta,\eps}|(\overline U)
    \leq
        C
        \big(
             \rho + C(\rho)\delta + g_{\delta,\rho}(\eta)
        \big)
            \, .
\end{align}
Consequently, by the Lipschitz property \eqref{eq:Lipschitz-F-eps} of $F_\eps$ we readily check that
\begin{align}
\label{eq: global estimates}
    \begin{cases}
    \displaystyle
        \limsup_{\eps \to 0}
                \|
                    \iota_\eps  J_{\rho,\delta,\eps} - j
                \|_{\tKR(\overline U)}
            \leq
       C
        \big(
             \rho + C(\rho)\delta + g_{\delta,\rho}(\eta)
        \big)
            \, ,
    \\
    \displaystyle
        \dive \iota_\eps  J_{\rho,\delta,\eps} = \mu_\eps
            \, ,
    \\
    \displaystyle
        \limsup_{\eps \to 0} F_\eps(J_{\rho,\delta,\eps}, \overline U)
            \leq
        \limsup_{\eps \to 0} F_\eps(\tilde J_{\rho,\delta,\eps}, \overline U)
            +
      C
        \big(
             \rho + C(\rho)\delta + g_{\delta,\rho}(\eta)
        \big)
            \, .
    \end{cases}
\end{align}

\smallskip
\noindent
\emph{Step 5}: \
Combining the estimates obtained in Step 5 in \eqref{eq: global estimates}, recalling the energy estimates provided in \eqref{eq: mollification estimates}, \eqref{eq: discretization estimates}, and \eqref{eq: micro estimates}, we can choose a diagonal sequence $\eta(\eps) \to 0$, $\delta(\eps)\to 0$, and  $\rho(\eps)\to 0$ such that for $J_\eps:=J_{\rho(\eps),\delta(\eps),\eps}$ we have
\begin{equation}
\label{eq: diagonal estimates}
    \begin{cases}
 \displaystyle
	\limsup_{\eps \to 0}\|\iota_\eps J_{\eps} - \nu\|_{\tKR(\overline U)} = 0
        \,  ,
\\
\displaystyle
	\dive \iota_\eps  J_{\eps} = \mu_\eps
            \,  ,
\\
\displaystyle
	\limsup_{\eps \to 0} F_\eps(J_{\eps}, \overline U) \leq \bF_{\hom}(\nu,\overline U)
        \, ,
    \end{cases}
\end{equation}
thus providing the sought recovery sequence 
and finishing the proof of the upper bound for bounded $U$.

We are now left to discuss how to prove the upper bound when $U=\R^d$. The proof is very similar to the one for bounded domains, let us describe how to adapt each step to the full space setting.

\smallskip 
\noindent
\textit{Step 1}: \ 
For every $\rho>0$, we seek  $j_\rho\in C^\infty(\R^d;V\otimes \R^d)\cap \Lip(\R^d; V \otimes \R^d)$ satisfying, with $\nu_\rho:= j_\rho \Lm^d$, the properties described in \eqref{eq: mollification estimates}.
To do so, we simply take the convolution of the measure $\nu$ with a standard mollifier $\psi_{\rho'}$ such that $\psi_{\rho'} \in C_c^\infty(B_\delta(0))$, for some suitably small $\rho'>0$.
The $\tKR$-estimate follows by a direct computation, observing that for every $\varphi \in C^1(\R^d;W^*)$, $W:= V \otimes \R^d$, with $\|\nabla \varphi \|_\infty \leq 1$, we have that 
\begin{align}
    \Big| 
        \int \varphi \dd (\psi_{\rho'} * \nu) - \int \varphi \dd \nu 
    \Big|
        \leq 
    \int 
    \Big|
        \varphi * \psi_{\rho'}
            - 
        \varphi 
    \Big|
        \dd |\nu| 
        \leq 
    \rho' |\nu|(\R^d) 
        \, ,
\end{align}
uniformly in such test functions $\varphi$. This shows the first estimate, the second one follows by the very same computations, using that $\dive (\psi_{\rho'} * \nu) = \psi_{\rho'} * \dive \nu$ and that $|\dive \nu|(\R^d) <\infty$ by assumption. Finally, the energy estimate follows once again from \cite[Corollary~2.11]{Baia-Chermisi-Matias-Santos:2013}, together with the fact that $f_{\hom}(j) \leq c_2 |j|$. Indeed, by equi-tightness of $\{ \nu \}_\rho$, for every $\lambda>0$, we can find $K_\lambda \subset \R^d$ compact so that 
\begin{align}
    \sup_{\rho \in (0,1)} 
        |\nu_\rho|(\R^d \setminus K_\lambda)
            \vee 
        |\nu|(\R^d \setminus K_\lambda)
            \leq \lambda 
                \, , 
\end{align}
as well as ensuring that $\nu_\rho|_{K_\lambda} \to \nu|_{K_\lambda}$ narrowly in $\M(K_\lambda: V \otimes \R^d)$. 
By the growth condition on $\bF_{\hom}$, we deduce that, for every $\lambda \in (0,1)$,  
\begin{align}
    \limsup_{\rho \to 0}
        \bF_{\hom}(\nu_\rho)
            &\leq 
        c_2 \lambda 
            + 
         \limsup_{\rho \to 0}
            \bF_{\hom}(\nu_\rho;K_\lambda)
        \\
        &\leq 
            c_2 \lambda 
                +
            \bF_{\hom}(\nu_{\rho,\delta};K_\lambda)
        \leq 
        c_2 \lambda 
                +
            \bF_{\hom}(\nu_{\rho,\delta})
                \,  ,
\end{align}
where at last applied once again \cite[Corollary~2.11]{Baia-Chermisi-Matias-Santos:2013}. Sending $\lambda \to 0$ we conclude.

\smallskip 
\noindent
\textit{Step 2}: \ 
For $\delta>0$, we cover the domain using countably many cubes $Q_z:= z + [0,\delta)^d$, where $z \in \delta \Z^d$ and define a piecewise constant $j_{\rho,\delta}$, $\nu_{\rho,\delta}= j_{\rho,\delta} \Lm^d$  as in \eqref{eq:piecewise}, this time ensuring that $|j_{\rho,\delta}(x)| \leq |j_\rho(x)|$ for every $x \in \R^d$. Note that this guarantess equitightness of $\{\nu_{\rho,\delta}\}_\delta$. A possible choice is to define
\begin{align}
    j_{\rho,\delta}(x)
        \in 
    \argmin
    \left\{
        |j_\rho(x)|
            \suchthat 
        x \in Q_\delta(z)
    \right\}
        \, .
\end{align}
Now we claim that the following modified version of \eqref{eq: discretization estimates} holds in the full space: for every $\lambda \in (0,1)$, there exists a constant $C(\lambda,\rho) \in \R_+$ depending on $\lambda,\rho$ such that  
\begin{equation}\label{eq: discretization estimates_2}
	\begin{cases}
		\|\nu_{\rho,\delta} - \nu_\rho\|_{\tKR(\R^d)} \leq \lambda + C(\lambda,\rho)\delta
			\, ,
	\\
 \displaystyle
		\|\dive \nu_{\rho,\delta} - \dive \nu_\rho\|_{\tKR(\R^d)} \leq \lambda + C(\lambda,\rho)\delta
			\,  ,
	\\
 \displaystyle
		\bF_{\hom}(\nu_{\rho,\delta}) \leq \bF_{\hom}(\nu_\rho) + \lambda + C(\lambda,\rho)\delta
			\, .
	\end{cases}
\end{equation}
In order to show the first bound, we observe that for every given $\lambda>0$, by tightness we have that there exists a compact set $K_\lambda = K_\lambda(\rho) \subset \R^d$ such that, for all $\delta \in (0,1)$, 
\begin{align}
\label{eq:tightness}
    |\nu_\rho|(\R^d \setminus K_\lambda) \vee |\nu_{\rho,\delta}|(\R^d \setminus K_\lambda) \leq \frac\lambda 2
        \, . 
\end{align}
In particular, it follows that 
\begin{align}
    \|\nu_{\rho,\delta} - \nu_\rho\|_{\tKR(\R^d)}
        \leq 
    \lambda 
        +
    \|\nu_{\rho,\delta} - \nu_\rho\|_{\tKR(K_\lambda)}
        \leq
    \lambda     
        +
    C(\lambda,\rho)\delta  
        \, , 
\end{align}
where at last we used the Lipschitz continuity of $j_\rho$. Similarly, for any given $\psi \in C^1(\R^d;V^*)$ with $\Lip \psi \leq 1$ we have that
\begin{align}
     \big|
    \langle
        \dive j_{\rho,\delta} - \dive j_\rho
            ,
        \psi
    \rangle
    \big|
        =
    \int
    \big|
    \langle
        \nabla \tilde \psi
            ,
        j_{\rho,\delta} - j_\rho
    \rangle
    \big|
        \dd x
    \leq 
        \lambda 
            &+
    \int_{K_\lambda}
    \big|
    \langle
        \nabla \tilde \psi
            ,
        j_{\rho,\delta} - j_\rho
    \rangle
    \big|
        \dd x    
\\
            &\leq 
    \lambda + C(\rho) \Lm^d(K_\lambda) \delta 
        \, , 
\end{align}
where at last we used once again the Lipschitz regularity of $j_\rho$. 

Concerning the energy estimate, we use once again the tightness property \eqref{eq:tightness} and the linear growth property (Lemma~\ref{lemma:prop_fhom}) of $f_{\hom}$ with $f_{\hom}(0)=0$ to achieve
\begin{align}
    \bF_{\hom}(\nu_{\rho,\delta})
        &=
    \int_{\R^d \setminus K_\eta}
        f_{\hom}(j_{\rho,\delta}) \dd x 
            +
    \int_{K_\lambda}
        f_{\hom}(j_{\rho,\delta}) \dd x 
    \\
        &\leq 
    C \lambda 
        + 
    \bF_{\hom}(\nu_\rho; K_\eta)
        +
    c_2 \int_{K_\lambda}
        |j_{\rho,\delta}(x) - j_\rho(x)|
            \dd x
    \\
        &\leq 
    \bF_{\hom}(\nu_\rho)
        +
    c_2 \lambda 
        + 
    C_\rho \Lm^d(K_\lambda) \delta 
        \, , 
\end{align}
where we used the Lipschitz property of $j_\rho$ once again and the nonnegativity of $f_{\hom}$.

\smallskip 
\noindent
\textit{Step 3}: this step works precisely as in the setting of bounded domains. In fact, it is simpler, because we can can cover the whole $\R^d$ exactly with countably many disjoints cubes $Q_\delta(z)$, $z \in \delta \Z^d$. In particular, we do not need to introduce the excess set $D_{\delta,\eta}$ in the case of $U=\R^d$ (cfr. \eqref{eq:ub_step3_energy_split}). 

\smallskip 
\noindent
\textit{Step 4-5}: these steps follows from the previous steps exactly as in the case with bounded $U$, with additional dependence on $\lambda$ arising from Step 2. The conclusion then follows by taking suitable diagonal sequence $\lambda(\eps)\to 0$ as well. 
\end{proof}

\section{The lower bound}
\label{sec:lb}
Thorughout the whole section, we will omit the $\omega$-dependence everywhere, as done in the previous chapter. 
Consider any sequence $J_\eps$ with $\sup_\eps F_\eps(J_\eps, \overline{U})<\infty$ with $\DIVE J_\eps = m_\eps$ and such that $\iota_\eps J_\eps \to \nu$ vaguely. We define the positive Borel measures
\begin{align}
\label{eq:def_nu_eps}
    \nu_\eps := F_\eps(J_\eps,\cdot) \in \M_+(\overline U)
        \, .
\end{align}
 By the local compactness of $\M_+(\overline U)$ in the vague topology and Portmanteau's theorem, up to extracting a subsequence,
we obtain that 
	\begin{equation}
 \label{eq:def_nu}
		\begin{cases} \displaystyle
			\liminf_{\eps \to 0} \nu_\eps(\overline U) = \liminf_{\eps \to 0} F_{\eps}(J_\eps,\overline{U}) \geq \nu(\overline U) \, , \\
			\nu_\eps \to \nu\in \M_+(\overline U) \, , \\
			\iota_\eps J_\eps \to \xi \in \M(\overline U;V\otimes \R^d) \, ,
		\end{cases}
	\end{equation}
	   with respect to the vague topology.
	Our goal is to prove the two inequalities
	\begin{align}	\label{eq:lb_abs_cont_part}
		f_{\omega,\hom}\left(\frac{\dd \xi}{\dd \Lm^d}\right) \leq \frac{\dd \nu}{\dd \Lm^d} \, ,  \qquad & \Lm^d\text{-almost everywhere in }\overline U \, ,
		\\
			\label{eq:ub_blowup_recess}
		f_{\omega,\hom}^\infty\left( \frac{\dd \xi}{\dd |\xi|}\right) \leq \frac{\dd \nu}{\dd |\xi|} \, ,  \qquad & |\xi|^s \text{-almost everywhere in }\overline U \, .
	\end{align}

	Let us for the moment assume that \eqref{eq:lb_abs_cont_part} and \eqref{eq:ub_blowup_recess} hold.
	By the Radon--Nikodym theorem, we can decompose the measure $\nu$ into three parts:
	an absolutely continuous part with respect to $\Lm^d$,
	a singular part which is absolutely continuous with respect to $|\xi|^s$,
	and a measure $\nu_j^s \in \M_+(\R^d)$
	which is mutually singular with respect to both $\Lm^d$ and $|\xi|^s$, in the form
	\begin{align}
	\label{eq:decomp_nu_j}
		\nu =
		\frac{\dd \nu}{\dd \Lm^d}
			\Lm^d 	+
		\frac{\dd \nu}{\dd |j|}
			|\xi|^s 	+
		\nu_j^s
			\, .
	\end{align}
	The lower bound then follows from
	\begin{equation}
		\begin{aligned}
			\bF_{\omega,\hom}^\mu(\xi)
				&=  \int_{\overline U} f_{\omega,\hom}\left(\frac{\dd \xi}{\dd \Lm^d}\right) \ddd \Lm^d + \int_{\overline U} f_{\omega,\hom}^\infty\left(\frac{\dd \xi}{\dd |\xi|}\right) \ddd |\xi|^s\\
				&\leq  \int_{\overline U}
					\frac{\dd \nu}{\dd \Lm^d} \dd \Lm^d + \int_{\overline U}
					\frac{\dd \nu}{\dd |\xi|} \dd |\xi|^s\\
				&\leq \nu(\overline U) = \liminf_{\eps \to 0} F_{\omega,\eps}(J_\eps) \, ,
		\end{aligned}
	\end{equation}
where the first inequality uses \eqref{eq:lb_abs_cont_part} and \eqref{eq:ub_blowup_recess}, and the second inequality uses the decomposition of $\nu$ \eqref{eq:decomp_nu_j} and the fact that $\nu_j^s$ is a positive measure.

\subsection{Proof for the absolutely continuous part}

In this section, we show \eqref{eq:lb_abs_cont_part}, i.e.,
\begin{align}
	\label{eq:lb_absolute}
	f_{\hom}
	\biggl(\frac{\ddd  \xi}{\ddd \Lm^d}\biggr)
		\leq \frac{\ddd \nu}{\ddd \Lm^d }  \qquad & \Lm^d\text{-almost everywhere in }\overline U
		\, .
\end{align}
It follows immediately from \eqref{eq:fhom_formula} that
\begin{align*}
	f_{\hom}(j)
		\leq
	\liminf_{\eps \to 0}
		\frac{F_\eps(J_\eps,A)}{\Lm^d(A)}
\end{align*}
whenever
$J_\eps \in \Rep_{\eps,\calR}(\eps^{d-1} j;A)$.
Our next result shows that the error in the above inequality is quantitatively controlled by how far $J_\eps$ is from being an actual representative, in a suitable sense.

In view of Proposition \ref{prop:J0},
there exists a uniform-flow operator
	$\calR \in \mathrm{Lin}(V\otimes \R^d; V_\rma^\cE)$,
which we fix from now on.

\begin{prop}[Non-asymptotic behavior of the energy on a cube I]
\label{prop:asymptotic_cube_AC}
Let
	$Q \subseteq \R^d$ be an open cube
and $j \in V \otimes \R^d$.
For $\eps \in (0,1)$,
let $J \in V_\rma^{\cE_\eps}$
and take $\eta > 0$ such that
\begin{align}
\label{eq:assumption_eta_quantitative}
    \max\{  \eps R_\partial, \eps R_\Lip \}
        <
    \eta
        <
    \min
        \Big\{
			\frac13
                ,
            \frac{1 + | j | }{16}
            \frac
                { \Lm^d(Q) }
                { |\iota_\eps J|(Q) }
        \Big\}
        \, .
\end{align}
Then we have
\begin{align}	\label{eq:lowerbound_prop_cube_AC}
        f_{\eps,\calR}
        \big(
            j ,  Q
        \big)
\leq
		F_\eps(J,Q)
    + C \err_{\eps, \eta}(J , j)
		\,  ,
\end{align}
where $C < \infty$ only depends on the constants $R_i$, $C_i$, $c_i$ appearing in the assumptions on $(\cX,\cE)$ and $F$, and where 
\begin{align}
    \err_{\eps, \eta}(J , j)
        & :=
    \frac1\eta \| \dive \iota_\eps J \|_{\tKR(\overline{Q})}
        +
    \frac1{\eta^2} \| \iota_\eps (J - \calR_\eps j) \|_{\tKR(\overline{Q})}
\\ & \qquad +
     \sqrt{\eta}  \Big( (1 + | j | ) { \Lm^d(Q)}  + | \iota_\eps J |(\overline Q) \Big)
\\ & \qquad +
    \eps
    \Big(
        |j| {\Lm^d(Q)}
            +
        | \dive \iota_\eps J |(\overline Q)
            +
        \frac1{\eta}
        | \iota_\eps J| (\overline Q)
    \Big) \, .
\end{align}
\end{prop}

The plan of the proof is to replace the vector field $J$ by an $\eps$-representative $J_3 \in \Rep_{\eps,\calR}(j;Q)$,
and show that the error can be controlled in terms of its divergence and the distance from the constant measure $j \Lm^d$.
To achieve this, we proceed in two correction steps: first we  correct the boundary values and then we correct the divergence. In the first operation, we have some freedom in the choice of where to perform the cutoff, which we will then optimise to obtain a nice error estimate as claimed in the proposition.

There are three length-scales that play a role in the proof below (cfr. Figure~\ref{fig:cutoff}):
\begin{itemize}
\item edges in the graph are of length $\sim \eps$,
\item  the transition region between bulk and boundary behaviour has width $\sim \eta$,
\item the location of this transition region will be carefully chosen near the boundary in a zone of width $\sim N \eta$.
\end{itemize}

\begin{figure}[h]
    \centering
\includegraphics[scale=1]{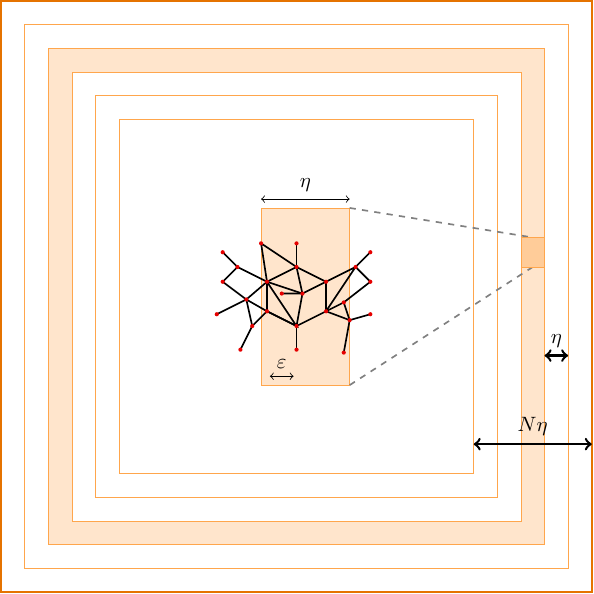}
    \caption{A representation of the three layers involved in the proof of Proposition \ref{prop:asymptotic_cube_AC}: a microscopic scale $\eps$ (edge length), a mesoscopic scale $1 \gg \eta \gg \eps$, representing the size of the strip where we apply a cutoff function to fix the boundary conditions, which must be chosen in a location of size $N \eta$ around the boundary of the cube, typically of order $\sqrt{\eta} \ll 1$.}
    \label{fig:cutoff}
\end{figure}

\begin{proof}[Proof of Proposition \ref{prop:asymptotic_cube_AC}]

The proof consists of three steps.

\medskip
\noindent
\emph{Step 1 (Boundary value correction)}. \
Fix a cutoff length-scale $\eta > 0$ satisfying
	\eqref{eq:assumption_eta_quantitative} 
and
$N \in \N$ so that $2N + 1 < \frac1{\eta}$.
The value of $N$ will be optimised below.

For $\ell = 1, \dots, N$,
let $\psi_\eta^\ell \in \cC_\rmc^\infty(\R^d)$
be a cutoff function
satisfying
	$0 \leq \psi_\eta^\ell \leq 1$,
	$\Lip (\psi_\eta^\ell) \leq C/\eta$,
	$\Lip (\nabla \psi_\eta^\ell) \leq C / \eta^2$,
and such that
\begin{align*}
	\psi_\eta^\ell (x) = 1
		\quad \text{for} \ x \in Q_{1-2\ell\eta}
	\tand
	\psi_\eta^\ell (x) = 0
		\quad \text{for} \ x \notin Q_{1-(2\ell-1)\eta} \, ,
\end{align*}
for some $C < \infty$ depending only on the dimension $d$.
Here, $Q_\alpha$ denotes the rescaled cube with the same center as $Q$ and side-length rescaled by a factor $\alpha > 0$.

For $\ell = 1, \dots, N$, we then define a vector field
	$J_2^\ell \in V_\rma^{\cE_\eps}$
by
\begin{align}
	\label{eq:J_2}
	J_2^\ell
	:=
	\hat \psi_\eta^\ell   J
	+
	(1-\hat\psi_\eta^\ell )
	\calR_\eps j
			\, ,
\end{align}
where the rescaled operator
$\calR_\eps$ has been defined in
 Definition~\ref{def:admissible}
and $\hat\psi_\eta^\ell$ has been defined in
	\eqref{eq:def_hat_0}.

We will show that
the vector field $J_2^{\bar \ell}$
has two desirable properties
for a suitable choice of
	$N$ and $\bar \ell \in \{ 1, \dots, N\}$:
namely, $J_2^{\bar \ell}$
has the right boundary conditions
to be a representative in the sense of Definition \ref{def:rep}
and
its energy is controlled by the energy of $J$.
More precisely:
\begin{align}
	&\text{Claim 1a:} &&
	J_2^{\bar \ell}(x,y) = \calR_\eps j(x,y)
		\text{ for all }
		(x,y) \in \cE_\eps
		\text{ with }
		\dist\bigl([x,y], Q^c \bigr)
		\leq \eps R_\partial \;,
	\\
	&\text{Claim 1b:} &&
    F_\eps(J_2^{\bar \ell},Q)
        -
    F_\eps(J,Q)
	\lesssim
		   \sqrt{\eta}  \Big( ( 1+ | j | ) { \Lm^d(Q)}
		   	+ | \iota_\eps J |(Q)\Big)
        \, .
\end{align}
Having obtained such $\bar{\ell}$, we simplify notation by writing $J_2 := J_2^{\bar \ell}$.

\medskip
\noindent \emph{Step 2 (Divergence correction)}. \
We will construct a discrete vector field $J_3 \in V_\rma^{\cE_\eps}$
that is divergence free with the same boundary values as $J_2$.
To do this, we note that $\DIVE J_2 = 0$ on $B(Q^c,\eps R_\partial)$
by \eqref{eq:assumption_eta_quantitative},
hence an application of Proposition~\ref{prop:correctors}
with
	$m = - \DIVE J_2$
yields a vector field
	$K \in V_\rma^{\cE_\eps}$
satisfying
\begin{align}	\label{eq:tv_bound_Cs_AC}
	\begin{cases}
		\DIVE K = -\DIVE J_2
		&  \text{on $(\cX_\eps, \cE_\eps)$}\, ,
	\\
		K \equiv 0
		\qquad
		&	\text{on }	B(Q^c,\eps R_\partial) \,  ,
	\\
		| \iota_\eps K |(\R^d)
			\lesssim
		\| \iota_\eps ( \DIVE J_2 )	 \|_{\tKR(\R^d)}
		+ \eps |\iota_\eps ( \DIVE J_2 ) |(\R^d)
	\end{cases}
	\end{align}
	We then define
	$J_3 \in V_\rma^{\cE_\eps}$
	by
		$J_3 := J_2 + K$.
We will show the following properties of $J_3$:
\begin{align}
	& \text{Claim 2a:} &&
		J_3 \in \Rep_{\eps, \calR}(j; Q)  \,, \\
	& \text{Claim 2b:} &&
		| \iota_\eps(J_3 - J_2)|(\R^d)
		\lesssim
	\frac1\eta \| \iota_\eps \dive J \|_{\tKR(\overline Q)}
        +
    \frac1{\eta^2} \| \iota_\eps (J - \calR_\eps j) \|_{\tKR(\overline Q)}
\\
       & && \hspace{3.6cm} +
    \eps
    \Big(
        | \iota_\eps \dive J |(\overline Q)
            +
        \frac1{\eta}
        | \iota_\eps( J - \calR_\eps j)| (\overline Q)
    \Big) \, .
 \end{align}

\medskip
\noindent
\textit{Step 3 (Energy estimate).} \
 It remains to estimate the discrete energy of $J_3$.
Since
	$J_3 \in \Rep_{\eps, \calR}(j; Q) $
we obtain,
using the Lipschitz property from \eqref{eq:Lipschitz-F-eps},
\begin{align}	\label{eq:energy_est_AC}
	f_{\eps,\calR}
		\big(
			j,
			Q
		\big)
	\leq
	F_\eps(J_3, Q)
	\leq
	F_\eps(J_2, Q)
		+
		2C_1
		\bigl|\iota_\eps(J_3 - J_2)
		\bigr|
		\bigl(B(Q,\eps R_\Lip)\bigr)
		\, .
\end{align}
We then apply the estimates obtained in Claim 1b and Claim 2b and conclude the proof using at last the boundedness of $\calR_\eps$ (Definition~\ref{def:admissible}, (3)).
It remains to prove the claims.
\medskip

\emph{Proof of Claim 1a.} \
Take $(x,y) \in \cE_\eps$ satisfying
	$\dist\bigl( [x,y], Q^c \bigr)
		\leq \eps R_\partial$.
Since $(G3)$ ensures that $|x-y|\leq \eps R_3$,
it follows that $[x,y]$ cannot be near the centre of the cube,
provided $\eps \ll \eta$.
More precisely, we have
	$[x,y] \subseteq ( Q_{1- (2 \ell -1) \eta} )^c$.
Therefore, $\hat\psi_\eta^\ell(x,y) = 0$, so that
	$J_2^\ell(x,y) = \calR_\eps j(x,y)$.
\qed

\medskip
\emph{Proof of Claim 1b.} \
For every $\ell \in \{1, \dots, N \}$, we decompose the cube  $Q$ as disjoint union
	 $Q = A_1^\ell \cup A_\tra^\ell \cup A_0^\ell$,
where $A_1^\ell$ is a smaller concentric cube surrounded
by $L^\infty$-spherical shells $A_\tra^\ell$ and $A_0^\ell$.
They are chosen so that
	$A_1^\ell$ lies well inside the set $\{ \psi_\eta^\ell = 1 \}$,
	$A_0^\ell$ lies well inside the set $\{ \psi_\eta^\ell = 0 \}$,
	and $A_\tra^\ell$ is a transition region.
More precisely,
\begin{align}
    A_1^\ell
		:= Q_{1-(2\ell+1)\eta}
        \, , \quad
	A_\tra^\ell
		=  Q_{1-(2\ell - 2) \eta} \setminus  A_1^\ell
		\, , \quad
	A_0^\ell
		:= Q \setminus \big( A_\tra^\ell \cup A_1^\ell \big)
        \, .
\end{align}
Note that
	$1 -(2\ell+1)\eta>0$,
since we assumed that
	$2N+1 < \frac1{\eta}$.
By additivity of $F_\eps$, which follows from $(F3)$, we have
\begin{align}
	\label{eq:addi}
    F_\eps(J_2^\ell,Q)
        =
    F_\eps(J_2^\ell,A_1^\ell)
        +
    F_\eps(J_2^\ell,A_\tra^\ell)
        +
    F_\eps(J_2^\ell,A_0^\ell)
        \, .
\end{align}
We will estimate the three terms separately.

\smallskip
\emph{i. Bulk term.} \
By the definitions, we have
\begin{align}
\label{eq:difference_cutoff}
    J - J_2^\ell
        =
    (1-\hat \psi_\eta)
    \big(
        J - \calR_\eps j
    \big)
\, .
\end{align}
Since $J = J_2^\ell$ on $\tilde B := B(A_1^\ell,\eps R_{\Lip})$ (in the sense that
	$|\iota_\eps(J - J_2^\ell)|
		(\tilde B) = 0$),
it follows from Lemma \ref{lemma:tv_bounds}\eqref{item:F-lip} that
\begin{align}
	\label{eq:bulk}
     F_\eps(J_2^\ell,A_1^\ell)
        =
     F_\eps(J,A_1^\ell)
        \leq
     F_\eps(J,Q)
        \, ,
\end{align}
by additivity and nonnegativity of $F_\eps$. 

\smallskip
\emph{ii. Boundary term.} \
Observe that
	$J_2^\ell \equiv \calR_\eps j$
on $B(A_0^\ell, \eps R_\Lip)$
for every $\ell \in \{1, \dots, N \}$.
Using Lemma \ref{lemma:tv_bounds}, we infer that
\begin{align}
    F_\eps(J_2^\ell,A_0^\ell)
        &=
    F_\eps(\calR_\eps j , A_0^\ell)
	\\&
	=
	F_\eps(0 , A_0^\ell)
	+
	\bigl( F_\eps(\calR_\eps j , A_0^\ell)
	- F_\eps(0 , A_0^\ell) \bigr)
    \\&
	\leq
    C_2 \Lm^d(A_0^\ell)
        +
   2C_1  |\iota_\eps \calR_\eps j|
   	\bigl(B(A_0^\ell,\eps R_{\Lip})\bigr)
\\
        &\lesssim
    \big(
        \Lm^d
            +
        |\iota_\eps \calR_\eps  j|
    \big)
    (A_0^{\ell+1})
        \, ,
\end{align}
using the inclusion of sets
	$A_0^\ell
		\subset B(A_0^\ell,\eps R_{\Lip})
		\subset A_0^{\ell+1}$,
which holds by the assumption that
	$\eta \leq \eps R_\Lip)$.
As $A_0^{\ell+1}$ is a union of orthotopes,
we conclude
using Remark~\ref{rem:vague_J0_cubes}
that
\begin{align}
	\label{eq:boundary}
    F_\eps(J_2^\ell,A_0^\ell)
        \lesssim
    (1+ | j | ) \Lm^d(A_0^{\ell+1})
        \lesssim
    (1+| j | )
    { \Lm^d(Q)}
    N \eta
\end{align}
for all $l \in \{1 , \dots , N\}$.

\smallskip
\emph{iii. Transition term.} \ \
Let us define
\begin{align}
    \tilde A_\tra^\ell := (A_0^{\ell-1})^c\setminus A_1^{\ell+1}
        \, , \quad
        \text{so that }
    A_\tra^\ell
        \subset
    B(A_\tra^\ell, \eps R_\Lip)
        \subset
    \tilde A_\tra^\ell
        \, ,
\end{align}
and note that $\tilde A_\tra^\ell$ is a disjoint union of orthotopes, each containing a cube of side-length $\eps > 0$.
Observing that $|J_2^\ell| 
 \leq |J| + |\calR_\eps j|$ edge-wise,
using again Remark~\ref{rem:vague_J0_cubes}
and the growth conditions on $F$ from  Lemma \ref{lemma:tv_bounds}, we find
\begin{align*}
	F_\eps(J_2^\ell, A_\tra^\ell)
	&=
	F_\eps(0 , A_\tra^\ell)
	+
	\bigl( F_\eps( J_2^\ell , A_\tra^\ell )
	- F_\eps(0 , A_\tra^\ell) \bigr)
    \\&
	\leq
    C_2 \Lm^d( A_\tra^\ell)
        +
   2C_1  |\iota_\eps J_2^\ell  |
   	\bigl( \tilde A_\tra^\ell \bigr)
\\
        &
	\lesssim
        \Lm^d( A_\tra^\ell )
            +
			|\iota_\eps J |
			\bigl( \tilde A_\tra^\ell \bigr)
			+ 
			|\iota_\eps \calR_\eps j  |
			\bigl( \tilde A_\tra^\ell \bigr)
\\ &
\lesssim
        (1 + |j|) \Lm^d( \tilde A_\tra^\ell )
            +
			|\iota_\eps J  |
			\bigl( \tilde A_\tra^\ell \bigr)
\\ &
\lesssim
        (1 + |j|)  { \Lm^d(Q)} \eta
            +
			|\iota_\eps J  |
			\bigl( \tilde A_\tra^\ell \bigr)		
        \, ,
\end{align*}
Now, we choose $\bar{\ell} \in \{1, \dots, N \}$ so that
\begin{align}
    |\iota_\eps J|(\tilde A_\tra^{\bar{\ell}})
        \leq
    \frac1N
        \sum_{\ell=1}^N
            |\iota_\eps J|(\tilde A_\tra^\ell)
        \leq
    \frac6N
        |\iota_\eps J|(Q)
            \, ,
\end{align}
where we used that 
each point in $Q$ is contained in at most 6 sets in the collection
$\{ \tilde A_\tra^\ell\}_{\ell = 1}^N$.
We thus arrive at
\begin{align}
	\label{eq:transition}
	F_\eps(J_2^\ell, A_\tra^\ell)
	\lesssim
        (1 + |j|)  { \Lm^d(Q)} \eta
            +
			\frac1N
			|\iota_\eps J|(Q) \, .
\end{align}

Summing the three contributions 
\eqref{eq:bulk},
\eqref{eq:boundary}, and
\eqref{eq:transition},
we find using \eqref{eq:addi},
\begin{align}
	\label{eq:summing}
	F_\eps(J_2^\ell,Q)
	- F_\eps(J,Q)
	\lesssim 
	\frac1N
	|\iota_\eps J|(Q)
	+
	( 1 + | j | )
    { \Lm^d(Q)}
    N \eta\, .
\end{align}
It remains to optimise in $N$.
To do so, define 
\begin{align}
    \nu = 
	\sqrt{\frac1{\eta (1+| j | )} 
		\frac{|\iota_\eps J|(Q)}{
    { \Lm^d(Q)}}}
        \, ,
\end{align}
which minimises the right-hand above among all positive \emph{real} values of $N$.

If $\nu \leq 2$, we have $
|\iota_\eps J|(Q)
	\leq 
4 (1+| j |) \Lm^d(Q)\eta $.
Using this inequality, we insert $N = 1$ at the right-hand side of \eqref{eq:summing} to obtain
\begin{align*}
	F_\eps(J_2^\ell,Q)
	- F_\eps(J,Q)
	\lesssim 
	|\iota_\eps J|(Q)
	+
	( 1 + | j | )
    { \Lm^d(Q)}
     \eta
	 \leq
	5 ( 1 + | j | )
    { \Lm^d(Q)}
     \eta
	 \, ,
\end{align*}
as desired.

If $\nu > 2$, we can fix a positive \emph{integer} $N$ with $\nu/2 \leq N \leq \nu$.
We claim that $N$ satisfies the requirement $2N+1 < \frac1\eta$.
Indeed, 
\eqref{eq:assumption_eta_quantitative}
yields $4 \eta \nu \leq 1 $, 
so that
$
2N
\leq 
2\nu
\leq 
\frac{1}{2\eta}
\leq
\frac{1}{\eta}-1
$.
Plugging $N$ into \eqref{eq:summing} we find
\begin{align*}
	F_\eps(J_2^\ell,Q)
	- F_\eps(J,Q)
&
	\lesssim 
	\frac1\nu
	|\iota_\eps J|(Q)
	+
	( 1 + | j | )
    { \Lm^d(Q)}
    \nu \eta
\\ &
	\lesssim
    \sqrt{\eta(1+| j | )
        |\iota_\eps J |(Q) 
    { \Lm^d(Q)}
        }
\\
        &\lesssim
    \sqrt{\eta}
    \Big(
        |\iota_\eps J|(Q) + (1+| j | )
    { \Lm^d(Q)}
    \Big)
	\, ,
\end{align*}
which concludes the proof of Claim 1b.

\medskip
\noindent \emph{Proof of Claim 2a.} \
By construction, we have
	$\DIVE J_3 = \DIVE J_2 + \DIVE K=0$.
Moreover, $J_3$ satisfies the property stated for $J_2$ in Claim 1a, 
since $K \equiv 0$ on $B(Q^c,\eps R_\partial) $.
\qed

\medskip
\noindent \emph{Proof of Claim 2b.} \
Note that, thanks to
\eqref{eq:tv_bound_Cs_AC} 
and Remark~\ref{rem:extension_Lip},
\begin{align}
\label{eq:Claim2b_beginning}
	| \iota_\eps(J_3 - J_2) | (\R^d)
	=
	| \iota_\eps K | (\R^d)
	\lesssim
		 \|
		 	\dive \iota_\eps J_2
		 \|_{\tKR(\overline Q)}
    + \eps |\dive \iota_\eps J_2 |(Q) \, .
\end{align}
Using the notation from
	Section~\ref{sec:preliminaries_derivatives},
it follows from the discrete Leibniz rule
\eqref{eq:Leibniz}
and the fact that $\DIVE \calR_\eps j \equiv 0$ on $\cX_\eps$,
that
\begin{align}	\label{eq:DIVE_J2}
	\DIVE J_2 = \psi_\eta \DIVE J
		+
	(\grad  \psi_\eta) \star
		(
			J - \calR_\eps j
		) \, .
\end{align}
By Proposition~\ref{prop:equivalence_convergence_KR}, we have
\begin{align}
    \| \dive (\iota_\eps J_2) \|_{\tKR(\overline Q)}
        &\lesssim
    \| \dive (\iota_\eps J_2) \|_{\KR(\overline Q)}
\\
        &\leq
    \| \psi_\eta \dive (\iota_\eps J_2) \|_{\KR(\overline Q)}
        +
    \| \iota_\eps (\grad  \psi_\eta) \star
		(
			J - \calR_\eps j
		) \|_{\KR(\overline Q)}
  \, .
\end{align}
We analyse both terms on the right-hand side separately.
On the one hand, using the definition of $\KR$-norm and the bound $\Lip(\psi_\eta) \lesssim 1/\eta$ with $\| \psi_\eta \|_\infty \leq 1$, we obtain
\begin{align}
\label{eq:bound_divergence_AC}
     \| \psi_\eta \dive (\iota_\eps J_2) \|_{\KR(\overline Q)}
        \lesssim
    \frac1 \eta  \| \dive (\iota_\eps J_2) \|_{\KR(\overline Q)}
        \, .
\end{align}
On the other hand, \eqref{eq:gradients_comparison} in Lemma~\ref{lemma:gradients} shows that
\begin{align}
    \big\| \iota_\eps
    \big(
        (\grad \psi_\eta )
        \star
         (J - \calR_\eps j )
    \big)
        \big\|_{\KR(\overline Q)}
&\lesssim
	\|
		\nabla \psi_\eta
			\cdot
		\iota_\eps (J - \calR_\eps j )
	\|_{\KR(\overline Q)}
\\ & \quad +
	\eps | \nabla \psi_\eta \cdot \iota_\eps (J - \calR_\eps j ) |(\overline Q)\, .
\end{align}
By construction, we have $\| \eta \nabla \psi_\eta \|_\infty \lesssim 1$ and $\Lip (\eta^2 \nabla \psi_\eta) \lesssim 1$, so that the previous estimate yields
\begin{align}
\label{eq:square_KR}
 \big\| \iota_\eps
    \big(
        (\grad \psi_\eta )
        &\star
         (J - \calR_\eps j )
    \big)
        \big\|_{\KR(\overline Q)}
\\
        &\lesssim
    \frac1{\eta^2} \| \iota_\eps (J - \calR_\eps j )
	\|_{\KR(\overline Q)}
        +
    \frac\eps\eta |\iota_\eps (J - \calR_\eps j)| (\overline Q)
        \, .
\end{align}
Putting the last inequality, together with \eqref{eq:bound_divergence_AC}, the chain rule \eqref{eq:DIVE_J2} for the divergence, and \eqref{eq:Claim2b_beginning},  we conclude the proof of Claim 2b.
\end{proof}

We can now conclude the proof of the lower bound for the absolutely continuous part, i.e.,~we prove \eqref{eq:lb_absolute}.
Let us briefly recall the setup.

We work with a sequence of discrete vector fields 
	$J_\eps \in V^{\cE_\eps}_\rma$ 
and write $m_\eps := \DIVE J_\eps$.
As $\eps \to 0$, we assume that
\begin{align}	\label{eq:final_hyp_0_AC}
	\iota_\eps J_\eps \to \xi
        \quad \text{vaguely} \quad 
		\tand
	m_\eps 
	=
	\DIVE(\iota_\eps J_\eps)
	 \to \mu
		\quad \text{in } \tKR
		\, ,
\end{align}
for suitable measures
	$\xi \in \M(\R^d; V \otimes \R^d)$
and 
	$\mu \in \M(\R^d; V)$.
By the uniform boundedness principle, $\{ |m_\eps| \}_\eps$ is locally uniformly bounded in total variation, hence we may and will assume without loss of generality that, as $\eps \to 0$,
\begin{align}
|m_\eps| \to \sigma
\quad \text{vaguely} \, ,
\end{align}
for some $\sigma \in \M_+(\R^d)$.

Consider the measures 
	$\nu_\eps := F_\eps(J_\eps,\cdot) \in \M_+(\R^d)$
and assume that $\nu_\eps$ converges vaguely to a measure 
	$\nu \in \M_+(\R^d)$.
Our goal is to show that 
\begin{align}
	\label{eq:lb_absolute-2}
	f_{\hom}
	\biggl(\frac{\ddd  \xi}{\ddd \Lm^d}\biggr)
		\leq \frac{\ddd \nu}{\ddd \Lm^d }  \qquad & \Lm^d\text{-a.e.} 
\end{align}

The proof is based on a blow-up strategy around a fixed point $x_0 \in \R^d$
with magnification factor $1/\delta$ for some $\delta > 0$. 
We thus consider the rescaled measures
\begin{align*}
	\xi_{\delta, x_0} 
		:= \delta^{-d} ({\rho_{\delta,x_0}})_\# \xi
	\tand
	\xi_{\delta, x_0}^\eps 
		:= \delta^{-d} ({\rho_{\delta,x_0}})_\# (\iota_\eps J_\eps )
	\, .
\end{align*}
Note that the edge-lengths in the blown-up vector field $\xi_{\delta, x_0}^\eps$ are of order $s := \eps/\delta$. 
Since the graph $\cX$ is not assumed to be translation invariant, 
$\xi_{\delta, x_0}^\eps$ is \emph{not} in general the continuous embedding of a discrete vector field on $\cX_s$. 
However, this is true up to translation by $x_0/\delta$: namely, 
\begin{align}
	\label{eq:K-def}
	(\rho_{1,-x_0/\delta})_\# \xi_{\delta, x_0}^\eps 
	= 
	 \iota_{\eps/\delta} K_\delta^\eps \, ,
\end{align}
where $K_\delta^\eps \in V_\rma^{\cE_s}$ is defined by
\begin{align} \label{eq:def_K_s_AC}
    K_\delta^\eps(x,y) := \delta^{1-d} J_\eps(\delta x, \delta y)
\end{align}
for $(x,y) \in \cE_s$. 
It follows from the scaling relations that
\begin{align}
	\label{eq:scaling-F}
	F_{\eps/\delta}(K_\delta^\eps, Q_1(x_0/\delta)) 
	=
	\frac{F_\eps(J_\eps,Q_\delta(x_0))}
		 {\Lm^d(Q_\delta(x_0))}
\end{align}

Very loosely speaking, the proof proceeds as follows: we fix 
	$x_0 \in \supp(|\xi|) \setminus E$,
where $E$ is a suitable Lebesgue null-set outside of which the measures $\xi$, $\nu$, and $\sigma$ are well-behaved. Set $j_0 := \frac{\ddd \xi}{\ddd \Lm^d }(x_0)$.
We will rigorously justify the following chain of approximations, for $\eps \ll \delta \ll 1$: 
\begin{align*}
	f_{\hom}(j_0)
\stackrel{(1)}{\approx}
	f_{\eps/\delta,\calR}
		(j_0, Q_1(x_0 / \delta))	
\stackrel{(3)}{\lessapprox} 
	F_{\eps/\delta}(
		K_\delta^\eps, 
		Q_1(x_0 / \delta) 
=
	\frac{F_\eps(J_\eps,Q_\delta(x_0))}
		{\Lm^d(Q_\delta(x_0))}
\stackrel{(2)}{\approx}
	\frac{\ddd \nu}{\ddd \Lm^d}(x_0) \, .
\end{align*}

Let us now make the argument rigorous. 

\smallskip
\emph{Selection of a good set of full measure.} \
Proposition~\ref{prop:Besicovitch},
Lemma~\ref{lemma:tangent_measures}, and
Lemma~\ref{lemma:divergence_measures} ensure
that there exists a Borel set $E$ with $\Lm^d(E)=0$
such that, for every $x_0 \in \supp(|\xi|)\setminus E$, the following properties hold:
\begin{enumerate}
    \item[(i)]
	$x_0$ is a Lebesgue point of 
		$\frac{\ddd \xi}{\ddd \Lm^d}$,
	so that 
	\begin{align*}
	j_0 
	:= \lim_{\delta \to 0}
	\frac{\xi(Q_\delta(x_0))}{\Lm^d(Q_\delta(x_0))} 
	\end{align*}
	exists.
	In view of Lemma~\ref{lemma:tangent_measures}:
	there exists a sequence $\delta=\delta_m(x_0) \to 0$ as $m \to \infty$ such that
\begin{align}	\label{eq:final_blowup_AC}
	\quad
	\xi_{\delta,x_0}
	\to
		j_0 \Lm^d
	\text{ vaguely in }
	\M(\R^d; V \otimes \R^d)
	 	\, .
\end{align}
\item[(ii)]
	In view of Proposition~\ref{prop:Besicovitch},
\begin{align}
\label{eq:Leb_point_nu}
	f_0 := \lim_{\delta \to 0}
	\frac{\nu(Q_\delta(x_0))}{\Lm^d(Q_\delta(x_0))} 
	\ \text{ exists} \, .
\end{align}
    \item[(iii)]
	In view of Proposition~\ref{prop:Besicovitch}
	we have, for every bounded Borel set $B \subseteq \R^d$,
    \begin{align}	\label{eq:tv_bound_mu_rescale_AC}
	C_\sigma :=
	\sup_{\delta>0}
		\frac{\sigma( B_\delta(x_0))}{\Lm^d(Q_\delta(x_0))}
		< \infty
		\, .
\end{align}
\end{enumerate}
From now on, we fix $x_0 \in \supp(|\xi|)\setminus E$.

Next we rigorously justify the three approximation steps from the informal argument above.
In each of the three steps, we perform a diagonal argument, which allows us to pass to the limit $\delta \to 0$
while 
simultaneously 
$\eps := \eps(\delta) \to 0$ fast enough.

\smallskip
\emph{Approximation 1.} \
For all fixed $\delta > 0$, Lemma~\ref{lemma:fhom_formula}, 
implies that,
as $\eps \to 0$,	
\begin{align}
	\label{eq:1-rigorous}
       f_{\eps/\delta, \calR}(j_0,Q_1(x_0/\delta))
	\to f_{\hom}(j_0)
	\quad \text{ vaguely. } 
\end{align}
Thus, by a diagonal argument,
the same convergence holds
as 
$\delta \to 0$ 
and
$\eps = \eps(\delta) \to 0$ sufficiently fast.

\smallskip
\emph{Approximation 2.} \
Since $\nu$ is a finite Radon measure, 
$\nu(\partial Q_\delta(x_0)) = 0$
except for but countably many values of 
	$\delta > 0$, 
which will be avoided in the remainder of the proof.
Hence, since $F_\eps( J_\eps, \cdot) \to \nu$ vaguely as $\eps \to 0$, it follows from the Portmanteau theorem that, as $\eps \to 0$,
\begin{align*}
	F_\eps( J_\eps, Q_\delta(x_0))
	\to 
	\nu(Q_\delta(x_0)) \, .
\end{align*}
Since \eqref{eq:Leb_point_nu} implies 
\begin{align*}
	\frac{\nu(Q_\delta(x_0))}
			{\Lm^d(Q_\delta(x_0))}
		\to 
	f_0
	\, ,
\end{align*}
another diagonal argument
yields 
\begin{align}
	\label{eq:2-rigorous}
	\frac{
		F_\eps( J_\eps, Q_\delta(x_0))
		}
			{\Lm^d(Q_\delta(x_0))}	
	\to 
	f_0
			\, ,
\end{align}
whenever $\delta \to 0$ and 
$\eps = \eps(\delta) \to 0$
sufficiently fast.

\smallskip
\emph{Approximation 3.} \
Our goal is to apply 
Proposition~\ref{prop:asymptotic_cube_AC} 
to the rescaled graph $(\cX_s,\cE_s)$
with $s = \eps(\delta) / \delta$, 
the cube $Q_1(x_0 / \delta)$,
the momentum vector $j_0$,
and the approximating discrete vector field $K_\delta^\eps$.
For this purpose, we first bound the total variation of $\iota_{\eps/\delta} K_\delta^\eps$.

Since $\iota_\eps J_\eps \to \xi$ vaguely as $\eps \to 0$,
we have
\begin{align*}
	\xi_{\delta, x_0}^\eps
	= 
	\delta^{-d} {\rho_{\delta,x_0}}_\# (\iota_\eps J_\eps)
	\to 
	\delta^{-d} {\rho_{\delta,x_0}}_\# \xi
	\quad \text{ vaguely, as } \eps \to 0\, . 	
\end{align*}
In view of \eqref{eq:final_blowup_AC},
we have
\begin{align*}
	\delta^{-d} {\rho_{\delta,x_0}}_\# \xi
	\to 
	j_0 \Lm^d
	\quad \text{ vaguely, as } \delta \to 0 \, .
\end{align*}
Hence, by yet another diagonal argument, 
we infer that 
\begin{align}
\label{eq:xi-conv}
	\xi_{\delta, x_0}^{\eps}
	\to 
	j_0 \Lm^d 
	\quad \text{ vaguely }\, ,
\end{align}
whenever $\delta \to 0$ and 
$\eps = \eps(\delta) \to 0$
sufficiently fast.
Hence, by the uniform boundedness principle, 
\begin{align*}
	T := \sup_
	{\substack{\delta > 0\\ 
	\eps = \eps(\delta)}
	}
		| \iota_{\eps / \delta}
			K_\delta^\eps 
		|( Q_1(x_0/\delta))
	= 
	\sup_
		{\substack{\delta > 0\\ 
		\eps = \eps(\delta)}
		}
			| \xi_{\delta, x_0}^\eps 
			|( Q_1(0))
	< \infty \, .
\end{align*}

Let $\eta > 0$ be a cut-off lengthscale
satisfying $\eta <      
\frac12 \min
        \Big\{
            1
                ,
            \frac18 \frac{1 + | j_0 | }{T}
        \Big\}$,
so that the assumptions of 
Proposition~\ref{prop:asymptotic_cube_AC} hold 
whenever $\eps/\delta$ is sufficiently small, i.e.,~whenever      
$\max\{  \frac{\eps}{\delta} R_\partial, 
 	 		 \frac{\eps}{\delta} R_\Lip \}
         <
     \eta$.
Then Proposition~\ref{prop:asymptotic_cube_AC} ensures that
\begin{align}
	\label{eq:proof_AC_energy_error}
		 f_{\eps/\delta, \calR}
			\big(
				j_0 ,  Q_1(x_0/\delta)
			\big)
	\leq
		F_{\eps/\delta}
				(K_\eps^\delta, 
					Q_1(x_0/\delta))
		+ C \err_{\eps/\delta,\eta}
				(K_\eps^\delta, j_0)
			\,  ,
	\end{align}
	where
	\begin{align}
		\err_{\eps/\delta,\eta}(K_\eps^\delta, j_0)
			& =
		\frac1\eta 
					\| \dive \iota_{\eps / \delta}
						K_\delta^\eps  
					\|_{\tKR(\overline{Q_1(x_0/\delta)})}
			+
		\frac1{\eta^2} 
					\| \iota_{\eps/\delta} (K_\eps^\delta 
							- \calR_{\eps/\delta} j_0) 
					\|_{\tKR(\overline{Q_1(x_0/\delta)})}
	\\ & \quad
	  		+
			 \sqrt{\eta}  \bigl(  1 + | j_0 | + T \bigr)
			+
		\eps
		\Big(
			|j_0|
				+
			| \dive \iota_{\eps / \delta}
			K_\delta^\eps   |(\overline{Q_1(x_0/\delta)})
				+
			\frac{T}{\eta}
		\Big) \, .
	\end{align}
We will bound the terms on the right-hand side of this expression. 

First, we estimate 
	$| \dive \iota_{\eps / \delta}
			K_\delta^\eps   
	 |(\overline{Q_1(x_0/\delta)})$.
For this purpose, we observe that
\begin{align*}
	\dive \iota_{\eps / \delta} K_\delta^\eps
	= 
	\dive ({\rho_{\delta, 0}}_\# (\iota_\eps J_\eps))
	= 
	\delta^{1-d} {\rho_{\delta, 0}}_\# \dive ( \iota_\eps J_\eps)
	= 
	\delta^{1-d} {\rho_{\delta, 0}}_\# (m_\eps) \, .
\end{align*}			
Since we assume that
$|m_\eps| \to \sigma$
vaguely as $\eps \to 0$,
it follows from the Portmanteau theorem 
and \eqref{eq:tv_bound_mu_rescale_AC} 
that,
for all compact sets $B \subseteq \R^n$ and fixed $\delta > 0$,
\begin{align*}
	\limsup_{\eps \to 0} 
		|\dive \iota_{\eps / \delta} K_\delta^\eps|(B)
	= 
	\frac{1}{\delta^{d-1}}
	\limsup_{\eps \to 0} 
		| m_\eps |(\delta B)
	= 
	\frac{1}{\delta^{d-1}}
		\sigma(\delta B)
	\leq C_\sigma \delta \, .
\end{align*}
Hence, by another diagonal argument, whenever 
$\delta \to 0$ and $\eps = \eps(\delta) \to 0$ fast enough,
\begin{align}
\label{eq:zero-1}
	\lim_
		{\substack{\delta \to 0\\ 
				   \eps = \eps(\delta) }}
		|\dive \iota_{\eps / \delta} K_\delta^\eps|(B)
	= 
	0 \, .
\end{align}
Since total variation controls the $\tKR$-norm, we also obtain
\begin{align}
\label{eq:zero-2}
	\lim_
		{\substack{\delta \to 0\\ 
				   \eps = \eps(\delta) }}
				   \| \dive \iota_{\eps / \delta}
				   K_\delta^\eps  
			   \|_{\tKR(\overline{Q_1(x_0/\delta)})}
	= 
	0 \, .
\end{align}

Second, we estimate 
$\| \iota_{\eps/\delta} (K_\eps^\delta 
- \calR_{\eps/\delta} j_0) 
\|_{\tKR(\overline{Q_1(x_0/\delta)})}$
by showing that both 
	$\iota_{\eps/\delta} K_\eps^\delta$
and 	 
	$\iota_{\eps/\delta}\calR_{\eps/\delta} j_0$ 
are near $j_0 \Lm^d$.
Without loss of generality, we can also assume that
\begin{align}
\label{eq:assumption_TV_lb_AC}
    |\xi_{\delta,x_0}^\eps| \to \lambda \geq |j_0| \Lm^d
        \quad \text{vaguely}
        \, , \quad
    \text{as }\eps \to 0
        \, ,
\end{align}
if not we simply consider a suitable subsequence in $\eps \to 0$ (note that $|\xi_\eps|_\eps$ is locally bounded in total variation).  Moreover, we can also assume that $\lambda(\partial Q_1(0))=0$, if not we replace $Q_1(0)$ with $Q_h(0)$ with some $h<1$ and repeat the proof.
Using 
\eqref{eq:KR-scaling}, 
the translation invariance of $\Lm^d$, 
\eqref{eq:K-def},
\eqref{eq:xi-conv}, and $\lambda(\partial Q_1(0))=0$,
we find
\begin{equation}
\label{eq:iKjL}
	\begin{aligned}
		\| 
		\iota_{\eps/\delta} K_\eps^\delta 
			-  j_0 \Lm^d
	\|_{\tKR(\overline{Q_1(x_0 / \delta)})}
	& = 
	\| 
		{\rho_{1,x_0/\delta}}_\# (\iota_{\eps/\delta} K_\eps^\delta) 
			-  j_0 \Lm^d
	\|_{\tKR(\overline{Q_1(0)})}
	\\& =
	\| 
		\xi_{\delta, x_0}^\eps 
			-  j_0 \Lm^d
	\|_{\tKR(\overline{Q_1(0)})}
	\to 0  \, 
	\end{aligned}
\end{equation}
whenever $\delta \to 0$ and 
$\eps = \eps(\delta) \to 0$ as above.
Moreover, note that $\| \iota_{\eps/\delta} \calR_{\eps/\delta} j_0
- j_0 \Lm^d \|_{\tKR(\overline{Q_1^\delta})} \to 0$ when $\eps \to 0$, for all fixed $\delta > 0$, since $\calR$ is a uniform-flow operator.
Hence, by another diagonal argument, the same convergence holds when $\delta \to 0$ and $\eps = \eps(\delta) \to 0$ fast enough.
Combined with \eqref{eq:iKjL} we obtain using the triangle inequality,
\begin{align}
\label{eq:zero-3}
	\lim_
		{\substack{\delta \to 0\\ 
				   \eps = \eps(\delta) }}
	\| \iota_{\eps/\delta} (K_\eps^\delta 
		- \calR_{\eps/\delta} j_0) 
	\|_{\tKR(\overline{ Q_1(x_0/\delta)})}
	= 0 \, .
\end{align}
Inserting 
\eqref{eq:zero-1}, \eqref{eq:zero-2}, and \eqref{eq:zero-3}, 
we find, for fixed $\eta$ as above,
\begin{align*}
	\limsup_{\substack{\delta \to 0\\ 
	\eps = \eps(\delta) }}
		\err_{\eps/\delta,\eta}(K_\eps^\delta, j_0)
	& \leq
	\kappa(\eta)\, , \quad
\text{ where } 
	\kappa(\eta):= \sqrt{\eta}  \bigl(  1 + | j_0 | + T \bigr)
		\, ,
\end{align*}
so that, by \eqref{eq:proof_AC_energy_error},
\begin{align}
\label{eq:f-F}
	\lim_
		{\substack{\delta \to 0\\ 
		\eps = \eps(\delta)}}
	f_{\eps/\delta, \calR}
	\big(
		j_0 ,  Q_1(x_0/\delta)
	\big)
\leq
\lim_
		{\substack{\delta \to 0\\ 
		\eps = \eps(\delta)}}
F_{\eps/\delta}
		(K_\eps^\delta, 
			Q_1(x_0/\delta))
+ C \kappa(\eta)
	\,  .
\end{align}

\smallskip
\emph{Putting everything together.} \
Combining 
\eqref{eq:1-rigorous},
\eqref{eq:f-F},
\eqref{eq:scaling-F},
and \eqref{eq:2-rigorous},
we find
\begin{align*}
	f_{\hom}(j_0)
	 = 
	\lim_
		{\substack{\delta \to 0\\ 
		\eps = \eps(\delta)}}
	f_{\eps/\delta, \calR}(j_0,Q_1(x_0/\delta))
	& \leq
	\lim_
		{\substack{\delta \to 0\\ 
		\eps = \eps(\delta)}}
	F_{\eps/\delta}
		(K_\eps^\delta,
		Q_1(x_0 / \delta))
	+
	C \kappa(\eta)
	\\& =
	\lim_
		{\substack{\delta \to 0\\ 
		\eps = \eps(\delta)}}
		\frac{F_\eps(J_\eps,Q_\delta(x_0))}
		{\Lm^d(Q_\delta(x_0))}
	+
	C \kappa(\eta)
	= 
	f_0
	+
	C \kappa(\eta)
	\, .
\end{align*}
Since 
$
	\frac{\nu(Q_\delta(x_0))}{\Lm^d(Q_\delta(x_0))}
	\to 
	\frac{\ddd \nu}{\ddd \Lm^d}(x_0)$ 
	 for $\Lm^d$-a.e.~$x_0 \in \R^d$, 
the result follows by sending $\eta \to 0$.

\subsection{Proof for the singular part}

In this section, we show 
\eqref{eq:ub_blowup_recess},~i.e.,
\begin{align}	\label{eq:lb_singular}
	f_{\omega,\hom}^\infty
	\left(
		\frac{\de \xi}{\de |\xi|}
	\right)
		\leq
	\frac{\de \nu}{\de |\xi|}
	\qquad |\xi|^s \text{-almost everywhere in }\overline U.
\end{align}

In the first part of this section, we discuss some geometric constructions and the asymptotic behaviour of the discrete energies, that we will apply in the proof of \eqref{eq:lb_singular}. 
While performing a blow-up around a point in the singular part of $j$, due to the more complex structure of the tangent measure of $j$, it will become necessary to construct suitable correctors this time \textit{not} on the whole cube, but only in a thin strip with some specific orientation.

\begin{rem}[Oriented strips]
    The need of choosing a suitable thin strip with specific orientations around singular points appears in other works in literature, see for example \cite[Lemma~3.9]{Bouchitte-Fonseca-Leoni-Mascarenhas:2002}. We refer also to \cite{Ruf-Zeppieri:2023}, in particular the proof of $(5.20)$ therein: in that work, stochastic homogenisation of integral functions of the gradient with linear growth are considered, which corresponds to replace our divergence constraint with a curl one. The structure of these two problems are significantly different: from one side, a curl-free vector field is necessarily of gradient forms, which allows to work directly in terms of unconstrained problems in the space of scalar maps. Moreover, the shape of the tangent measures around singular points of the gradient of BV functions has a different (somehow simpler) structure, due to the celebrated rank-one theorem by Alberti \cite{Alberti:1993}.
\end{rem}

We start with some notation: for a given orthonormal basis 
	$\{e_1, \dots, e_d\}$ of $\R^d$, let 
$Q_1=Q_1(0)$ be the open cube centered at $0$ having sides of length $1$ parallel to $\{e_i\}_i$. In formula,
\begin{align}	\label{eq:cube}
	Q_1:=
		\left\{
			x = \sum_{i=1}^d x_i e_i
		\suchthat
			x_i \in \Big( - \frac1 2, \frac1 2 \Big)
				\, \ \text{for } 
					i = 1, \dots, d
		\right\}  \, .
\end{align}
For fixed $k \in \{1, \dots, d-1\}$
we consider the subspace 
 $L_k := \Span\{ e_{k+1}, \dots, e_d\}$ 
and its orthogonal complement 
	$L_k^\perp := \Span \{e_1, \dots, e_k\}$. 
Define the corresponding strip of size 
	$\alpha \in (0,1)$ as
\begin{align}
\label{eq:def_Ralpha}
\hspace{-8mm}  	R_\alpha:=
	\left\{
	x = \sum_{i=1}^d x_i e_i \in Q_1
	\suchthat
	x_i  \in \Big( - \frac\alpha2, \frac\alpha2 \Big)
	\quad \text{for} \ i \in \{ k+1, \dots, d \} \
	\right\} \, .
\end{align}
In other words, $R_\alpha$ is an open rectangle centered in $0 \in L_k$  with sides parallel to $L_k^\perp$ and $L_k$, where the first $k$ sides (parallel to $L_k^\perp$) have lenght $1$ and the other $d-k$ have length $\alpha$. See Figure~\ref{fig:strip} for a representation of this set.
\begin{figure}
    \centering
\includegraphics{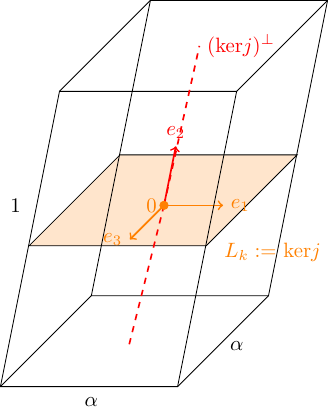}
    \caption{An oriented strip with sides paralell to $\ker j$ and $(\ker j)^\perp$.}
    \label{fig:strip}
\end{figure}

The next result is the key estimate in the proof of the lower bound around a singular point, and plays the same role as Proposition~\ref{prop:asymptotic_cube_AC} in the absolutely continuous part. We will apply this result to a sequence of discrete vector fields $J_\eps$ later that are converging to a tangent measure $\tau$ around a singular point. 
Such tangent measures are in general different from Lebesgue measure; their structure is described in Proposition~\ref{prop:structure_tangent_measures}.

In a similar spirit as for the absolutely continuous part, the next proposition shows a quantitative error estimates in terms of how far a discrete flux is from being a competitor of the cell formula, when working on a set which is a translation of a strip $R_\alpha$. This is one of the key estimates which we shall later employ to conclude the proof of the lower bound. This time is crucial to quantify the error in terms of how far the discrete flux is from a general tangent measure, as when performing the blow-up around a singular point of a divergence measure, a more complex tangent structure may arise, cfr. Proposition~\ref{prop:structure_tangent_measures}.

In view of Proposition \ref{prop:J0}, there exists a uniform-flow operator
	$\calR \in \mathrm{Lin}(V\otimes \R^d; V_\rma^\cE)$,
which we fix from now on.

\begin{prop}[Non-asymptotic behavior of the energies on a strip]
\label{prop:asymptotic_strip}
Fix $\eps \in (0,1)$.
Let $j \in V \otimes \R^d$ with $\emph{rank}(j) \leq n-1$.
For $\alpha \in (0,1)$, let $R:= R_\alpha + \bar{x}$, $\bar{x} \in \R^d$, where $R_\alpha$ is a strip as in \eqref{eq:def_Ralpha} with $k:= \dim (\ker j)$ and $L_k := \ker j$.
For every $J \in V_\rma^{\cE_\eps}$, let $\eta > 0$ be such that
\begin{align}
\label{eq:assumption_eta_quantitative_II}
    \frac1\alpha
    \max\{  \eps R_\partial, \eps R_\Lip \}
        <
    \eta
        \leq \max\{\eta,\alpha\}
    \EEE
        <
    \min
        \Big\{
            \frac13
                ,
            \frac{(1 + |j|){\Lm^d(R)}}{16|\iota_\eps J|(R)}
        \Big\}
        \, .
\end{align}
Then we have
\begin{align}	\label{eq:lowerbound_prop_strip}
        f_{\eps,\calR}
        \big(
            j ,  R
        \big)
\leq
		F_\eps(J,R)
    + C \err_{\eps,\eta,\alpha}^\tau(J, j)
		\,  ,
\end{align}
where $C < \infty$ only depends on the constants $R_i$, $C_i$, $c_i$ appearing in the assumptions on $(\cX,\cE)$ and $F$, and where 
\begin{align}
	\err_{\eps,\eta,\alpha}^\tau(J, j) &
        :=
    \frac1{\alpha^2 \eta} \| \dive \iota_\eps J \|_{\tKR(\overline R)}
        +
    \frac1{(\alpha^2 \eta)^2}
    \bigg(
        \| \iota_\eps J - \tau\|_{\tKR(\overline R)}
            +
        \| \iota_\eps \calR_\eps j - j \Lm^d)\|_{\tKR(\overline R)}
    \bigg)
\\
        &+
    \eps
    \Big(
        |j| {  \Lm^d(R)}
            +
        | \dive \iota_\eps J |(\overline R)
            +
        \frac1{\alpha^2 \eta}
        | \iota_\eps J| (\overline R)
    \Big)
        \\& +
    (\sqrt{\eta} + \sqrt{\alpha} ) \Big( (1+  | j |) {  \Lm^d(R)}  + | \iota_\eps J |(\overline R) \Big)
    \\
        &+
     \frac{\sqrt{\alpha}}\eta  \Big(  (1+ |j|) {  \Lm^d(R)} + |\iota_\eps J|(\overline R) \Big)
 \, ,
\end{align}
for $\tau$ being any measure satisfying, for $|\tau|$-a.e. $x \in \R^d$,
\begin{align}
\label{eq:tangent_measure_j0}
    \frac{\dd \tau }{\dd |\tau|}(x)
        =
    \frac{j}{| j |}
        \, , \quad
    \dive \tau =0
        \, ,
        \tand
    |\tau|(\overline R) = | j | \Lm^d(\overline R) = | j | \alpha^k
        \, .
\end{align}
\end{prop}

\begin{rem}\label{rem:lack_bdry_singular}
It is interesting to note the differences between the setting of Proposition~\ref{prop:asymptotic_strip} and Proposition~\ref{prop:asymptotic_cube_AC}: we will apply it to a  tangent measure $\tau$ which is divergence free, and therefore it is in product form, but it does not necessarily coincide with the Lebesgue measure in every direction (cfr. Proposition~\ref{prop:structure_tangent_measures}). In particular,  whenever $\iota_\eps J \approx \tau$, the discrete field $K$ is not expected to be approximately constant in the direction of $\ker j$. This is why we need a thin strip $R_\alpha$ (of side-length $\alpha\ll1$) parallel to $\ker j$, where we have no information about $\tau$ (cfr. Figure~\ref{fig:strip}). The price we pay to correct the discrete fields at the boundary is of order $\alpha$, which also shows in \eqref{eq:lowerbound_prop_strip}. At the same time, the orientation of the strip is also crucial to ensure that, while fixing the right boundary conditions, we do not create extra divergence, see in particular \eqref{eq:final_claim_SIN} below.
\end{rem}

\begin{proof}
As done in Proposition~\ref{prop:asymptotic_cube_AC}, the plan of the proof is to replace the vector field $K$ by an exact $\eps$-representative $J_3 \in \Rep_{\eps,\calR}(j;R)$,
and show that the error can be controlled in terms of its divergence and the distance from \textit{any} tangent measure $\tau$ with $|\tau|(\overline R) = |j| \Lm^d(\overline R) = |j| \alpha^k$ with density $\frac{\dd \tau}{\dd |\tau|} = \frac{j}{| j |}$.
In a similar way as in Proposition~\ref{prop:asymptotic_cube_AC}, we will perform two corrections: first we  correct the boundary values and then the divergence.
In order to obtain the right boundary conditions, this time we will make use of two different cut-off functions at the boundary of $R$: one for the sides parallel to $(\ker j)^\perp$ and one for the sides parallel to $\ker j$. While performing these corrections, we need to control the extra divergence that we create (due for example to the gradient of the cut-off functions). 
The way we control the extra divergence, in contrast with Proposition~\ref{prop:asymptotic_cube_AC}, is very different between the two cut-off operations: for the sides which are parallel to $(\ker j)^\perp$, 
we take advantage of the fact that the gradient of the cutoff function belongs to $\ker j$, hence it enjoys good orthogonality properties with the tangent measures associated with $j$ (see "Part 1.1" in Step 4). 
In the other case, we control the divergence in $\tKR$-norm linearly in $\alpha$, thanks to the fact that tangent measures coincide with the Hausdorff measure when restricted to $(\ker j)^\perp$ (see "Part 1.2" in Step 4), thanks to Proposition~\ref{prop:structure_tangent_measures}.

As in Proposition~\ref{prop:asymptotic_cube_AC}, 
when performing boundary corrections, we have some freedom in the choice of where to perform the cutoff. Its location needs to be optimised to obtain the  claimed error estimate.

\medskip
\noindent
\emph{Step 1 (Boundary value correction)}. \

Recall that $R = \bar{x} + R_\alpha$, with $R_\alpha$ represented in Figure~\ref{fig:strip}.
Fix $\eta > 0$ satisfying
	\eqref{eq:assumption_eta_quantitative_II} and
$N \in \N$ so that $2N + 1 < \frac1{\eta}$. The value of $N$ will be optimised below.  Fix also $\eta' >0$ that we will later choose as a function of $\alpha$, also satisfying \eqref{eq:assumption_eta_quantitative_II}, and fix $N^\perp \in \N$ so that $2N^\perp +1 < \frac1{\eta'}$. The value of $N^\perp$ will too be optimised below.

We denote by $Q_\lambda^k = Q_\lambda^k(\bar{x}^\parallel)$ the cube on $\ker j \simeq \R^k$ of side length $\lambda \in \R_+$ and center $\bar{x}^\parallel$.
Similarly, we denote by $Q_\lambda^{d-k} = Q_\lambda^{d-k}(\bar{x}^\perp)$ the cube in $(\ker j)^\perp \simeq \R^{d-k}$.

\smallskip
(1)  \
For $\ell = 1, \dots, N$,
let $\tilde \psi_\eta^\ell \in \cC_\rmc^\infty(\R^{d-k})$
be a cutoff function
satisfying
	$0 \leq \tilde \psi_\eta^\ell \leq 1$,
	$\Lip (\tilde \psi_\eta^\ell) \leq C/\eta$,
	$\Lip (\nabla \tilde \psi_\eta^\ell) \leq C / \eta^2$,
and such that
\begin{align*}
	\tilde \psi_\eta^\ell (x) = 1
		\quad \text{for} \ x \in Q_{1-2\ell\eta}^{d-k}
	\tand
	\tilde \psi_\eta^\ell (x) = 0
		\quad \text{for} \ x \notin Q_{1-(2\ell-1)\eta}^{d-k} \, .
\end{align*}
We then define $\psi_\eta^\ell:\R^d \to [0,1]$ as  $\psi_\eta^\ell(x_1, \dots, x_d) = \tilde \psi_\eta^\ell(x_{k+1}, \dots, x_d)$. In particular,
\begin{align}
    \nabla \psi_\eta^\ell(x) \in (\ker j)^{\perp}
         \quad
    \forall x \in \R^d
    \, .
\end{align}

\smallskip
(2)  \
For $\ell' = 1, \dots, N^\perp$, let $\tilde \psi_{\eta'}^{\ell',\perp} \in \cC_\rmc^\infty(\R^k)$
be a cutoff function
satisfying
	$0 \leq \tilde \psi_{\eta'}^{\ell',\perp} \leq 1$,
	$\Lip (\tilde \psi_{\eta'}^{\ell',\perp}) \leq C/(\alpha\eta)$,
	$\Lip (\nabla \tilde \psi_{\eta'}^{\ell',\perp}) \leq C / (\alpha \eta)^2$,
and such that
\begin{align*}
	\tilde \psi_{\eta'}^{\ell',\perp} (x) = 1
		\quad \text{for} \ x \in Q_{(1-2\ell\eta')\alpha}^k
	\tand
	\tilde \psi_{\eta'}^{\ell',\perp} (x) = 0
		\quad \text{for} \ x \notin Q_{(1-(2\ell-1)\eta')\alpha}^k \, .
\end{align*}
We then define $\psi_{\eta'}^{\ell',\perp}:\R^d \to [0,1]$ as  $\psi_{\eta'}^{\ell',\perp}(x_1, \dots, x_d) = \tilde \psi_{\eta'}^{\ell',\perp}(x_1, \dots, x_k)$. In particular,
\begin{align}
    \nabla \psi_{\eta'}^{\ell',\perp}(x) \in \ker j
         \quad
    \forall x \in \R^d
    \, .
\end{align}

To ensure the correct boundary data, we define new vector fields in $V_\rma^{\cE_s}$ as
\begin{align}	\label{eq:J_2_SIN}
	H_2^\ell :=
	\hat \psi_\eta^\ell
	 J
	+ (1-\hat\psi_\eta^\ell ) \calR_\eps j
		\quad \tand \quad
	J_2^\ell:= \hat \psi_{\eta'}^\perp  H_\eps^2 + (1-\hat\psi_{\eta'}^\perp )  \calR_\eps j
			\, ,
\end{align}
where the rescaled operator
$\calR_\eps$ has been defined in
 Definition~\ref{def:admissible}
and $\hat\psi^\ell$ has been defined in
	\eqref{eq:def_hat_0}.

The construction is similar to the one used in the proof of the liminf inequality for the absolutely continuous part, except that now we use two different mesoscopic scales $\eta, \eta'$ to perform different cutoff in the directions of $\ker j$ and $(\ker j)^\perp$, see Figure~\ref{fig:cutoff_singular}.

\begin{figure}[h]
    \centering
    \includegraphics[scale=0.9]{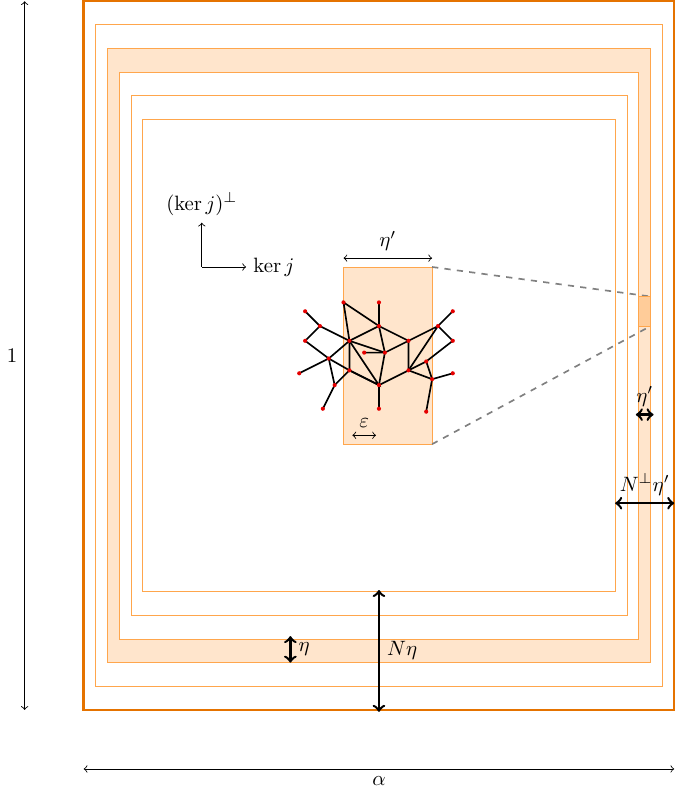}
    \caption{Around a singular point, we perform a cutoff procedure close to the boundary of the oriented strip $R_\alpha$. On the sides parallel to $\ker j$, we choose a mesoscopic scale $\eta \gg \eps$, whilst for the faces paralell to $(\ker j)^\perp$ we pick a mesoscopic scale $\eta' \gg \eps$, chosen in such a way that $\eta'\ll \eta$.}
    \label{fig:cutoff_singular}
\end{figure}

Note that $H_2^\ell$ has the right boundary condition on $R_\alpha$ \textit{only} on the facets that are parallel to $\ker j$, whereas $J_2^\ell$ has the right boundary condition also on the remaining facets.

It is useful to introduce the notation: $\Psi_{\eta,\eta'}^{\ell,\ell'}= \psi_\eta^\ell \psi_{\eta'}^{\ell',\perp}$. In particular, by construction we have
 \begin{align}
 \label{eq:J_2_SIN_2}
     J_2^{\ell,\ell'}:= \hat \Psi_{\eta,\eta'}^{\ell,\ell'} J + (1-\hat\Psi_{\eta,\eta'}^{\ell,\ell'})  \calR_\eps j
        \, .
 \end{align}

Arguing as in Step 1 in the proof of Proposition~\ref{prop:asymptotic_cube_AC}, thanks to \eqref{eq:assumption_eta_quantitative_II}, we will show we can suitably choose $N$, $\bar\ell\in \{ 1, \dots, N\}$, and $\bar\ell'\in \{ 1, \dots, N^\perp\}$
so that the vector field $J_2^{\bar\ell,\bar\ell'}$
has three desirable properties:
it has the right boundary conditions
to be a representative in the sense of Definition \ref{def:rep}
and
its energy is controlled by the energy of $K$.
More precisely:
\begin{align}
	&\text{Claim 1a:} &&
	J_2^{\bar\ell,\bar\ell'}(x,y) = \calR_\eps j(x,y)
		\text{ for all }
		(x,y) \in \cE_\eps
		\text{ with }
		\dist\bigl([x,y], R^c \bigr)
		\leq \eps R_\partial \;,
	\\
	&\text{Claim 1b:} &&
    F_\eps(J_2^{\bar\ell,\bar\ell'},R)
        -
    F_\eps(J,R)
	\lesssim
	\big(
            \sqrt\eta + \sqrt{\eta'}
        \big)
        \Big( (1+ | j |) {  \Lm^d(R)}
		   	+ | \iota_\eps J |(\overline R)\Big)
        \, .
    \\
        &\text{Claim 1c:} &&
        |\iota_\eps J|
        \big(
            \supp
            \big(
                \psi_{\eta'}^{\bar\ell',\perp} (1-\psi_{\eta'}^{\bar\ell',\perp})
            \big)
        \big)
            \lesssim
         \sqrt{\eta'} \Big( (1+ | j |) {  \Lm^d(R)}
		   	+ | \iota_\eps J |(\overline R)\Big)
        \, .
\end{align}
Having obtained such $\bar{\ell}, \bar\ell'$, we simplify notation by writing $J_2 := K^{2,\bar\ell,\bar\ell' }$ (and generally omit the dependence on $\bar{\ell},\bar\ell'$ for every cutoff function).

\bigskip
\noindent
\textit{Step 2 (Divergence correction)}. \
As in Proposition~\ref{prop:asymptotic_cube_AC}, the next step is to find a corrector to $J_2$ so that the new vector field is divergence free, while preserving the right boundary conditions. Using that $\iota_\eps J =0$ on $B(R^c,\eps R_\partial)$ (by \eqref{eq:assumption_eta_quantitative_II} and $\alpha \in (0,1)$), we employ Proposition~\ref{prop:correctors} once again and find $C \in V_\rma^{\cE_\eps}$ so that $\DIVE K = -\DIVE J_2$ on $(\cX_\eps, \cE_\eps)$, and $c=c(d)<\infty$, such that (for $\eps$ small enough)
\begin{align}	\label{eq:tv_bound_Cs_SIN}
	\begin{cases}
	\iota_\eps K = 0  \quad \text{on } B(R^c,\eps R_\Lip)
	\, ,
		\\
	| \iota_\eps K |(\R^d)
		\lesssim
		\| \iota_\eps  \DIVE J_2 \|_{\tKR(\overline R)} + \eps |\iota_\eps \DIVE J_2 |(\overline R)	\, ,
	\end{cases}
\end{align}
where at last we used the Remark~\ref{rem:extension_Lip}.
Finally, we define the competitor vector field as
$
	J_3 := J_2 + C \in V_\rma^{\cE_\eps}
$.
We will show that

\bigskip
\noindent
\quad \,  Claim 2a: \
$
		J_3 \in \Rep_{\eps, \calR}(j; R)
$ ,

\bigskip
\noindent
\quad \,  Claim 2b: \
we have the following bound:
 \begin{align}
    | \iota_\eps(J_3 - J_2)|(\R^d)
		&\lesssim
	 \frac1{\alpha \eta \eta'} \| \dive \iota_\eps J \|_{\tKR(\overline R)}
\\
        &+
    \frac1{(\alpha \eta \eta')^2}
    \bigg(
        \| \iota_\eps J - \tau\|_{\tKR(\overline R)}
            +
        \| \iota_\eps \calR_\eps j - j \Lm^d)\|_{\tKR(\overline R)}
    \bigg)
\\
        & +
    \eps
    \Big(
        | \dive \iota_\eps J |(\overline R)
            +
        \frac1{\alpha \eta \eta'}
        | \iota_\eps (J - \calR_\eps j)| (\overline R)
    \Big)
\\
        &+
    \frac{\sqrt{\eta'}}\eta  \Big( (1+ | j |) {  \Lm^d(R)}
		   	+ | \iota_\eps J |(\overline R)\Big)
    +  \frac\alpha\eta |j|
       {  \Lm^d(R)}
 \, ,
\end{align}
for every measure $\tau$ satisfying \eqref{eq:tangent_measure_j0}.

\bigskip
\noindent
\textit{Step 3: the energy estimate} \
Thanks to Lemma~\ref{lemma:tv_bounds}, we control the error we make by going from $K$ to $J_3$ quantitatively in terms of total variation, and get the lower bound
\begin{align}	\label{eq:energy_est_s_sing_2}
	F_\eps(J_2, R)
		&\geq
	F_\eps(J_3, R)
		-
  C
	|\iota_\eps(J_2 - J_3) |(B(R,\eps R_{\Lip}))
\\
		&\geq
	f_{\eps,\calR}
		\big(
			j , R
		\big)
		-
C	|\iota_\eps(J_2 - J_3) |(\R^d)
		\, ,
\end{align}
where in the second inequality we used that
	$J_3 \in \Rep_{\eps,\calR}(j; R_\alpha) $.
We then apply the estimates obtained in Claim 1b and Claim 2b and conclude the proof by choosing for example $\eta' = \alpha$. Note that this choice of $\eta'$ satisfies the required \eqref{eq:assumption_eta_quantitative_II}.

It remains to prove the Claims.
\bigskip

\smallskip
\emph{Proof of Claim 1a.} \
This follows by mimicking the argument from  Proposition~\ref{prop:asymptotic_cube_AC}.

\smallskip
\emph{Proof of Claim 1b and 1c.} \
We adapt the proof of Claim 1b in Proposition~\ref{prop:asymptotic_cube_AC}.
This time we need to optimise both in the direction of $\ker j_0$ and its orthogonal complement,
choosing different scaling limits for $N$ and $N^\perp$, respectively depending on $\eta$ and $\eta'$, and hence on $\alpha$.

For every $\ell \in \{1, \dots, N \}$ and $\ell' \in \{ 1, \dots, N^\perp\}$,
we decompose the strip  $R$ as disjoint union
	 $R = A_1^{\ell,\ell'} \cup A_\tra^{\ell,\ell'} \cup A_0^{\ell,\ell'}$,
where $A_1^{\ell,\ell'}$ is a smaller concentric strip surrounded
by $L^\infty$-spherical shells $A_\tra^{\ell,\ell'}$ and $A_0^{\ell,\ell'}$.
The first set is chosen in product form $A_1^{\ell,\ell'}
        =
A_1^{\ell} \times A_1^{\ell'}
$, with respect to the decomposition $\ker j_0 \oplus (\ker j_0)^\perp$, and overall they are chosen in such a way that $A_1^{\ell,\ell'}$ lies well inside the set $\{ \Psi_{\eta,\eta'}^{\ell,\ell'} = 1 \}$,
	$A_0^{\ell,\ell'}$  lies well inside the set $\{ \Psi_{\eta,\eta'}^{\ell,\ell'} = 0 \}$,
	and $A_\tra^{\ell,\ell'}$is a transition zones (see Figure~\ref{fig:cutoff_singular}).
More precisely,
\begin{align}
    A_1^\ell
		:= Q_{1-(2\ell+1)\eta}^{d-k}
        \, , \quad
     A_1^{\ell'}
		:= Q_{(1-(2\ell+1)\eta')\alpha}^k
        \, , \quad
    A_1^{\ell,\ell'}
        :=
    A_1^{\ell'} \times  A_1^{\ell}
\\
	A_\tra^{\ell,\ell'}
		=
        \Big(
        Q_{(1-(2\ell - 2) \eta')\alpha}^k
            \times
        Q_{1-(2\ell - 2) \eta}^{d-k}
        \Big)
        \setminus  A_1^{\ell,\ell'}
		\, , \quad
	A_0^{\ell,\ell'}
		:= R \setminus \big( A_\tra^\ell \cup A_1^{\ell,\ell'} \big)
        \, ,
\end{align}
Note that $1 -(2\ell+1)\eta>0$ (resp. $1 -(2\ell'+1)\eta'>0$) since we assumed $2N+1 < \frac1{\eta}$ (resp. $2N^\perp+1 < \frac1{\eta'}$ ).

By additivity of $F_\eps$, which follows from $(F3)$, we have
\begin{align}
    F_\eps(J_2^{\ell,\ell'},R)
        &=
    F_\eps(J_2^{\ell,\ell'},A_1^{\ell,\ell'})
        +
    F_\eps(J_2^{\ell,\ell'},A_\tra^{\ell,\ell'})
        +
    F_\eps(J_2^{\ell,\ell'},A_0^{\ell,\ell'}) \
\\
\label{eq:energy_threeparts_SIN}
        &\leq
    F_\eps(J,R)
        +
    F_\eps(J_2^{\ell,\ell'},A_\tra^{\ell,\ell'})
        +
    F_\eps(J_2^{\ell,\ell'},A_0^{\ell,\ell'})
        \, .
\end{align}
by arguing exactly as in Proposition~\ref{prop:asymptotic_cube_AC}(Claim 1b).
From this, we will proceed by estimating the remaining two terms.

Arguing as in \eqref{eq:boundary},  by \eqref{eq:assumption_eta_quantitative_II} one shows that
\begin{align}
\label{eq:choice_cutoff_zero_estimate_SIN}
    F_\eps(J_2^{\ell,\ell'},A_0^{\ell,\ell'})
        \lesssim
    \big(
        \Lm^d
            +
        |\iota_\eps \calR_\eps  j|
    \big)
    (A_0^{\ell+1,\ell'+1})
        \lesssim
    (1 + | j |) {  \Lm^d(R) } (N \eta + N^\perp \eta)
        \, , \quad
\end{align}
where at last we used that  $A_0^{\ell+1, \ell'+ 1}$ is a union of orthotopes and Remark~\ref{rem:vague_J0_cubes}.

Concerning the transition set $A_\tra^{\ell,\ell'}$, we argue as for \eqref{eq:transition} and obtain
\begin{align}
    F_\eps(J_2^{\ell,\ell'},A_{\tr}^{\ell,\ell'})
        \lesssim
    |\iota_\eps J|(\tilde A_\tra^{\ell,\ell'})
        +
    (1 + | j | ) {  \Lm^d(R) } (\eta + \eta')
        \, ,
\end{align}
where we defined
\begin{align}
\label{eq:def_Atilde_double}
    \tilde A_\tra^{\ell,\ell'} := (A_0^{\ell-1, \ell' -1})^c\setminus A_1^{\ell+1, \ell'+1}
        \, , \quad
        \text{for which }
    A_\tra^{\ell,\ell'}
        \subset
    B(A_\tra^{\ell,\ell'}, \eps R_\Lip)
        \subset
    \tilde A_\tra^{\ell,\ell'}
        \, . \quad 
\end{align}

Putting this together with \eqref{eq:choice_cutoff_zero_estimate_SIN}, we end up with
\begin{align}
\label{eq:choice_cutoff_second_estimate_SIN}
    F_\eps(J_2^{\ell,\ell'},A_0^{\ell,\ell'})
        +
    F_\eps&(J_2^{\ell,\ell'}, A_\tra^{\ell,\ell'})
\\
        &\lesssim
    |\iota_\eps J|(\tilde A_\tra^{\ell,\ell'})
        +
    (1+ | j | ) {  \Lm^d(R) }  (N \eta + N^\perp \eta)
        \, ,
\end{align}
for every $\ell \in \{ 1, \dots, N \}$, $\ell' \in \{ 1 , \dots, N^\perp \}$.

We are left with the choice of $\ell$, $\ell'$ to ensure that Claim 1b and 1c hold. To this purpose, in a similar spirit as in \eqref{eq:def_Atilde_double}, it is also useful to introduce a notation to denotes transition layers in the direction of $\ker j_0$ and its orthogonal: we thus define
\begin{align}
\label{eq:def_orizontal_vertical_layers}
    \tilde A_\tra^{\ell'}
        =
        \Big(
        Q_{(1-(2\ell' - 4) \eta')\alpha}^k
             \setminus
        A_1^{\ell'+1}
        \Big)
            \times
        Q_1^{d-k}
            \, , \quad
    \tilde A_\tra^\ell
     =
    Q_1^k
        \times
     \Big(
    Q_{1-(2\ell - 4) \eta}^{d-k}
         \setminus
    A_1^{\ell+1}
    \Big)
        \, ,
\end{align}
and observe that by construction
\begin{align}
\label{eq:inclusion_for_1c}
    \tilde A_\tra^{\ell,\ell'}  \subset \tilde A_\tra^{\ell'} \cup  \tilde A_\tra^\ell
        \tand
    \supp
    \big(
        \psi_{\eta'}^{\ell',\perp} (1-\psi_{\eta'}^{\ell',\perp})
    \big)
        \subset
    \tilde A_\tra^{\ell'}
        \, .
\end{align}

Using  that the sets $\{ \tilde A_\tra^{\ell'} \suchthat \ell' =1, \dots, N^\perp \}$, as well as the sets $\{ \tilde A_\tra^\ell \suchthat \ell =1, \dots, N \}$,  overlap a finite number of times, we now first choose  $\bar{\ell'} \in \{1, \dots, N^\perp \}$ so that
\begin{align}
\label{eq:mass_control_1c}
    |\iota_\eps J|(\tilde A_\tra^{\bar\ell'})
        \leq
    \frac1{N^\perp}
        \sum_{\ell'=1}^{N^\perp}
            |\iota_\eps J|(\tilde A_\tra^{\ell'})
        \lesssim
    \frac1{N^\perp}
        |\iota_\eps J|(R)
            \, ,
\end{align}
and then choose $\bar\ell \in \{1, \dots, N \}$ so that
\begin{align}
    |\iota_\eps J|(\tilde A_\tra^{\bar\ell})
        \leq
    \frac1{N}
        \sum_{\ell=1}^{N}
            |\iota_\eps J|(\tilde A_\tra^\ell)
        \lesssim
    \frac1{N}
        |\iota_\eps J|(R)
            \, .
\end{align}
Together with \eqref{eq:choice_cutoff_second_estimate_SIN}, this provides
\begin{align}
\label{eq:choice_cutoff_third_estimate_SIN}
    F_\eps(J_2^{\bar\ell,\bar\ell'},A_0^{\bar\ell,\bar\ell'})
        +
    F_\eps&(J_2^{\bar\ell,\bar\ell'},A_\tra^{\bar\ell,\bar\ell'})
\\
        &\lesssim
    \Big(
        \frac1N
            +
        \frac1{N^\perp}
    \Big)
        |\iota_\eps J|(R)
        +
    (1+ | j | ) \Lm^d(R)
    \big(
        N \eta
            +
        N^\perp \eta'
    \big)
        \, .
\end{align}

Finally we choose as $N$ and $N^\perp$ the values
\begin{align}
    N^2 = \frac1\eta \frac{|\iota_\eps J|(R)}{(1+ | j |)\Lm^d(R) }
        \tand
    (N^\perp)^2 = \frac1{\eta'} \frac{|\iota_\eps J|(R)}{(1+ | j |)\Lm^d(R)}
        \, ,
\end{align}
so that from
\eqref{eq:energy_threeparts_SIN} and \eqref{eq:choice_cutoff_third_estimate_SIN} we obtain
\begin{align}
    F_\eps(J_2^{\bar\ell,\bar\ell'},R)
        &\leq
    F_\eps(J,R)
        +
    (
        \sqrt{\eta}
            +
        \sqrt{\eta'}
    )
    \sqrt
        {
        |\iota_\eps J |(R) (1+ | j |)\Lm^d(R)
        }
\\
        &\lesssim
    (
        \sqrt{\eta}
            +
        \sqrt{\eta'}
    )
    \big(
        |\iota_\eps J |(R)  + (1+ | j |)\Lm^d(R)
    \big)
        \, ,
\end{align}
which concludes the proof of Claim 1b.

Similarly, Claim 1c follows from \eqref{eq:inclusion_for_1c} and \eqref{eq:mass_control_1c}, as by definition of $N^\perp$
\begin{align}
    |\iota_\eps J|
        \big(
            \supp
            \big(
                \psi_{\eta'}^{\bar\ell',\perp} (1-\psi_{\eta'}^{\bar\ell',\perp})
            \big)
        \big)
        \leq
    |\iota_\eps J|
        (  \tilde A_\tra^{\bar\ell'}  )
        \lesssim
  \sqrt{\eta'}
    \big(
        |\iota_\eps J |(R)  + (1+ | j |)\Lm^d(R)
    \big)
        \, .
\end{align}
Finally, note that $N$, $N^\perp$ are so that $2N + 1 < \frac1\eta$ as well as $2N^\perp +1 < \frac1{\eta'}$.

\smallskip
\emph{Proof of Claim 2a.} \
By construction, $\DIVE J_3 = \DIVE J_2 + \DIVE C=0$, and $J_3$ satisfies the same property of $J_2$ as in Claim 1a, due to the fact that $C \equiv 0$ on $B(R^c,\eps R_\partial) $.
\qed

\smallskip
\emph{Proof of Claim 2b.} \
By construction and \eqref{eq:tv_bound_Cs_SIN}, we have that
\begin{align}
    | \iota_\eps(J_3 - J_2)|(\R^d)
        =
    | \iota_\eps(C)|(\R^d)
        \lesssim
		\| \iota_\eps  \DIVE J_2 \|_{\tKR(\overline R)} + \eps |\iota_\eps \DIVE J_2 |(R)
    \, .
\end{align}
Similarly as in \eqref{eq:DIVE_J2},
by \eqref{eq:J_2_SIN_2} we have
\begin{align}	\label{eq:DIVE_J2_SIN}
	\DIVE J_2 = \Psi_{\eta,\eta'}  \DIVE J
		+
	\grad \Psi_{\eta,\eta'} \star
		\Big(
			J - \calR_\eps j
		\Big) \, .
\end{align}
The first term can be estimated by $\Lip(\Psi_{\eta,\eta'}) \lesssim \frac1{\alpha \eta \eta'}$, using the bound
\begin{align}
\label{eq:K2_estimate_KR_1}
    \| \Psi_{\eta,\eta'} \cdot \dive \iota_\eps J  \|_{\tKR(\overline R)}
        \lesssim
    \frac1{\alpha \eta \eta'}
     \| \dive \iota_\eps J  \|_{\tKR(\overline R)}
        \, ,
\end{align}
as well as the inequality
\begin{align}
\label{eq:K2_estimate_TV_1}
    | \Psi_{\eta,\eta'} \cdot \dive \iota_\eps J  |(R)
        \leq
     | \dive \iota_\eps J  |(R)
        \, ,
\end{align}
which holds since $\| \Psi_{\eta,\eta'} \|_\infty \leq 1$.
Concerning the second term in \eqref{eq:DIVE_J2_SIN}, we use \eqref{eq:gradients_comparison} in Lemma~\ref{lemma:gradients} to infer that
\begin{align}
\label{eq:perp_estimate}
	\| \iota_\eps \big( \grad \Psi_{\eta,\eta'} \star (J - \calR_\eps j ) \big) \|_{\tKR(\overline R)}
        \lesssim
    &\| \nabla \Psi_{\eta,\eta'} \cdot \iota_\eps (J - \calR_\eps j )\|_{\tKR(\overline R)}
\\
	&+ \eps
	| \nabla \Psi_{\eta,\eta'}  \cdot \iota_\eps (J - \calR_\eps j )|(\overline R)
		\, .
\end{align}
Using the Lipschitz bound on $\Psi_{\eta,\eta'}$ once again, we estimate the second term as
\begin{align}
\label{eq:perp_estimate_2}
    | \nabla \Psi_{\eta,\eta'}  \cdot \iota_\eps (J - \calR_\eps j )|(\overline R)
        \lesssim
    \frac1{\alpha \eta \eta'}
        | \iota_\eps (J - \calR_\eps j )|(\overline R)
        \, .
\end{align}
A similar argument shows that
\begin{align}
\label{eq:K2_estimate_TV_2}
    | \iota_\eps \big( \grad  \Psi_{\eta,\eta'}  \cdot  (J - \calR_\eps j ) \big) |(\overline R)
        \lesssim
    \frac1{\alpha \eta \eta'}
        | \iota_\eps (J - \calR_\eps j )|(\overline R)
        \, .
\end{align}
Combining
\eqref{eq:DIVE_J2_SIN}, \eqref{eq:K2_estimate_KR_1}, \eqref{eq:K2_estimate_TV_1}, \eqref{eq:perp_estimate}, \eqref{eq:perp_estimate_2}, and \eqref{eq:K2_estimate_TV_2},
the proof of Claim 2b will be completed by showing that
\begin{align}
    \| \nabla \Psi_{\eta,\eta'} \cdot \iota_\eps (J - &\calR_\eps j )\|_{\tKR(\overline R)}
     \lesssim
    \frac1{\eta(\alpha \eta')^2}
    \bigg(
        \| \iota_\eps J - \tau\|_{\tKR(\overline R)}
            +
        \| \iota_\eps \calR_\eps j - j \Lm^d)\|_{\tKR(\overline R)}
    \bigg)
\\
    \label{eq:K2_estimate_KR_2}
        & \quad +
    \frac{\sqrt{\eta'}}\eta  \Big( (1+ | j |) {  \Lm^d(R)}
		   	+ | \iota_\eps J |(\overline R)\Big)
    +  \frac\alpha\eta |j|
       {  \Lm^d(R)}
        \, .
\end{align}
For this purpose, we write $\nabla \Psi_{\eta,\eta'} = \psi_\eta \nabla \psi_{\eta'}^\perp + \psi_{\eta'}^\perp \nabla \psi_\eta$. Concerning the first term, we use that $\nabla \psi_{\eta'}^\perp(x) \in \ker j$.  It follows that
$
    \nabla \psi_{\eta'}^\perp \cdot \tau = 0
$ for every measure $\tau$ as in \eqref{eq:tangent_measure_j0}.
In particular, this identify holds for $\tau = j \Lm^d$.
Therefore, for every $\tau$ satisfying \eqref{eq:tangent_measure_j0}  we have
\begin{align}
    \| \psi_\eta \nabla \psi_{\eta'}^\perp \cdot \iota_\eps (J - \calR_\eps j )&\|_{\tKR(\overline R)}
        \leq
    \Big\| \psi_\eta \nabla \psi_{\eta'}^\perp \cdot \Big(
        (\iota_\eps J - \tau )
            +
        (\iota_\eps \calR_\eps j - j\Lm^d )
    \Big) \Big\|_{\tKR(\overline R)}
\\
\label{eq:final_est_SIN_1}
        &\lesssim
    \frac1{\eta(\alpha \eta')^2}
    \Big(
        \| \iota_\eps J - \tau \|_{\tKR(\overline R)}
            +
        \| \iota_\eps \calR_\eps j - j\Lm^d \|_{\tKR(\overline R)}
    \Big)
        \, ,
\end{align}
where at last we used that $\Lip(\psi_\eta \nabla \psi_{\eta'}^\perp) \lesssim     \frac1{\eta(\alpha \eta')^2} $.

We are left with the estimate involving $\psi_{\eta'}^\perp \nabla \psi_\eta$. Note that $\nabla \psi_\eta \in (\ker j)^\perp$ and not in $\ker j$, therefore we cannot simply substitute $\iota_\eps \calR_\eps$ with $\tau$ as done in \eqref{eq:final_est_SIN_1}. We will instead quantify the error of this replacement, in terms of  the size $\alpha$ of the strip.
First of all, we denote by $R_{\eta,\eta'} := R \cap \supp(\psi_{\eta'}^\perp \nabla \psi_\eta)$, and note that by Remark~\ref{rem:extension_Lip}
\begin{align}
\label{eq:KR_error_vertical_0}
     \| \psi_{\eta'}^\perp \nabla \psi_\eta
        &\cdot \iota_\eps (J - \calR_\eps j )\|_{\tKR(\overline R)}
        \lesssim
    \| \psi_{\eta'}^\perp \nabla \psi_\eta \cdot \iota_\eps (J - \calR_\eps j )\|_{\tKR(\overline R_{\eta,\eta'})}
\\
        &\leq
    \| \nabla \psi_\eta \cdot \iota_\eps (J - \calR_\eps j )\|_{\tKR(\overline R_{\eta,\eta'})}
        +
    \|
        (1-\psi_{\eta'}^\perp )
      \nabla \psi_\eta \cdot \iota_\eps (J - \calR_\eps j )\|_{\tKR(\overline R_{\eta,\eta'})}
        \, ,
\end{align}
where at last we used the triangle inequality.
Thanks to Claim 1c and the bound $\| (1-\psi_{\eta'}^\perp) \nabla \psi_\eta \|_\infty \lesssim \frac1\eta$,
we control the second term by
\begin{align}
    \|
        (1-\psi_{\eta'}^\perp )
            \nabla \psi_\eta \cdot \iota_\eps (J - \calR_\eps j )
    \|_{\tKR(\overline R_{\eta,\eta'})}
        &\lesssim
    \frac1\eta
    \big|
        \iota_\eps (J - \calR_\eps j )
    \big|(\overline R_{\eta,\eta'})
\\
\label{eq:KR_error_vertical}
        &\lesssim
     \frac{\sqrt{\eta'}}\eta  \Big( ( 1+ | j | ) {  \Lm^d(R)}
		   	+ | \iota_\eps J |(\overline R)\Big)
        \, . \quad
\end{align}
Taking \eqref{eq:final_est_SIN_1}, \eqref{eq:KR_error_vertical_0}, and  \eqref{eq:KR_error_vertical} into account,
the claimed estimate \eqref{eq:K2_estimate_KR_2} will follows once we show
\begin{align}
\label{eq:prefinal_claim_SIN}
    \| \nabla \psi_\eta &\cdot \iota_\eps (J - \calR_\eps j )\|_{\tKR(\overline R)}
\\
        &\lesssim
    \frac1{\eta^2}
    \Big(
         \|
            \iota_\eps J - \tau )
        \|_{\tKR(\overline R)}
            +
         \|
            j \Lm^d - \calR_\eps j
         \|_{\tKR(\overline R)}
    \Big)
        +
    \frac\alpha\eta |j|
       {  \Lm^d(R)}
            \, .
\end{align}
Note that it is enough to show that
\begin{align}
\label{eq:final_claim_SIN}
     \|  \nabla \psi_\eta \cdot  (\tau - j \Lm^d)\|_{\tKR(\overline R)}
        \lesssim
    \frac\alpha\eta |j| {  \Lm^d(R)}
        \, ,
\end{align}
for every $\tau$ which satisfies \eqref{eq:tangent_measure_j0}, the claimed \eqref{eq:prefinal_claim_SIN} then following by an application of the triangle inequality (twice) and from the fact that $\Lip(\nabla \psi_\eta) \lesssim \frac1{\eta^2}$.

\smallskip
\noindent \underline{Proof of \eqref{eq:final_claim_SIN}} \
We apply Proposition~\ref{prop:structure_tangent_measures} and write $\tau=j_0 \lambda \otimes \kappa$, where $\lambda \in \M_+\big((\ker j_0)^\perp \big)$ is the Lebesgue measure on $\ker j_0$ and $\kappa \in \M_+(\ker j_0)$.

We fix $\varphi \in \cC(\overline R; V^*)$ with $\text{Lip}(\varphi) \leq 1$. For $x \in \R^d$, we write the orthogonal decomposition $x = x^\parallel + x^\perp$, with $x^\parallel \in \ker j_0$ and $x^\perp \in (\ker j_0)^\perp$. Let us consider  the function $\bar \varphi \in \cC(\overline Q_1^{d-k}); V^*)$ obtained from $\varphi$ averaging out the variables of the $\ker j_0$, namely
\begin{align}
	\bar \varphi(x^\perp) := \int_{Q_\alpha^k} \varphi(x^\perp + x^\parallel) \de \Lm^k(x^\parallel) \, .
\end{align}
In particular, using the fact that $\varphi$ is $1$-Lipschitz, we have that
\begin{align}
	\sup_{x = x^\parallel + x^\perp \in \overline R}
		| \varphi(x) - \bar \varphi(x^\perp)|_{V^*}
			\leq C \alpha \, ,
\end{align}
for some $C=C(k) < \infty$. Therefore, we estimate
\begin{align}
\label{eq:averaging_1}
\Big|
	\int_{\overline R} \varphi
		\nabla \psi_\eta \cdot  \de  &\big( \tau  -  j_0 \Lm^d \big)
\Big|
\\
    &\leq
\Big|
	\int_{\overline R}
	\bar \varphi
		\nabla \psi_\eta \cdot  \de  \big( \tau  -  j_0 \Lm^d \big)
\Big|
		  +
	C \alpha
	\int_{\overline R}
	\big\|
		\nabla \psi_\eta
	\big\|_{\R^d}
	\de
	\big| \tau  -  j_0 \Lm^d \big|	\, .
\end{align}
We claim that in fact the first term on the right-hand side vanishes for every $\varphi$, due to the fact that  the tangent measure is indeed $j_0$ times the Lebesgue measure in the direction of $(\ker j_0)^\perp$.
Indeed, note that by construction, the set $R$ is in product form with respect to $\ker j_0 \oplus (\ker j_0)^\perp$ as $R = Q_1^{d-k} \times  Q_\alpha^k$.
Using that $\nabla \psi_\eta(x) = \nabla \psi_\eta(x^\perp)$ and the product structure of $\tau$ we then obtain
\begin{align}
    \int_{\overline R}
    \bar \varphi
        \nabla \psi_\eta \cdot  \de  \big( \tau  -  j_0 \Lm^d \big)
        =
    \int_{\overline Q_1^{d-k}}
    \bar \varphi(x^\perp)
        \nabla \psi_\eta (x^\perp) \cdot
        j_0
        \big(
            \kappa(Q_\alpha^k)
                -
            \Lm^k(Q_\alpha^k)
        \big)
        \de  \lambda(x^\perp)
            =
        0
            \, ,
\end{align}
where the last equality follows from the fact that
\begin{align}
    |j_0| \kappa(Q_\alpha^k)
        =
    |\tau|(\overline R)
        =
    |j_0| \Lm^d(\overline R)
        =
    |j_0| \Lm^k(Q_\alpha^k)
            \, .
\end{align}
By taking the supremum over $\varphi$, from this and \eqref{eq:averaging_1} we conclude that
\begin{align}
\label{eq:f}
     \|  \nabla \psi_\eta \cdot  (\tau - j \Lm^d)\|_{\tKR(\overline R)}
        \lesssim
    \alpha
	\int_{\overline R}
	\big\|
		\nabla \psi_\eta
	\big\|_{\R^d}
	\de
        \big|
            \tau - j_0 \Lm^d
        \big|
            \, .
\end{align}
Finally, using that $\| \nabla \psi_\eta \| \lesssim \frac1\eta$ and that $|\tau|(\overline R) = |j_0| \Lm^d(\overline R)$, a simple triangle inequality conclude the proof of \eqref{eq:final_claim_SIN}.
\end{proof}

We are now ready to conclude the proof of the lower bound for the singular part, i.e. we prove \eqref{eq:lb_singular}.
Recall the decomposition \eqref{eq:decomp_nu_j} and that we work with a sequence of discrete fields $J_\eps$ so that, as $\eps \to 0$,
\begin{align}	\label{eq:final_hyp_0}
	\iota_\eps J_\eps \to \xi
        \quad \text{vaguely} \quad 
		\tand
	m_\eps 
	=
	\dive(\iota_\eps J_\eps)
	 \to \mu
		\quad \text{in } \tKR
		\, .
\end{align}
With no loss of generality, we will assume that $|m_\eps| \to \sigma \in \M_+(\R^d)$ vaguely as $\eps \to 0$ (due to the fact that $\{ |m_\eps| \}_\eps$ is locally uniformly bounded in total variation).

Recall the definition of the measures $\nu_\eps$, $\nu$ in \eqref{eq:def_nu_eps}, \eqref{eq:def_nu}.
The idea is similar to the one for the absolutely continuous part, where this time we do a blow-up around a singular point.
Writing the Radon--Nikodym decomposition  $\xi = \frac{\dd \xi}{\dd \Lm^d} \Lm^d + \xi^s$, an application of Proposition~\ref{prop:Besicovitch}, Lemma~\ref{lemma:tangent_measures}, and Lemma~\ref{lemma:divergence_measures} ensure the existence of a set $E$ so that $|\xi|^s(E)=0$ and such that, for every $x_0 \in \supp(|\xi|)\setminus E$, we have the following properties:  defining
\begin{align*}
	j_0 := \frac{\de \xi}{\de |\xi|}(x_0)
		\, ,
\end{align*}
\begin{enumerate}
    \item Set $R_\alpha(x_0):= R_\alpha + x_0$ where $R_\alpha$ is the strip in the sense of \eqref{eq:def_Ralpha} with $k=\dim(\ker j_0)$ and $L_k := \ker j_0$. Then we have that
    \begin{align}
\label{eq:density_point_s_singular}
	\lim_{\delta \to 0}
		\frac
			{\xi(x_0 + \delta R_\alpha(x_0) )}
			{|\xi|(x_0 + \delta R_\alpha(x_0))}
		= j_0
	= \lim_{\delta \to 0}
	\frac
		{\xi(x_0 + \delta R_\alpha(x_0))}
		{|\xi|^s(x_0 + \delta R_\alpha(x_0))}
\end{align}
as well as
\begin{align}
\label{eq:density_point_s_singular_infinity}
	\lim_{\delta \to 0}
 \frac
		{|\xi|(x_0 + \delta R_\alpha(x_0))}
		{\Lm^d(x_0 + \delta R_\alpha(x_0))}
    \geq
	\lim_{\delta \to 0}
    C_\alpha
     \frac
		{|\xi|(x_0 + \delta B_{\alpha/2}(x_0))}
		{\Lm^d(x_0 + \delta B_{\alpha/2}(x_0))}
			=  +\infty \, .
\end{align}
    \item There exists a tangent measure $\tilde\tau \in \Tan_{R_\alpha(x_0)}(\xi,x_0)$ which is divergence free, with constant density $j_0$, such that as
\begin{align}	\label{eq:final_blowup_s_sing}
	\xi_{\delta,x_0}=
	\frac{1}{|\xi|(x_0 + \delta R_\alpha(x_0))}
		(\rho_{\delta,x_0})_{\#} \xi
	\to
		\tilde\tau
	\text{ vaguely in }
	\M(\R^d; V \otimes \R^d)
\end{align}
where we also guarantee that
\begin{align}
\label{eq:final_blowup_s_sing_2}
    |\tilde\tau|(R_\alpha)=1
	 	\tand
    |\tilde\tau|(\partial R_\alpha) = 0
        \, ,
\end{align}
for some suitable sequence $\delta=\delta_m(x_0) \to 0$ as $m \to \infty$.
    \item The density of $\nu$ with respect to $|j|$ in $x_0$ can be computed as
    \begin{align}
    \label{eq:sing_point_mu}
    	\frac{\de \nu}{\de |\xi|}(x_0) = \lim_{\delta \to 0}
    	\frac{\nu(x_0 + \delta R_\alpha(x_0))}{|\xi|(x_0 + \delta R_\alpha(x_0))}
            \, .
\end{align}
    \item For every bounded set $B \subset \R^d$, we have that
    \begin{align}
    \label{eq:tv_bound_mu_rescale}
    	\sup_{\delta > 0 }
    		\frac{\sigma(\delta B + x_0)}{|\xi|(x_0 + \delta R_\alpha(x_0))}
    		< \infty
    		\, .
\end{align}
\end{enumerate}
We often omit the explicit choice of the subsequence $\delta_m \to 0$, we will simply write $\delta \to 0$.
Arguing as for the absolutely continuous part, from \eqref{eq:sing_point_mu} we can write
\begin{align}	\label{eq:density_lower_bound_SIN}
	\frac{\de \nu}{\de |\xi|}(x_0)
		&= \lim_{\delta \to 0} \lim_{\eps \to 0}
			\frac{\nu_\eps(\delta R_\alpha(x_0))}{|\xi|(\delta R_\alpha(x_0))}
	=
		\lim_{\delta \to 0} \lim_{\eps \to 0}
			\frac1{t_\delta}
			\frac{F_\eps(J_\eps,\delta R_\alpha(x_0))}{\Lm^d(\delta R_\alpha(x_0))}
		\\
		&= \lim_{\delta \to 0} \lim_{\eps \to 0}
			\frac1{t_\delta}
			\frac{F(\eps^{1-d} J_\eps(\eps \cdot), \frac1\eps \delta R_\alpha(x_0))}{\Lm^d\big( \frac1\eps \delta R_\alpha(x_0) \big)} \, ,
\end{align}
where we define
\begin{align}
\label{eq:def_tdelta}
    t_\delta :=  \frac
        {|\xi|(x_0 + \delta R_\alpha(x_0))}
        {\Lm^d(x_0 + \delta R_\alpha(x_0))}
            \to +\infty
\end{align}
when $\delta \to 0$ thanks to \eqref{eq:density_point_s_singular_infinity}.
Observe by the very definition of the rescaled energy
\begin{align}	\label{eq:energy_rescale_SIN_1}
	\frac
	{F(\eps^{1-d} J_\eps(\eps \cdot), \frac1\eps  \delta R_\alpha(x_0))}
	{\Lm^d\big( \frac1\eps \delta R_\alpha(x_0) \big)}
	=
		F_{\eps/\delta}( K_\delta^\eps,  R_\alpha^\delta )
	\, , \where
    R_\alpha^\delta:= R_\alpha\Big( \frac{x_0}\delta \Big)
        \, ,
\end{align}
where for $s:=\eps/\delta \in \mathbb{N}^{-1}$ we set 
\begin{align}	\label{eq:def_K_s_SIN}
	K_\delta^\eps \in V_\rma^{\cE_s}
	\, , \quad
	K_\delta^\eps(x,y) := \delta^{1-d} J_\eps(\delta x, \delta y)
	\, , \quad
	\forall (x,y) \in \cE_s \, .
\end{align}
Altogether, \eqref{eq:density_lower_bound_SIN} and \eqref{eq:energy_rescale_SIN_1} yield
\begin{align} 	\label{eq:eqs_s_sing_tangent}
	\frac{\de \nu}{\de |\xi|}(x_0) =
		\lim_
		{\substack{\delta \to 0\\ 
		\eps = \eps(\delta)}}
			\frac
			{F_{\eps/\delta}(K_\delta^\eps), R_\alpha^\delta)}
			{t_\delta{   \Lm^d(R_\alpha)}} \, ,
  \end{align}
for $\eps(\delta) \to 0$ fast enough.

For $\delta>0$ fixed, using that $\iota_\eps J_\eps \to j$ vaguely, we deduce that $K_\delta^\eps \to \frac{1}{\Lm^d(Q_\delta)}
(\rho_{\delta,x_0})_{\#} \nu$ vaguely as $\eps \to 0$. Together with \eqref{eq:final_blowup_s_sing} and \eqref{eq:def_tdelta}, this shows that for fixed $\alpha$
\begin{align}
\label{eq:hyp_final_1_SIN}
    \xi_{\eps,\delta,\alpha} :=
    \frac1{  \Lm^d(R_\alpha^\delta) }
    (\tau_{\frac{x_0}{\delta}})_{\#}
    \bigg(
        \frac{\iota_{\eps/\delta} K_\delta^\eps}{t_{\delta}}
    \bigg)
        \to  \tilde\tau
            \quad \text{vaguely in }
                \M(\R^d ; V \otimes \R^d)
        \, ,
\end{align}
as $\delta \to 0$, if $\eps=\eps(\delta) \to 0$ fast enough.
In particular, with this suitable choice of the parameter we have that
\begin{align}
\label{eq:def_T_SIN}
    T:=\sup_
    {\substack{\delta,\alpha > 0\\ 
	\eps = \eps(\delta)}}
    |\xi_{\eps,\delta,\alpha}|(R_\alpha)
        =
    \sup_
    {\substack{\delta,\alpha > 0\\ 
	\eps = \eps(\delta)}}
        \frac
            {|\iota_{\eps/\delta} K_\delta^\eps|(R_\alpha^{\delta})}
            {t_{\delta} {  \Lm^d(R_\alpha^\delta) }}
        < \infty
            \, .
\end{align}
We now want to apply Proposition~\ref{prop:asymptotic_strip}  on the rescaled graph $(\cX_s,\cE_s)$, $s = \eps/\delta$, with
\begin{align}
    K=K_\delta^\eps
        \, , \quad
    R=R_\alpha^\delta
        \, , \quad
    j = t_\delta j_0
        \, , \quad
    \tau = t_\delta\Lm^d(R_\alpha)  (\tau_{\frac{-x_0}{\delta}})_{\#} \tilde\tau
        \, .
\end{align}
Note that by construction and \eqref{eq:final_blowup_s_sing}
\begin{align}
    |\tau|(R_\alpha^\delta)
        =
    t_\delta \Lm^d(R_\alpha^\delta) |\tilde \tau|(R_\alpha)
        =
    t_\delta \Lm^d(R_\alpha^\delta)
        =
    |j| \Lm^d(R_\alpha^\delta)
        \, .
\end{align}
Moreover by the very definition of $T$ \eqref{eq:def_T_SIN} we have that
\begin{align}
\label{eq:bound_constraint}
    \frac{(1 + t_\delta) {  \Lm^d(R_\alpha^\delta) } }{|\iota_s K_s|(R_\alpha^\delta)}
        \geq
    \frac1T
        \, .
\end{align}
Now, let $\eta,\alpha >0$ satisfying
\begin{align}
     \frac1\alpha
     \max\{  \eps R_\partial, \eps R_\Lip \}
        <
    \eta
        \leq
    \max\{ \eta,\alpha \}
        <
    \frac12
    \min
        \Big\{
            1
                ,
            \frac1{8T}
        \Big\} =: \eta_T
        \, .
\end{align}
Then Proposition~\ref{prop:asymptotic_strip} with \eqref{eq:bound_constraint} ensure that
\begin{align}
\label{eq:proof_SIN_energy_error}
    f_{\eps/\delta, \calR} 
    \big(
        t_\delta j_0 ,  R_\alpha^\delta
    \big)
\leq
    F_{\eps/\delta}(K_\delta^\eps,R_\alpha^\delta)
            +
    C
        \err_{\eps/\delta,\alpha}^\tau(K_\delta^\eps  ; t_\delta j_0 ; \eta)
		\,  ,
\end{align}
where by \eqref{eq:bound_constraint} and by definition of err$_{\eps/\delta,\alpha}^\tau$, 
we have that
\begin{align}
     \frac
        {\err_{\eps/\delta,\alpha}^\tau(K_\delta^\eps  ; t_{\delta} j_0 ; \eta)}
        {t_{\delta} \Lm^d(R_\alpha) }
        & \leq
    \frac1{\alpha^2 \eta} \| \dive \xi_{\eps,\delta,\alpha} \|_{\tKR(R_\alpha)}
\\
        & \hspace{-15mm} +
    \frac1{(\alpha^2 \eta)^2}
    \bigg(
        \| \xi_{\eps,\delta,\alpha} - \tilde \tau\|_{\tKR(R_\alpha)}
            +
            \frac1{ \Lm^d(R_\alpha) }
        \big\|
            (\tau_{\frac{x_0}{\delta}})_{\#}(\iota_{\eps/\delta} \calR_{\eps/\delta} j_0)
                -
            j_0 \Lm^d
        \big\|_{\tKR(R_\alpha)}
    \bigg)
\\
        & \hspace{-15mm} +
    \eps
    \Big(
        1
            +
        | \dive \xi_{\eps,\delta,\alpha} |(R_\alpha)
            +
        \frac1{\alpha \eta} T
    \Big)
        +
    \Big(
        \sqrt{\eta}
            +
        \sqrt{\alpha}
            +
        \frac
            {\sqrt\alpha}
            {\eta}
    \Big)
    \Big(
        \frac{ 1 + t_{\delta} }{ t_{\delta} }
            +
            T
    \Big)
            \, .
\end{align}

Concerning the divergence of $\xi_{\eps,\delta,\alpha}$, we use \eqref{eq:divergence_rescaled_formula} and \eqref{eq:final_hyp_0} to infer that, for $\delta$ fixed,
\begin{align}
     \Lm^d(R_\alpha) 
        |\dive \xi_{\eps,\delta,\alpha}|
    =
        \frac{1}{\delta^{d-1}}
        {\rho_{\delta,x_0}}_{\#}(|m_\eps|)
    \to
    \frac{1}{\delta^{d-1}} {\rho_{\delta,x_0}}_{\#} (\sigma)
    \text{ vaguely in }
            \M(\R^d; V)
        \, ,
\end{align}
as $\eps \to 0$.  Arguing as in   \eqref{eq:zero-1}, by \eqref{eq:tv_bound_mu_rescale} we conclude that, for fixed $\alpha >0$,
\begin{align}	\label{eq:hyp_final_2_SIN}
    \dive \iota_{\eps/\delta} \xi_{\eps/\delta,\alpha}
        \to 0
    \text{ locally in TV in }
        \M(\R^d; V) \, ,
\end{align}
if $\delta \to 0$ and $\eps=\eps(\delta)\to 0$ fast enough.

Without loss of generality, we can also assume that
\begin{align}
\label{eq:assumption_TV_lb_SIN}
    |\xi_{\eps,\delta,\alpha}| \to \lambda \geq |j_0| \Lm^d
        \quad \text{vaguely as }\delta \to 0 , \, \eps=\eps(\delta), 
        \quad \text{with } \lambda(\partial R_\alpha) = 0
            \, ,
\end{align}
if not we argue as in \eqref{eq:iKjL} (and work with $hR_\alpha$, for some $h \in (0,1]$). In particular, by Remark~\ref{rem:localisation_weakconv},  Remark~\ref{rem:abs_cont_Reps}, the convergence obtained in \eqref{eq:hyp_final_2_SIN},  \eqref{eq:hyp_final_1_SIN}, and \eqref{eq:hyp_final_2_SIN} imply that, for every $\alpha >0$,
\begin{align}
    \| \dive \xi_{\eps,\delta,\alpha} \|_{\tKR(R_\alpha)}
        +
     \| \xi_{\eps,\delta,\alpha} - \tilde \tau\|_{\tKR(R_\alpha)}
            +
    \big\|
        (\tau_{\frac{x_0}{\delta}})_{\#}(\iota_s \calR_s j_0)
            -
        j_0 \Lm^d
    \big\|_{\tKR(R_\alpha)}
        \to 0
    \, ,
\end{align}
if $\delta \to 0$ and $\eps=\eps(\delta) \to 0$ fast enough.

Taken into account this and  \eqref{eq:hyp_final_2_SIN}, we can provide an upper bound of the error and from \eqref{eq:proof_SIN_energy_error} with $t_{\delta} \to +\infty$ infer that
\begin{align}
    \limsup_{\delta \to 0}
    \frac
         {f_{\eps/\delta,\calR}
            \big(
                t_{\delta} j_0 ,  R_\alpha^{\delta}
            \big)
        }
        { t_{\delta} { \Lm^d(R_\alpha) } }
            \leq
    \lim_{\delta \to 0}
    \frac
        { F_{\eps/\delta}(K_\delta^\eps,R_\alpha^{\delta}) }
        { t_{\delta} {   \Lm^d(R_\alpha) } }
    + C
    \Big(
        \sqrt{\eta}
            +
        \sqrt{\alpha}
            +
        \frac
            {\sqrt\alpha}
            {\eta}
    \Big)
    ( 1 + T )
            \, ,
\end{align}
for every $0<\eta, \alpha < \eta_T$, for $\eps=\eps(\delta) \to 0$ fast enough. Note that, for every $\delta>0$ fixed, by definition of $f_{\hom}$, we have
\begin{align}
\label{eq:final_SIN}
    f_{\hom}(t_\delta j_0)
        =
    \lim_{\eps \to 0}
    \frac
        { f_{\eps/\delta ,\calR}(t_\delta, R_\alpha^\delta) }
        {   \Lm^d(R_\alpha) }
            \, .
\end{align}
Consequently, up to choosing $\eps=\eps(\delta) \to 0$ fast enough once more, we have that
\begin{align}
    f_{\hom}^\infty(j_0)
        =
    \lim_{t \to +\infty}
        \frac{f_{\hom}(t j_0)}{t}
        =
   \lim_
		{\substack{\delta \to 0\\ 
		\eps = \eps(\delta)}}
            \frac
                { f_{\eps/\delta , \calR}(t_\delta j_0,R_\alpha^\delta) }
                { t_\delta {   \Lm^d(R_\alpha) } }
            \, ,
\end{align}
for every $\alpha \in (0,1)$. Putting this together with \eqref{eq:final_SIN} and \eqref{eq:eqs_s_sing_tangent}, we finally obtain
\begin{align}
    f_{\hom}^\infty(j_0)
        \leq
    \frac{\de \nu}{\de |\xi|}(x_0)
        +
    + C
    \Big(
        \sqrt{\eta}
            +
        \sqrt{\alpha}
            +
        \frac
            {\sqrt\alpha}
            {\eta}
    \Big)
    ( 1 + T )
            \, ,
\end{align}
for every  $0<\eta,\alpha<\eta_T$. Sending first $\alpha \to 0$ and then $\eta \to 0$, we conclude the proof of the claimed lower bound \eqref{eq:lb_singular}.

\appendix

\section{Convergence of measures}

We start by recalling the different notions of convergence of measures that are going to be used in the paper. See \cite[Appendix A]{Gladbach-Kopfer-Maas-Portinale:2023} for more details and proofs (see also  \cite[Section 8.10(viii)]{Bogachev}).

Let $(X,d)$ be a locally compact and separable metric space (we will almost exclusively consider subsets of $\R^d$), $W$ a finite dimensional normed space, and denote by $\M(X; W)$ the space of $W$-valued Borel measures on $X$.
For $\mu \in \M(X;W)$, denote by $|\mu|\in \M_+(X)$ its variation and
with $\|\mu\|_{\TV(X)} := |\mu|(X)$ its total variation.

\begin{defi}[Narrow and vague convergence]
	Let $\mu, \mu_n \in \M(X;W)$ for $n = 1, 2, \ldots$.
	\begin{enumerate}
		\item
		We say that	$\mu_n \to \mu$ narrowly in $\M(X;W)$ if
		$
			\int_{X} \psi \de \mu_n
		\to
			\int_{X} \psi \de \mu
		$
		for every $\psi \in \cC_{\rm b}(X)$.
		\item
		We say that $\mu_n \to \mu$ vaguely in $\M(X;W)$ if
		$
			\int_{\R^d} \psi \de \mu_n
		\to
			\int_{\R^d} \psi \de \mu
		$
		for every $\psi \in \cC_{\rm c}(X)$.
	\end{enumerate}
\end{defi}
\begin{rem}
\label{rem:localisation_weakconv}
Suppose that $\mu_n \to \mu$ narrowly in $\M(X;W)$ and let $A\subset X$ be such that $\lambda(\partial A) =0$, for every accumulation point $\lambda \in \M_+(X)$ of $|\mu_n|$ with respect to the vague topology. Then we have that $\mu_n|_A \to \mu|_A$ narrowly in $\M(A;W)$ (see e.g. \cite[Prop~1.62]{Ambrosio-Fusco-Pallara:2000}). More generally, if $f:X \to \R$ is a bounded, measurable function whose set of discontinuities is of $\lambda$-measure zero, then $f \mu_n \to f \mu$ narrowly in $\M(X;W)$. The same conclusions holds true if $\mu_n \to \mu$ vaguely  in $\M(X;W)$and supp$(f)$ is bounded.
\end{rem}

$\M(X;W)$ is a Banach space endowed with the norm $\|\mu\|_{\TV(X;W)} = |\mu|(X)$.
By the Riesz–Markov theorem, it is the dual space of the Banach space
$\cC_0(X;W)$, the closure (in the uniform topology) of all continuous functions $\psi : X \to W^*$ with compact support, endowed with the supremum norm
	$\|\psi\|_\infty = \sup_{x \in X} | \psi(x) |_{W^*}$.
For $\psi : X \to W$ let
	$\Lip(\psi) := \sup_{x \neq y}
		\frac{ |\psi(x) - \psi(y)|_{W^*}}{d(x,y)}
	$
be its Lipschitz constant.

\begin{defi}
\label{def:KR}
	The \emph{Kantorovich--Rubinstein norm} on $\M(X;W)$ is defined by
	\begin{align}
		\label{eq:KR}
		\| \mu \|_{\KR(X; W)}
			:=
		\sup\bigg\{ \int_{X}
			\langle
					\psi
					, \de \mu
			\rangle
				\ : \
				\psi \in \cC(X;W^*), \
				\| \psi \|_\infty \leq 1, \
				\Lip(\psi) \leq 1
			\bigg\}.
	\end{align}
\end{defi}

It follows immediately that 
\begin{align*}
	\| \mu \|_{\KR(X; W)} \leq \|\mu\|_{\TV(X;W)}
\end{align*} 
for all $\mu \in \M(X;W)$.
However, the norms $\|\cdot\|_{\KR(X;W)}$ and $\| \cdot \|_{\TV(X;W)}$ are not equivalent (hence $(\M(X;W), \|\cdot \|_{\KR(X;W)})$ is not a complete space).
A closely related norm on $\M(X;W)$ is
	\begin{align*}
		\| \mu \|_{\tKR(X; W)}
			:=
	 |\mu(X)|
	 +
		\sup\bigg\{ \int_{X}
					\langle
										\psi
										, \de \mu
								\rangle
				\ : \
				\psi \in \cC(X;W^*), \
				\psi(x_0) = 0, \
				\Lip(\psi) \leq 1
			\bigg\},
	\end{align*}
	for some fixed $x_0 \in X$;
	see \cite[Section 8.10(viii)]{Bogachev}.
	These two norms are in fact equivalent on compact metric spaces.

\begin{prop}
\label{prop:equivalence_convergence_KR}
	Let $(X,d)$ be a compact metric space.
	For $\mu \in \M(X;W)$ we have
	\begin{align*}
		\| \mu \|_{\KR(X;W)}
		\leq
		\| \mu \|_{\tKR(X;W)}
		\leq
		c_X \| \mu \|_{\KR(X;W)},
	\end{align*}
	where $c_X < \infty$ depends only on $\diam(X)$. Moreover, for
			$\mu_n, \mu \in \M(X;W)$
		we have
		\begin{align*}
			\mu_n \to \mu \, \text{narrowly}
		\qquad \text{if and only if} \qquad
			\| \mu_n - \mu \|_{\KR(X;W)} \to 0
                    \, \, \text{and} \, \,
                \sup_{n \in \N} |\mu_n|(X) < \infty
                    \, .
		\end{align*}
\end{prop}

\begin{proof}
For the equivalence between $\KR$ and $\tKR$, see e.g.  \cite[Appendix A]{Gladbach-Kopfer-Maas-Portinale:2023}. For the link to narrow convergence, we refer to  \cite[Theorem 8.3.2]{Bogachev} in the case of positive measures. Let us prove the left-to-right implication for real-valued measures (i.e. $W = \R$).
We write $\mu_n = (\mu_n)_+ - (\mu_n)_-$. By Riesz--Markov--Kakutani's theorem and the uniform boundedness principle on Banach spaces, we must have
\begin{align*}
	\sup_{n \in \N} \| \mu_n\|_{\tv(X)} < \infty
		\quad \Longrightarrow \quad
	\sup_{n \in \N}
		(\mu_n)_+(X)
		+
	\sup_{n \in \N}
		(\mu_n)_-(X) < \infty
	\, .
\end{align*}
Therefore, the positive measures $(\mu_n)_+$, $(\mu_n)_-$ are uniformly bounded, hence narrowly precompact. Denote by $(\mu_\infty)_+$, $(\mu_\infty)_-$ any  narrow limit (up to subsequence) of $(\mu_n)_+$, $(\mu_n)_-$. Note that it must hold $\mu = (\mu_\infty)_+ - (\mu_\infty)_-$ (although in general $(\mu_\infty)_\pm \neq \mu_\pm$). Then we apply the result for positive measures and deduce that (up to subsequence) $\| (\mu_n)_\pm - (\mu_\infty)_\pm\|_{\KR(X)} \to 0$. It is then easy to show that this implies $\| \mu_n - \mu \|_{\KR(X)} \to 0$.

The general case of $V$-valued measures follows by applying the scalar result to each component of $\mu_n$ and $\mu$.

The right-to-left implication directly follows by Riesz--Markov--Kakutani's theorem, Banach--Alaoglu theorem's, and the density of the Lipschitz functions in $\cC(X;W^*)$.
\end{proof}

Let $A$ be a closed subset of $\R^d$,
we write $A_\delta(z) := z + \delta A$.
For later use, we record the scaling relation
\begin{align}
\label{eq:KR-scaling}
	\Bigl( 1 \wedge \frac{1}{\delta}\Bigr) 
	\| \mu \|_{\tKR(A_\delta(z);W)}
	\leq
	\| (\rho_{\delta, z}) _\# \mu \|_{\tKR(A;W)}
	\leq
	\Bigl( 1 \vee \frac{1}{\delta}\Bigr) 
	\| \mu \|_{\tKR(A_\delta(z);W)} \, ,
\end{align} 
which holds for 
$\mu \in \M(\R^d;W)$,
$\delta > 0$, and $z \in \R^d$.

A related notion is the vectorial $1$-Wasserstein transport distance;
see \cite{Ciosmak:2021}:
for a given $\mu \in \M(X;W)$ with zero mass $\mu(X) = 0 \in W$, one considers the vectorial optimal transport problem given by
\begin{align}	\label{eq:def_T1}
	T_1(\mu) := \inf_{	\pi \in  \mathcal M(X \times X ; W)}
		\left\{
			\int | x-y | \de |\pi|(x,y)
				\suchthat
			(\tP_2)_{\#} \pi - (\tP_1)_{\#} \pi = \mu
		\right\}
	.
\end{align}
In particular, if $V = \R$, we can write any such $\mu$ as $\mu = \mu_+ - \mu_-$ and $T_1(\mu)= W_1(\mu_-,\mu_+)$ is the classical $1$-Wasserstein distance.
Moreover, \cite[Theorem~2]{Ciosmak:2021} asserts that
\begin{align}
	\label{eq:remark_KR_T1}
	\| \mu \|_{\tKR(X;W)} = T_1(\mu)
\end{align}
for all $\mu \in \M(X;W)$ with
	$\mu(X) = 0$,
which is the vectorial generalisation of the Kantorovich duality for $W_1$.

\begin{rem}[Localisation of the $\tKR$ norm]
\label{rem:extension_Lip}
	By the Kirszbraun theorem \cite{Kirszbraun:1934} (McShane's theorem when working on metric spaces),
	every Lipschitz function $\psi : A \subset X \to V$ defined on a subset $A \subset X$
	can be extended to a Lipschitz function $\hat \psi:X \to V$
		with Lipschitz constant bounded by $C \Lip (\psi,A)$, where $C$ depends only on the dimension of $V$.
	Therefore, for every $\mu \in \M(X; V)$
	we have $\| \mu \|_{\tKR(\R^d)} \leq C \| \mu \|_{\tKR(A)}$ for every $A \supseteq \supp(\mu)$.
\end{rem}

\section{Discrete calculus}
\label{sec:preliminaries_derivatives}

It will be convenenient to use notation from discrete calculus
on a countable graph $(\cX, \cE)$.
In particular, for 
$\psi : \cX \to V $ and 
$J \in V^\cE_\rma$, 
we set
\begin{align}	
	\label{eq:def_grad_1}
	& \grad \psi \in V_a^{\cE}
	\, , \quad
	& (\grad \psi)(x,y) 
	&:= \psi(y) - \psi(x)
	 \quad \text{for } (x,y) \in \cE \, ,
	\\
	& \DIVE J \in V^\cX 
	\, , \quad
	& \DIVE J(x) 
	 &:= \sum_{y\sim x}J(x,y)
	\quad \text{for } x \in \cX \, .
\end{align}
For $\psi: \cX \to \R $ we consider the arithmetic mean function
\begin{align}	\label{eq:def_hat_0}
	\hat\psi: \cE \to \R 
		\, , \quad
	\hat\psi(x,y) := \frac12 \big( \varphi(x) + \varphi(y) \big)
		\quad
	\forall (x,y) \in \cE \, .
\end{align}
We use the same notation when $\psi\in \cC(\R^d;V)$.
Moreover,
for $K \in \cE \to \R $ and $J \in V_\rma^{\cE}$
we define 
\begin{align}
	& K \star J : \cX \to V 
		\, ,  &&
	( K \star J )(x)
	:=
	\frac12\sum_{y \sim x} K(x,y) J(x,y)
	 && \forall x  \in \cX\, .
\end{align}
With this notation, the following discrete Leibniz rule holds
for $\psi\in \cX \to \R$
and
$J \in V_\rma^{\cE}$:
\begin{align}
	\label{eq:Leibniz}
	\DIVE( \hat\psi\cdot J)
	=
	\psi\DIVE J
	+
	(\grad \psi) \star J \, .
\end{align}

The next lemma shows a useful intertwining property for the discrete and continuous divergence operators
and the embedding map $\iota_\eps$.

\begin{lemma}[Discrete and continuous divergence]
\label{lemma:divergence_discr_cont}
	Let $\eps > 0$.
	For $J \in V_\rma^{\cE_\eps}$ we have
	\begin{align}
		 \DIVE J = \dive \iota_\eps J
			 \quad
		\text{in} \  \mathcal D'(\R^d; V) \, ,
	\end{align}
    where we identified $\DIVE J: \cX \to V$ with the corresponding atomic measure in $\M(\cX_\eps)$.
\end{lemma}

\begin{proof}
Using
the anti-symmetry of $J$,
and the fundamental theorem of calculus,
we obtain
for all test functions $\Psi \in C_c^\infty(\R^d; V^*)$,
\begin{align*}
&
	\int_{\R^d}
	\big\langle
		\de (\DIVE J) , \Psi
 	\big\rangle
 =
	 \sum_{x \in \cX_\eps}
		 \ip{
			 \DIVE J(x),
			 \Psi(x)
			}
=
	\sum_{x \in \cX_\eps}
	\Bip{
		\sum_{y: y \sim x} J(x,y)
		,
		\Psi(x)
	}
\\
& =
	\frac12 \sum_{(x,y) \in \cE_\eps}
	\big\langle
		J(x,y)
		,
		\Psi(x) - \Psi(y)
	\big\rangle
=
	- \frac12 \sum_{(x,y) \in \cE_\eps}
		\int_{[x,y]}
		\bip{
			\nabla \Psi(z)
			,
			J(x,y) \otimes \tau_{xy}
		}
		\de \Hm^1(z)
\\
 &=
	- \int_{\R^d}
	\bip{
		\de \iota_\eps J
		,
		\nabla \Psi
	}
=
	\int_{\R^d}
	\bip{
		\de (\dive \iota_\eps J)
		,
		\Psi
	} \, .
\end{align*}
Since $\Psi \in C_c^\infty(\R^d; V^*)$ is arbitrary,
the result follows.
\end{proof}

The following result provides a useful quantitative comparison of discrete and continuous gradients.

\begin{lemma}
\label{lemma:gradients}
Fix $\eps > 0$ and $J \in V_\rma^{\cE_\eps}$.
For $\psi \in C^1(\R^d)$ and $\Psi \in \Lip(\R^d; V^*)$ we have
	\begin{align}
	\label{eq:gradients_comparison}
	\quad \bigg|
		\big\langle
			\iota_\eps J
		,
		\Psi	\otimes
		\nabla \psi
		\big\rangle
	-
		\big\langle
				\grad \psi \star J
			,
			\Psi
		\big\rangle
	\bigg|
	\lesssim
		\eps \Lip(\Psi)
		\big|
		\nabla \psi \cdot
			\iota_\eps J
		\big|(B_\eps) \, , \quad
	\end{align}
	where $B_\eps := B_\eps(\supp(\Psi))$.
\end{lemma}

\begin{proof}
Let us write
	$\tau_{xy} := (y-x) / \| y-x\|_{\R^d}$
for $x,y \in \R^d$.
Using the definition of $\iota_\eps J$ from \eqref{eq:def-iota-J}
the fundamental theorem of calculus,
and the definition of $\iota_\eps$ 
we obtain
\begin{align}
	\int_{\R^d}
		\big\langle
			\de \iota_\eps J
		,
		\Psi	\otimes
		\nabla \psi
		\big\rangle
	&= \frac12\sum_{(x,y) \in \cE_\eps}
		\int_{[x,y]}
			\big\langle
				 J(x,y) \otimes \tau_{xy}
				,
				\Psi(z)	\otimes
				\nabla \psi(z)
			\big\rangle
		\de \Hm^1(z)
\\
	&= \frac12\sum_{(x,y) \in \cE_\eps}
	\int_{[x,y]}
		\big\langle
			 J(x,y) \otimes \tau_{xy}
			,
			\Psi(x)	\otimes
				\nabla \psi(z)
		\big\rangle
	\de \Hm^1(z)
	+ R_\eps
\\
	&= \frac12\sum_{(x,y) \in \cE_\eps}
		\big(\psi(y) - \psi(x)\big)
		\big\langle
			J(x,y)
			,
			\Psi(x)
		\big\rangle
	+ R_\eps
\\
	&= \sum_{x \in \cX_\eps}
	\big\langle
		(\grad\psi
		\star
		J )(x)
			,
			\Psi(x)
		\big\rangle
	+ R_\eps
\\
	& =
	\int_{\R^d}
	\big\langle
		\de 
		\big(
			\grad \psi \star J
		\big)
		,
		\Psi
	\big\rangle
		+R_\eps \, ,
\end{align}
where the remainder term is given by
\begin{align}
	R_\eps:=
	\frac12\sum_{(x,y) \in \cE_\eps}
	\int_{[x,y]}
		\big\langle
			 J(x,y) \otimes \tau_{xy}
			,
			\big(\Psi(z) - \Psi(x)\big)	\otimes
				\nabla \psi(z)
		\big\rangle
	\de \Hm^1(z)\, .
\end{align}
Since $\| x-y\|_{\R^d} \leq R_3\eps$ for $(x,y) \in \cE_\eps$, we have
\begin{align}
	|R_\eps|
		&\leq
		\frac12
		\sum_{\substack{(x,y) \in \cE_\eps \\ [x,y] \cap \supp(\Psi) \neq \emptyset}}
		\|   J(x,y)\|_V
		\text{Lip}(\Psi)
		\| x - y\|_{\R^d}
		\int_{[x,y]}
		|\nabla\psi(z)|\,
			  \de \Hm^1(z)
\\
		&\leq R_3 \eps \text{Lip}(\Psi) \big|
				\nabla \psi \cdot
					\iota_\eps J
				\big|(B_\eps) \, ,
\end{align}
which is the sought upper bound.
\end{proof}

\subsection*{Acknowledgement}
This research was funded in part by the Austrian Science Fund (FWF) project \href{https://doi.org/10.55776/F65}{10.55776/F65}. L.P. gratefully acknowledges funding from the Deutsche Forschungsgemeinschaft (DFG, German Research Foundation) under Germany’s Excellence Strategy- GZ 2047/1, Projekt-ID 390685813.
Financial support by the Deutsche Forschungsgemeinschaft (DFG) within the CRC 1060, at University of Bonn project number 211504053, is also gratefully acknowledged.

\subsection*{Data Availability Statement}
Data sharing not applicable to this article as no datasets were generated or analysed during the current study.

\bibliographystyle{my_alpha}
\bibliography{stoch.bib}

\newcommand{\etalchar}[1]{$^{#1}$}
\begin{thebibliography}{AA00}\itemsep0.1em

\bibitem[ACG11]{Alicandro-Cicalese-Gloria:2011}
{\scshape R.~Alicandro, M.~Cicalese, {\upshape and} A.~Gloria}.
\newblock Integral representation results for energies defined on stochastic lattices and application to nonlinear elasticity.
\newblock {\em Arch. Ration. Mech. Anal.}, 200(3), 881--943, 2011.

\bibitem[ACR15]{Alicandro-Cicalese-Ruf:2015}
{\scshape R.~Alicandro, M.~Cicalese, {\upshape and} M.~Ruf}.
\newblock Domain formation in magnetic polymer composites: an approach via stochastic homogenization.
\newblock {\em Arch. Ration. Mech. Anal.}, 218(2), 945--984, 2015.

\bibitem[ADH17]{Auffinger-Damron-Hanson:2017}
{\scshape A.~Auffinger, M.~Damron, {\upshape and} J.~Hanson}.
\newblock {\em 50 years of first-passage percolation}, volume~68 of {\em University Lecture Series}.
\newblock American Mathematical Society, Providence, RI, 2017.

\bibitem[AFP00]{Ambrosio-Fusco-Pallara:2000}
{\scshape L.~Ambrosio, N.~Fusco, {\upshape and} D.~Pallara}.
\newblock {\em Functions of bounded variation and free discontinuity problems}.
\newblock Oxford Mathematical Monographs. The Clarendon Press, Oxford University Press, New York, 2000.

\bibitem[AkK81]{akcoglu1981ergodic}
{\scshape M.~A.~Akcoglu {\upshape and} U.~Krengel}.
\newblock Ergodic theorems for superadditive processes.
\newblock 1981.

\bibitem[Alb93]{Alberti:1993}
{\scshape G.~Alberti}.
\newblock Rank one property for derivatives of functions with bounded variation.
\newblock {\em Proc. Roy. Soc. Edinburgh Sect. A}, 123(2), 239--274, 1993.

\bibitem[AmM92]{Ambrosio-DalMaso:1992}
{\scshape L.~Ambrosio {\upshape and} G.~D.~Maso}.
\newblock On the relaxation in {BV}($\omega$; $\mathbb{R}^m$) of quasi-convex integrals.
\newblock {\em Journal of Functional Analysis}, 109(1), 76--97, 1992.

\bibitem[BC{\etalchar{*}}13]{Baia-Chermisi-Matias-Santos:2013}
{\scshape M.~Ba\'{\i}a, M.~Chermisi, J.~Matias, {\upshape and} P.~M.~Santos}.
\newblock Lower semicontinuity and relaxation of signed functionals with linear growth in the context of {$\mathcal A$}-quasiconvexity.
\newblock {\em Calc. Var. Partial Differential Equations}, 47(3-4), 465--498, 2013.

\bibitem[Bec52]{MR68196}
{\scshape M.~Beckmann}.
\newblock A continuous model of transportation.
\newblock {\em Econometrica}, 20, 643--660, 1952.

\bibitem[BF{\etalchar{*}}02]{Bouchitte-Fonseca-Leoni-Mascarenhas:2002}
{\scshape G.~Bouchitt\'{e}, I.~Fonseca, G.~Leoni, {\upshape and} L.~Mascarenhas}.
\newblock A global method for relaxation in {$W^{1,p}$} and in {${\rm SBV}_p$}.
\newblock {\em Arch. Ration. Mech. Anal.}, 165(3), 187--242, 2002.

\bibitem[BFL00]{Braides-Fonseca-Leoni:2000}
{\scshape A.~Braides, I.~Fonseca, {\upshape and} G.~Leoni}.
\newblock A-quasiconvexity: relaxation and homogenization.
\newblock {\em ESAIM: Control, Optimisation and Calculus of Variations}, 5, 539--577, 2000.

\bibitem[BMS08]{Braides-Maslennikov-Sigalotti:2008}
{\scshape A.~Braides, M.~Maslennikov, {\upshape and} L.~Sigalotti}.
\newblock Homogenization by blow-up.
\newblock {\em Appl. Anal.}, 87(12), 1341--1356, 2008.

\bibitem[Bog07]{Bogachev}
{\scshape V.~I.~Bogachev}.
\newblock {\em Measure theory. {V}ol. {I}, {II}}.
\newblock Springer-Verlag, Berlin, 2007.

\bibitem[BrC23]{Braides-Caroccia:2023}
{\scshape A.~Braides {\upshape and} M.~Caroccia}.
\newblock Asymptotic behavior of the dirichlet energy on poisson point clouds.
\newblock {\em Journal of Nonlinear Science}, 33(5), 80, 2023.

\bibitem[Cio21]{Ciosmak:2021}
{\scshape K.~J.~Ciosmak}.
\newblock Optimal transport of vector measures.
\newblock {\em Calculus of Variations and Partial Differential Equations}, 60(6), 1--22, 2021.

\bibitem[CMO20]{Conti-Muller-Ortiz:2020}
{\scshape S.~Conti, S.~M\"{u}ller, {\upshape and} M.~Ortiz}.
\newblock Symmetric div-quasiconvexity and the relaxation of static problems.
\newblock {\em Arch. Ration. Mech. Anal.}, 235(2), 841--880, 2020.

\bibitem[DaF16]{Davoli-Fonseca:2016}
{\scshape E.~Davoli {\upshape and} I.~Fonseca}.
\newblock Homogenization of integral energies under periodically oscillating differential constraints.
\newblock {\em Calculus of Variations and Partial Differential Equations}, 55, 1--60, 2016.

\bibitem[DaM86]{DalMaso-Modica:1986}
{\scshape G.~Dal~Maso {\upshape and} L.~Modica}.
\newblock Nonlinear stochastic homogenization.
\newblock {\em Annali di matematica pura ed applicata}, 144, 347--389, 1986.

\bibitem[De~06]{DeLellis:2006}
{\scshape C.~De~Lellis}.
\newblock Lecture notes on rectifiable sets, densities, and tangent measures.
\newblock {\em Preprint}, 23, 2006.

\bibitem[DeL77]{DeGiorgi-Letta:1977}
{\scshape E.~De~Giorgi {\upshape and} G.~Letta}.
\newblock Une notion g\'en\'erale de convergence faible pour des fonctions croissantes d'ensemble.
\newblock {\em Annali della Scuola Normale Superiore di Pisa - Classe di Scienze}, 4e s{\'e}rie, 4(1), 61--99, 1977.

\bibitem[DeR16]{DePhilippis-Rindler:2016}
{\scshape G.~De~Philippis {\upshape and} F.~Rindler}.
\newblock On the structure of {$\mathcal A$}-free measures and applications.
\newblock {\em Ann. of Math. (2)}, 184(3), 1017--1039, 2016.

\bibitem[DiL15]{Disser-Liero:2015}
{\scshape K.~Disser {\upshape and} M.~Liero}.
\newblock On gradient structures for {M}arkov chains and the passage to {W}asserstein gradient flows.
\newblock {\em Netw. Heterog. Media}, 10(2), 233--253, 2015.

\bibitem[EHS23]{Esposito-Heinze-Schlichting:2023}
{\scshape A.~Esposito, G.~Heinze, {\upshape and} A.~Schlichting}.
\newblock Graph-to-local limit for the nonlocal interaction equation.
\newblock {\em arXiv:2306.03475}, 2023.

\bibitem[EP{\etalchar{*}}21]{Esposito-Patacchini-Schlichting:2021}
{\scshape A.~Esposito, F.~S.~Patacchini, A.~Schlichting, {\upshape and} D.~Slepcev}.
\newblock Nonlocal-interaction equation on graphs: gradient flow structure and continuum limit.
\newblock {\em Arch. Ration. Mech. Anal.}, 240(2), 699--760, 2021.

\bibitem[EPS24]{Esposito-Patacchini-Schlichting:2023}
{\scshape A.~Esposito, F.~S.~Patacchini, {\upshape and} A.~Schlichting}.
\newblock On a class of nonlocal continuity equations on graphs.
\newblock {\em European J. Appl. Math.}, 35(1), 109--126, 2024.

\bibitem[EsM23]{Esposito-Mikolas:2023}
{\scshape A.~Esposito {\upshape and} L.~Mikol{\'a}s}.
\newblock On evolution {PDE}s on co-evolving graphs.
\newblock {\em arXiv:2310.10350}, 2023.

\bibitem[FMP22]{Forkert2020}
{\scshape D.~Forkert, J.~Maas, {\upshape and} L.~Portinale}.
\newblock Evolutionary {$\Gamma$}-convergence of entropic gradient flow structures for {F}okker-{P}lanck equations in multiple dimensions.
\newblock {\em SIAM J. Math. Anal.}, 54(4), 4297--4333, 2022.

\bibitem[FoF62]{FordFulkerson}
{\scshape L.~R.~Ford, Jr. {\upshape and} D.~R.~Fulkerson}.
\newblock {\em Flows in networks}.
\newblock Princeton University Press, Princeton, N.J., 1962.

\bibitem[FoM92]{Fonseca-Mueller:1992}
{\scshape I.~Fonseca {\upshape and} S.~M{\"u}ller}.
\newblock Quasi-convex integrands and lower semicontinuity in {$L^1$}.
\newblock {\em SIAM J. Math. Anal.}, 23(5), 1081--1098, 1992.

\bibitem[FoM99]{Fonseca-Muller:1999}
{\scshape I.~Fonseca {\upshape and} S.~M\"{u}ller}.
\newblock {$\mathcal A$}-quasiconvexity, lower semicontinuity, and {Y}oung measures.
\newblock {\em SIAM J. Math. Anal.}, 30(6), 1355--1390, 1999.

\bibitem[Gar20]{Garcia-Trillos:2020}
{\scshape N.~Garc\'{\i}a~Trillos}.
\newblock Gromov-{H}ausdorff limit of {W}asserstein spaces on point clouds.
\newblock {\em Calc. Var. Partial Differential Equations}, 59(2), Paper No. 73, 43, 2020.

\bibitem[GiM13]{GiMa13}
{\scshape N.~Gigli {\upshape and} J.~Maas}.
\newblock Gromov-{H}ausdorff convergence of discrete transportation metrics.
\newblock {\em SIAM J. Math. Anal.}, 45(2), 879--899, 2013.

\bibitem[GK{\etalchar{*}}20]{Gladbach-Kopfer-Maas:2020}
{\scshape P.~Gladbach, E.~Kopfer, J.~Maas, {\upshape and} L.~Portinale}.
\newblock Homogenisation of one-dimensional discrete optimal transport.
\newblock {\em Journal de Math{\'e}matiques Pures et Appliqu{\'e}es}, 139, 204 -- 234, 2020.

\bibitem[GK{\etalchar{*}}23]{Gladbach-Kopfer-Maas-Portinale:2023}
{\scshape P.~Gladbach, E.~Kopfer, J.~Maas, {\upshape and} L.~Portinale}.
\newblock Homogenisation of dynamical optimal transport on periodic graphs.
\newblock {\em Calc. Var. Partial Differential Equations}, 62(5), Paper No. 143, 75, 2023.

\bibitem[GlK24]{Gladbach-Kopfer:2024}
{\scshape P.~Gladbach {\upshape and} E.~Kopfer}.
\newblock Stochastic homogenization of dynamical discrete optimal transport.
\newblock {\em arXiv:2411.04157}, 2024.

\bibitem[HrT23]{Hraivoronska-Tse:2023}
{\scshape A.~Hraivoronska {\upshape and} O.~Tse}.
\newblock Diffusive limit of random walks on tessellations via generalized gradient flows.
\newblock {\em SIAM J. Math. Anal.}, 55(4), 2948--2995, 2023.

\bibitem[HST24]{Hraivoronska-Tse-Schlichting:2023}
{\scshape A.~Hraivoronska, A.~Schlichting, {\upshape and} O.~Tse}.
\newblock Variational convergence of the {S}charfetter--{G}ummel scheme to the aggregation-diffusion equation and vanishing diffusion limit.
\newblock {\em Numer. Math.}, 156(6), 2221--2292, 2024.

\bibitem[IsL24]{ishida2024quantitative}
{\scshape S.~Ishida {\upshape and} H.~Lavenant}.
\newblock Quantitative convergence of a discretization of dynamic optimal transport using the dual formulation.
\newblock {\em Foundations of Computational Mathematics}, pages 1--36, 2024.

\bibitem[Kin73]{kingman1973subadditive}
{\scshape J.~F.~C.~Kingman}.
\newblock Subadditive ergodic theory.
\newblock {\em Ann. Probability}, 1, 883--909, 1973.

\bibitem[Kir34]{Kirszbraun:1934}
{\scshape M.~Kirszbraun}.
\newblock {\"U}ber die zusammenziehende und {L}ipschitzsche {T}ransformationen.
\newblock {\em Fundamenta Mathematicae}, 22(1), 77--108, 1934.

\bibitem[Lav21]{lavenant2021unconditional}
{\scshape H.~Lavenant}.
\newblock Unconditional convergence for discretizations of dynamical optimal transport.
\newblock {\em Mathematics of Computation}, 90(328), 739--786, 2021.

\bibitem[LiM02]{Licht-Michaille:2002}
{\scshape C.~Licht {\upshape and} G.~Michaille}.
\newblock Global-local subadditive ergodic theorems and application to homogenization in elasticity.
\newblock {\em Ann. Math. Blaise Pascal}, 9(1), 21--62, 2002.

\bibitem[MMS15]{Matias-Morandotto-Santos:2015}
{\scshape J.~Matias, M.~Morandotti, {\upshape and} P.~M.~Santos}.
\newblock Homogenization of functionals with linear growth in the context of {$\mathcal A$}-quasiconvexity.
\newblock {\em Applied Mathematics \& Optimization}, 72(3), 523--547, 2015.

\bibitem[Mor52]{Morrey:1952}
{\scshape C.~B.~Morrey, Jr.}
\newblock {Quasi-convexity and the lower semicontinuity of multiple integrals.}
\newblock {\em Pacific Journal of Mathematics}, 2(1), 25 -- 53, 1952.

\bibitem[NSS17]{Neukamm-Schaffner-Schlomerkemper:2017}
{\scshape S.~Neukamm, M.~Sch\"{a}ffner, {\upshape and} A.~Schl\"{o}merkemper}.
\newblock Stochastic homogenization of nonconvex discrete energies with degenerate growth.
\newblock {\em SIAM J. Math. Anal.}, 49(3), 1761--1809, 2017.

\bibitem[PoQ24]{Portinale-Quattrocchi:2024}
{\scshape L.~Portinale {\upshape and} F.~Quattrocchi}.
\newblock Discrete-to-continuum limits of optimal transport with linear growth on periodic graphs.
\newblock {\em European Journal of Applied Mathematics}, pages 1--29, 2024.

\bibitem[RuZ23]{Ruf-Zeppieri:2023}
{\scshape M.~Ruf {\upshape and} C.~I.~Zeppieri}.
\newblock Stochastic homogenization of degenerate integral functionals with linear growth.
\newblock {\em Calc. Var. Partial Differential Equations}, 62(4), Paper No. 138, 2023.

\bibitem[San15]{santambrogio}
{\scshape F.~Santambrogio}.
\newblock {\em Optimal transport for applied mathematicians}, volume~87 of {\em Progress in Nonlinear Differential Equations and their Applications}.
\newblock Birkh\"{a}user/Springer, Cham, 2015.

\bibitem[SlW23]{slepvcev2023nonlocal}
{\scshape D.~Slep{\v{c}}ev {\upshape and} A.~Warren}.
\newblock Nonlocal {W}asserstein distance: Metric and asymptotic properties.
\newblock {\em Calculus of Variations and Partial Differential Equations}, 62(9), 238, 2023.

\end{thebibliography}

\end{document}